%% file: Mueller_Woike_main.tex
\documentclass[10pt]{article}

\usepackage{tikz}
\usetikzlibrary{matrix,arrows,decorations.pathmorphing,shapes.geometric}

\usepackage{etex}
\usepackage{subcaption}
\usepackage[utf8]{inputenc}
\usepackage{a4wide}
\usepackage{graphicx}
\usepackage{wrapfig}
\usepackage[hmarginratio={1:1},     % equal left and right margins
vmarginratio={1:1},     % equal top and bottom margins
textwidth=17cm,        % new text width
textheight=22.5cm,
heightrounded,]{geometry}
\usepackage{amsmath,accents}
\usepackage{amsthm}
\usepackage{amssymb}

\usepackage{mathrsfs, mathtools}
\usepackage{overpic}
\usepackage[utf8]{inputenc}
\usepackage[hidelinks]{hyperref}
\numberwithin{equation}{section}
\usepackage{mathtools}
\usepackage{pdfpages}
\mathtoolsset{showonlyrefs}
\numberwithin{equation}{section}
\usepackage{tocloft}
\usepackage{multicol}
\usepackage{yhmath}

\usepackage{tikzit}
\input{rfa-diagrams.tikzstyles}
\usetikzlibrary{decorations.pathreplacing,decorations.markings}
\tikzset{
  on each segment/.style={
    decorate,
    decoration={
      show path construction,
      moveto code={},
      lineto code={
        \path [#1]
        (\tikzinputsegmentfirst) -- (\tikzinputsegmentlast);
      },
      curveto code={
        \path [#1] (\tikzinputsegmentfirst)
        .. controls
        (\tikzinputsegmentsupporta) and (\tikzinputsegmentsupportb)
        ..
        (\tikzinputsegmentlast);
      },
      closepath code={
        \path [#1]
        (\tikzinputsegmentfirst) -- (\tikzinputsegmentlast);
      },
    },
  },
  mid arrow/.style={postaction={decorate,decoration={
        markings,
        mark=at position .5 with {\arrow[#1]{stealth}}
      }}},
}
\tikzset{%
  link/.style    = { white, double = black, line width = 0.9pt,
                     double distance = 1.5pt },
  channel/.style = { white, double = gray, line width = 1.8pt,
                     double distance = 0.6pt },
}

\usepackage[all,cmtip]{xy}

\usepackage{tikz}
\usetikzlibrary{matrix,arrows,decorations.pathmorphing,shapes.geometric}
\usepackage{tikz-cd}

\makeatletter
\DeclareRobustCommand{\em}{%
	\@nomath\em \if b\expandafter\@car\f@series\@nil
	\normalfont \else \slshape \fi}
\makeatother

\newtheoremstyle{style1}% name of the style to be used
{13pt}% measure of space to leave above the theorem. E.g.: 3pt
{13pt}% measure of space to leave below the theorem. E.g.: 3pt
{}% name of font to use in the body of the theorem
{}% measure of space to indent
{\normalfont\bfseries}% name of head font
{.}% punctuation between head and body
{.5em}% space after theorem head; " " = normal interword space
{}

\theoremstyle{style1}

\newtheorem{definition}{Definition}[section]
\newtheorem{example}[definition]{Example}
\newtheorem{remark}[definition]{Remark}

\newtheoremstyle{style2}% name of the style to be used
{13pt}% measure of space to leave above the theorem. E.g.: 3pt
{13pt}% measure of space to leave below the theorem. E.g.: 3pt
{\slshape}% name of font to use in the body of the theorem
{}% measure of space to indent
{\normalfont\bfseries}% name of head font
{.}% punctuation between head and body
{.5em}% space after theorem head; " " = normal interword space
{}

\theoremstyle{style2}

\makeatletter
\newtheorem*{rep@theorem}{\rep@title}
\newcommand{\newreptheorem}[2]{%
	\newenvironment{rep#1}[1]{%
		\def\rep@title{#2 \ref{##1}}%
		\begin{rep@theorem}}%
		{\end{rep@theorem}}}
\makeatother

\newreptheorem{theorem}{Theorem}
\newreptheorem{corollary}{Corollary}
\newtheorem{lemma}[definition]{Lemma}
\newtheorem{theorem}[definition]{Theorem}
\newtheorem{proposition}[definition]{Proposition}
\newtheorem{corollary}[definition]{Corollary}

\usepackage{enumitem}

\newcommand{\R}{\mathbb{R}}
\newcommand{\C}{\mathbb{C}}
\newcommand{\Z}{\mathbb{Z}}

\newcommand{\N}{\mathbb{N}}

\newcommand{\Grpd}{\catf{Grpd}}
\newcommand{\Map}{\catf{Map}}
\newcommand{\ra}[1]{\xrightarrow{\ #1 \ }}

\newcommand{\SO}{\operatorname{SO}}

\newcommand{\catf}[1]{{\mathsf{#1}}}
\let\P\undefined
\newcommand{\P}{\mathcal{P}}
\newcommand{\framed}{\catf{f}E_2}
\newcommand{\szero}{\catf{S}_0}
\newcommand{\surf}{\catf{Surf}}
\newcommand{\Surf}{\mathbf{Surf}}
\newcommand{\surfc}{\catf{Surf}^{\catf{c}}}
\newcommand{\Surfc}{\mathbf{Surf}^{\catf{c}}}

\newcommand{\hocolimsub}[1]{\underset{#1}{\operatorname{hocolim}}\,}

\newcommand{\hbdy}{\catf{Hbdy}}
\newcommand{\Hbdy}{\mathbf{Hbdy}}
\newcommand{\hbdyg}{\catf{Hbdy}^\text{\normalfont a}}
\newcommand{\Hbdyg}{\mathbf{Hbdy}^\text{\normalfont a}}

\newcommand{\As}{\catf{As}}

\newcommand{\Legs}{\catf{Legs}}

\newcommand{\ModAlg}{\catf{ModAlg}\,}

\newcommand{\cat}[1]{\mathcal{#1}}
\newcommand{\Env}{\catf{U}\,}

\newcommand{\Envint}{\catf{U}_{\!\int}\,}

\newcommand{\End}{\operatorname{End}}

\newcommand{\Hom}{\operatorname{Hom}}

\newcommand{\id}{\operatorname{id}}

\newcommand{\ev}{\operatorname{ev}}

\newcommand{\RBr}{\mathsf{RBr}}
\newcommand{\Cat}{\catf{Cat}}
\newcommand{\FinVect}{\catf{Vect}^\catf{f}}

\newcommand{\DS}{\text{/\hspace{-0.1cm}/}}

\newcommand{\Set}{\catf{Set}}
\newcommand{\Calc}{\catf{Calc}}
\let\to\undefined
\newcommand{\to}{\longrightarrow}
\let\mapsto\undefined
\newcommand{\mapsto}{\longmapsto}

\newcommand{\Rex}{\catf{Rex}}

\newcommand{\Fin}{\catf{Lex}^\catf{f}}
\newcommand{\Vect}{\catf{Vect}^\catf{f}}

\newcommand{\RForests}{\catf{RForests}}
\newcommand{\Forests}{\catf{Forests}}
\newcommand{\Graphs}{\catf{Graphs}}
\newcommand{\RGraphs}{\catf{RGraphs}}

\newcommand{\opp}{\text{opp}}
\let\colon\undefined\newcommand{\colon}{:}

\newcommand{\disk}{\mathbb{D}}
\newcommand{\comp}{\mathsf{C}}

\DeclareMathSymbol{\Phiit}{\mathalpha}{letters}{"08} 
\DeclareMathSymbol{\Psiit}{\mathalpha}{letters}{"09}
\DeclareMathSymbol{\Sigmait}{\mathalpha}{letters}{"06}
\DeclareMathSymbol{\Xiit}{\mathalpha}{letters}{"04}
\DeclareMathSymbol{\Piit}{\mathalpha}{letters}{"05}\let\Pi\undefined\newcommand{\Pi}{\Piit}
\DeclareMathSymbol{\Gammait}{\mathalpha}{letters}{"00}
\DeclareMathSymbol{\Omegait}{\mathalpha}{letters}{"0A}\let\Omega\undefined\newcommand{\Omega}{\Omegait}
\DeclareMathSymbol{\Upsilonit}{\mathalpha}{letters}{"07}
\DeclareMathSymbol{\Thetait}{\mathalpha}{letters}{"02}
\DeclareMathSymbol{\Lambdait}{\mathalpha}{letters}{"03}\let\Lambda\undefined\newcommand{\Lambda}{\Lambdait}
\let\Phi\undefined\newcommand{\Phi}{\Phiit}
\let\Sigma\undefined\newcommand{\Sigma}{\Sigmait}
\let\Psi\undefined\newcommand{\Psi}{\Psiit}
\let\Gamma\undefined\newcommand{\Gamma}{\Gammait}

\begin{document}

	\vspace*{-1.5cm}
	\begin{flushright}
		\small
		{\sf CPH-GEOTOP-DNRF151} \\
		\textsf{October 2020}
	\end{flushright}
	
	\vspace{5mm}
	
	\begin{center}
		\textbf{\LARGE{Cyclic framed little disks algebras, \\[0.2ex] Grothendieck-Verdier duality \\[0.5ex]  and handlebody group representations}}\\
		\vspace{1cm}
		{\large Lukas Müller $^{a}$} \ \ and \ \ {\large Lukas Woike $^{b}$}
		
		\vspace{5mm}

	{\slshape $^a$ Max-Planck-Institut f\"ur Mathematik\\
Vivatsgasse 7 \\  D-\,53111 Bonn \\ lmueller4@mpim-bonn.mpg.de }
		\\[7pt]
		{\slshape $^b$ Institut for Matematiske Fag \\ 
			K\o benhavns Universitet\\ Universitetsparken 5\\ DK-2100 Copenhagen \O\\  ljw@math.ku.dk }
	\end{center}
	\vspace{0.3cm}
	\begin{abstract}\noindent 
We characterize cyclic algebras over the associative and the framed little 2-disks operad in any symmetric monoidal bicategory. The cyclicity is appropriately treated in a coherent way, i.e.\ up to coherent isomorphism. When the symmetric monoidal bicategory is specified to be a certain symmetric monoidal bicategory of linear categories subject to finiteness conditions, we prove that cyclic associative and cyclic framed little 2-disks algebras, respectively, are equivalent to pivotal Grothendieck-Verdier categories and ribbon Grothendieck-Verdier categories, a type of category that was introduced by Boyarchenko-Drinfeld based on Barr's notion of a  $\star$-autonomous category. We use these results and Costello's modular envelope construction to obtain two applications to quantum topology: I) We extract a consistent system of handlebody group representations from any ribbon Grothendieck-Verdier category inside a certain symmetric monoidal bicategory of linear categories and show that this generalizes the handlebody part of Lyubashenko's mapping class group representations. II) We establish a Grothendieck-Verdier duality for the category extracted from a modular functor by evaluation on the circle (without any assumption on semisimplicity), thereby generalizing results of Tillmann and Bakalov-Kirillov.  
	\end{abstract}

	\tableofcontents

\section{Introduction and summary}
Algebras over the associative and the framed little 2-disks operad 
with values in the symmetric monoidal bicategory of (linear) categories
are well-known to be equivalent to
(linear) monoidal and balanced 
braided monoidal categories, respectively.
It is equally well-known that both operads naturally 
have the structure of a cyclic operad,  as introduced by Getzler and Kapranov \cite{gk}. This means that they come with a specific way to cyclically permute inputs of operations with the output.
Given a cyclic operad, one may, of course, forget the cyclic structure
and consider \emph{ordinary} algebras over it, but one can also consider 
\emph{cyclic} algebras which are defined to be compatible with the cyclic
structure in the appropriate sense. 
This raises the immediate question how \emph{cyclic}
algebras (with values in a symmetric monoidal bicategory such as suitable bicategories of linear categories)
over the associative and 
the framed little 2-disks operad (hereafter just referred to as framed little disks operad) can be characterized. It is implicitly understood here that we will consider these
algebraic structures 	
\emph{up to coherent isomorphism} in the appropriate sense.
In this article, we give an explicit characterization of these cyclic algebras
 in terms of
\emph{Grothendieck-Verdier duality}; 
the precise statements appear as Theorems~\ref{thmcyclas} 
and~\ref{thmcycle2} below and will also be momentarily 
discussed in course of this introduction.  
Grothendieck-Verdier duality is a notion
 proposed and investigated by Boyarchenko and Drinfeld in \cite{bd}
as a weakening of rigidity. It is based on earlier notions due to Barr \cite{barr}. 
Afterwards, we present applications of our results
in quantum topology: We 
 give a new class of explicitly computable handlebody groups representations
that satisfy excision.
These representations generalize the 
handlebody part of Lyubashenko's mapping class group representations \cite{lyubacmp,lyu,lyulex}.
Moreover, we
prove a duality result for the category obtained from a modular functor
by evaluation on the circle.

Let us give the precise statements and elaborate on the structure of the article:
The sections 
\ref{seccycmodalg0} to \ref{secfE2GV}
are of purely operadic nature
and devoted to the characterization of cyclic associative and 
framed little disks algebras in an arbitrary symmetric monoidal bicategory $\cat{M}$. 
The fact that we consider cyclic algebras in a symmetric monoidal \emph{bi}category
has two reasons. Firstly, both the associative and the framed little disks operad are aspherical, i.e.\ 
they may be seen as category-valued (in fact, groupoid-valued) operads. This makes it possible
and natural to consider 
algebras over these operads in a
symmetric monoidal bicategory because any bicategory 
is naturally enriched over categories.
Secondly, this choice perfectly matches with our motivation
coming from quantum topology as our applications will show.

We use Costello's description of cyclic (and modular) operads \cite{costello}, i.e.\ 
we describe a
 category-valued cyclic operad as a 
symmetric monoidal functor
$\mathcal{O}:\Forests\to\Cat$ from the \emph{forest category} to the category of categories (symmetric monoidal functors will here be automatically understood in the weak sense).
The objects of $\Forests$ are the graphs with one vertex and 
$n$ legs for $n\ge 1$ (the so-called \emph{corollas})
and finite disjoint unions thereof. Morphisms are given by  forests; 
we recall the necessary details in Section~\ref{secrecallcos}.
Disjoint union provides a symmetric monoidal structure.
In order to define cyclic $\mathcal{O}$-algebras 
in a symmetric monoidal bicategory $\cat{M}$ in Section~\ref{seccycmodalg},
we define the cyclic endomorphism operad 
$\End_\kappa^X:\Forests\to\Cat$
for any object $X$ of $\cat{M}$ equipped with a symmetric non-degenerate pairing 
$\kappa :X\otimes X\to I$, i.e.\ a morphism from $X\otimes X$ to the monoidal unit $I$ of $\cat{M}$
that is symmetric up to coherent isomorphism and exhibits $X$ as its own dual in the homotopy category of $\cat{M}$.
Explicitly, $\End_\kappa^X$ will send a corolla $T$ with set $\Legs(T)$ of legs to the morphism category $\cat{M}(X^{\otimes \Legs(T)},I)$, where $X^{\otimes \Legs(T)}$ is the unordered monoidal product of $X$ over $\Legs(T)$ that we define in Section~\ref{seccyclicmodalg}.
The structure of a cyclic $\mathcal{O}$-algebra  on $(X,\kappa)$ is a symmetric monoidal transformation
$A:\mathcal{O}\to\End_\kappa^X$.

Cyclic associative algebras
can be characterized in more concrete terms
if $\cat{M}$ is the symmetric monoidal bicategory $\Fin$
of finite linear categories
over an algebraically closed field $k$ that we fix throughout the article (this symmetric monoidal bicategory is frequently used in quantum algebra).
Its objects are 
 finite $k$-linear categories (the full definition is given on page~\pageref{exsymmoncat}), its 1-morphisms are  left exact functors,
 and its 2-morphisms are natural transformations.
 For this target category, we  characterize cyclic associative algebras by means of 
Grothendieck-Verdier duality  \cite{bd}:
A \emph{Grothendieck-Verdier category}
(Definition~\ref{defGV}) is 
\begin{itemize}
	\item a monoidal category $\cat{C}$ together with  an
object $K\in\cat{C}$ 
(the \emph{dualizing object})
such that the functor $\cat{C}(K,X \otimes -)$ is representable for every $X\in\cat{C}$

	\item subject to the condition that the functor $D:\cat{C}\to\cat{C}^\opp$ sending $X$ to a representing object $DX$ for $\cat{C}(K,X\otimes-)$ is an equivalence.
	
	\end{itemize}
One should understand $DX$ as the \emph{dual} of $X$. The functor $D$ is referred to as the \emph{duality functor}.

If $\cat{C}$ is rigid, one obtains such a structure with the monoidal unit as the dualizing object,
but the notion of a Grothendieck-Verdier category is strictly weaker (see e.g.\ \cite[Example~2.3]{bd}). 
A \emph{pivotal structure} on a Grothendieck-Verdier category (Definition~\ref{defpivbd})
 consists of natural isomorphisms
$\psi_{X,Y} : \cat{C}(K, X\otimes Y)\to\cat{C}(K,Y\otimes X)$ subject to two coherence conditions.  
We should remark that our conventions are dual to the ones in \cite{bd} for convenience. 
Pivotal Grothendieck-Verdier categories can also be defined inside $\Fin$ (meaning that all structure consists of (higher) morphisms in $\Fin$ instead of $\Cat$), which allows us to formulate our first main result:
\begin{reptheorem}{thmcyclas}
	The structure of a cyclic associative algebra in $\Fin$ 
	amounts precisely  to a pivotal Grothendieck-Verdier category in $\Fin$.
\end{reptheorem}

Theorem~\ref{thmcyclas}
can be further combined with a result of Street who proves in \cite[Proposition~3.2]{street}
that any Grothendieck-Verdier category aka $\star$-autonomous category can be equivalently described
as a Frobenius pseudomonoid, see also Remark~\ref{remgvfa}. 
Then Theorem~\ref{thmcyclas} tells us that a cyclic associative algebra in $\Fin$ will in particular inherit the structure of a Frobenius pseudomonoid. 

The strategy for the proof of 
Theorem~\ref{thmcyclas} is as follows: When considering cyclic algebras e.g.\ in vector spaces, it is standard that a cyclic $\mathcal{O}$-algebra is an ordinary $\mathcal{O}$-algebra plus an invariant pairing \cite{gk,markl}. The same remains true in a higher categorical context (and in particular the bicategorical context considered in the paper) with the subtle difference that the invariance of the pairing becomes \emph{structure}
instead of just a \emph{property} and leads to two types of coherence conditions that we identify in the \emph{Lifting Theorem}~\ref{thmlifting}
(that gives us the structure to \emph{lift} a non-cyclic algebra to a cyclic one).
 We formulate the Lifting Theorem
in Corollary~\ref{remgenrel} more concretely for an operad given in terms of generators and relations.
On this basis, Theorem~\ref{thmcyclas} can  be obtained through relatively tedious algebraic manipulations
and several facts on finite linear categories extracted from \cite{fss}. 
It is important to note that we can explicitly characterize cyclic associative algebras in 
\emph{any} symmetric monoidal bicategory, but prove the relation to Grothendieck-Verdier duality only for $\Fin$ or similar categories
(Remark~\ref{remalttolex}).
For example, the description  of cyclic associativity  through Grothendieck-Verdier duality is not possible in the bicategory $\Cat$ of categories. 

In order to treat the cyclic framed little disks operad $\framed$ in a similar way by means of the Lifting Theorem,
we use the presentation of this operad through ribbon braids \cite{salvatorewahl}. On the operad of ribbon braids, we
establish a cyclic structure (Proposition~\ref{Prop: Combinatorial model for cyclic structure}) which under the equivalence to the framed little disks operad corresponds to the cyclic structure coming from the identification of $\framed$ with the cyclic operad of genus zero surfaces (or also the cyclic structure from \cite{budney}), see 
Proposition~\ref{Thm: Relation ribbon braids and E2}.
We may then prove the second main result:

\begin{reptheorem}{thmalgrbr}
	The structure of a cyclic framed little disks algebra in $\Fin$ 
	is equivalent  to a  ribbon Grothendieck-Verdier category in $\Fin$.
\end{reptheorem}

The precise definition of a ribbon structure on a (pivotal) Grothendieck-Verdier category $\cat{C}$ is given  in Section~\ref{balancedGV}:
Roughly, it amounts to a braiding, i.e.\ natural isomorphisms $c_{X,Y}:X\otimes Y\to Y\otimes X$ subject to the usual hexagon relations, and a balancing, i.e.\ a natural automorphism $\theta_X : X\to X$ with $\theta_I=\id_I$, $\theta_{X\otimes Y}=c_{Y,X}c_{X,Y}(\theta_X\otimes \theta_Y)$; 
moreover, one requires $D\theta_X=\theta_{DX}$. 
Together, the braiding and the balancing give rise to a pivotal structure. Therefore, a ribbon Grothendieck-Verdier category may be equivalently  described as a pivotal ribbon Grothendieck-Verdier category with a compatibility condition on the pivotal structure, braiding and balancing (Definition~\ref{defcomppivbrbal}) as we explain in Lemma~\ref{lemmapivribbon}.
 
 In order to profit from the results in concrete applications,
 we make use of the following construction:
 A modular operad \cite{gkmod} is, roughly speaking, a cyclic operad additionally admitting self-compositions of operations.
 To any cyclic operad, one may assign the `smallest' modular operad containing it, namely the  modular envelope \cite{costello}. For a given cyclic operad $\mathcal{O}$, it is
  obtained by freely completing it to a modular operad (in a homotopically correct way).
Moreover, any cyclic algebra over this cyclic operad extends to a modular algebra over its modular envelope by a purely abstract argument.
The reason why this becomes particularly interesting for the associative and framed little disks operad is that their modular envelopes
have been computed by Costello  \cite{costello} and Giansiracusa \cite{giansiracusa} in terms of interesting and well-studied objects in low-dimensional topology.
In light of these results, Theorem~\ref{thmcyclas} and \ref{thmcycle2} have important consequences in quantum topology that are treated in Section~\ref{secapp} and summarized now.

\subparagraph{\nameref{app-hbdy}.}
By a result of Giansiracusa
\cite[Theorem~A]{giansiracusa}
there is a canonical map 
from the modular envelope of $\framed$ to the modular operad 
$\hbdyg$ 
of handlebodies
(the subscript `a' indicates  the restriction to certain allowed handlebodies used in \cite{giansiracusa}: the three-dimensional ball and the three-dimensional ball with one embedded disk are excluded).
On the level of topological modular operads, this map is an isomorphism between connected components and a homotopy equivalence 
on all connected components except for the one of the solid closed torus.
By means of Theorem~\ref{thmcycle2}
 we may now prove that 
a  ribbon Grothendieck-Verdier category leads to a consistent system of handlebody group representations:
\begin{reptheorem}{thmmainbalancedbraidedcalc2}[combined with Theorem~\ref{thmmainbalancedbraidedcalc1}]
	Let $\cat{C}$ be a  ribbon Grothendieck-Verdier category in $\Fin$
	with dualizing object $K$.
	For integers $g,n \ge 0$ and any family $X_1,\dots,X_n \in \cat{C}$ of objects in $\cat{C}$,
	the finite-dimensional morphism space
	\begin{align} V_{g,n}(X_1,\dots,X_n) := \cat{C}(K,X_1\otimes\dots\otimes X_n\otimes \mathbb{F}^{\otimes g})\label{eqnvgn} \end{align}
	defined using the canonical coend $\mathbb{F}=\int^{X\in\cat{C}}X\otimes DX$ ($D$ is the duality functor of $\cat{C}$)
	comes naturally with an action of the handlebody group, i.e.\ the mapping class group of the handlebody of genus $g$ and $n$ embedded disks in the boundary, whenever $(g,n)\neq (1,0)$. 
	The vector spaces \eqref{eqnvgn} behave locally under the sewing of handlebodies. More explicitly, there are canonical isomorphisms
	\begin{align}
	\oint^{Y\in\cat{C}} V_{g,n+2}(-,Y,DY) &\cong V_{g+1,n}(-)  \ , \\
	\oint^{Y\in\cat{C}}	V_{g_1,n+1}(-,Y)\otimes V_{g_2,m+1}(-,DY)&\cong V_{g_1+g_2,n+m}(-)
	\end{align}
	of left exact functors  $\cat{C}^{\boxtimes n}\to\FinVect$ and $\cat{C}^{\boxtimes (n+m)}\to\FinVect$, respectively, where $\oint$ is the left exact coend. These isomorphisms are compatible with the handlebody group actions.
\end{reptheorem}

For the precise definition of $\mathbb{F}$, we refer to the statement of Theorem~\ref{thmmainbalancedbraidedcalc1} in the main text.
We refer to Remark~\ref{remtorus} for a comment on the case $(g,n)=(1,0)$.

Strictly speaking,
the objects $X_1,\dots,X_n\in\cat{C}$ should not be thought of as ordered, but rather attached to the embedded disks in the boundary of the handlebody.
In order to obtain the vector space 
\eqref{eqnvgn},
an order is \emph{chosen}. These subtleties are suppressed in the above Theorem for presentation purposes, but explained in detail in Remark~\ref{remcaveatorder}.
A graphical presentation summarizing the Theorems~\ref{thmcycle2} and~\ref{thmmainbalancedbraidedcalc2} is given in Figure~\ref{Fig:sketch results}.

Theorem~\ref{thmmainbalancedbraidedcalc2} applies in particular to finite ribbon categories in the sense of \cite{egno}.
For a finite ribbon category that arises as the category of finite-dimensional modules over a ribbon Hopf algebra, 
Theorem~\ref{thmmainbalancedbraidedcalc2} specializes to the following statement:

\begin{repcorollary}{corhopf}
	Let $A$ be a finite-dimensional ribbon Hopf algebra and denote by $A_\text{coadj}^*$ the dual of $A$ with coadjoint action.
	Then for any non-negative integers $g$ and $n$ with $(g,n)\neq (1,0)$ 
	and any finite-dimensional $A$-modules $X_1,\dots,X_n$,
	the vector space 	
	\begin{align} \Hom_A\left( k ,  X_1\otimes \dots \otimes X_n \otimes \left(   A_\text{coadj}^*   \right)^{\otimes g}  \right)
	\end{align}
	of $A$-invariants of the module 
	$X_1\otimes \dots \otimes X_n \otimes \left(   A_\text{coadj}^*   \right)^{\otimes g}$
	comes canonically with an action of the 
	mapping class group of the handlebody with genus $g$ and $n$ disks embedded in the boundary of the handlebody.
\end{repcorollary}
Note that $A$ is not assumed to be factorizable.
In Example~\ref{exampleGVnonrigid}, we consider the handlebody group representations for a ribbon Grothendieck-Verdier category whose Grothendieck-Verdier duality does not come from rigidity.

The basic ingredient for the proof of Theorem~\ref{thmmainbalancedbraidedcalc2} is clearly Theorem~\ref{thmcycle2} because it abstractly leads to the desired handlebody representation using Giansiracusa's Theorem. To derive the explicit description from above, however, one needs further tools, namely a locality property for the handlebody group representations. This is afforded by an excision result (Theorem~\ref{thmexcision}) that states, roughly, that  modular algebras
with values in $\Fin$ behave `locally' under Lyubashenko's left exact coend $\oint$ \cite{lyulex}.
As a by-product, this opens a new perspective on Lyubashenko's left exact coend as the composition in the cyclic (actually, even modular) endomorphism operad.

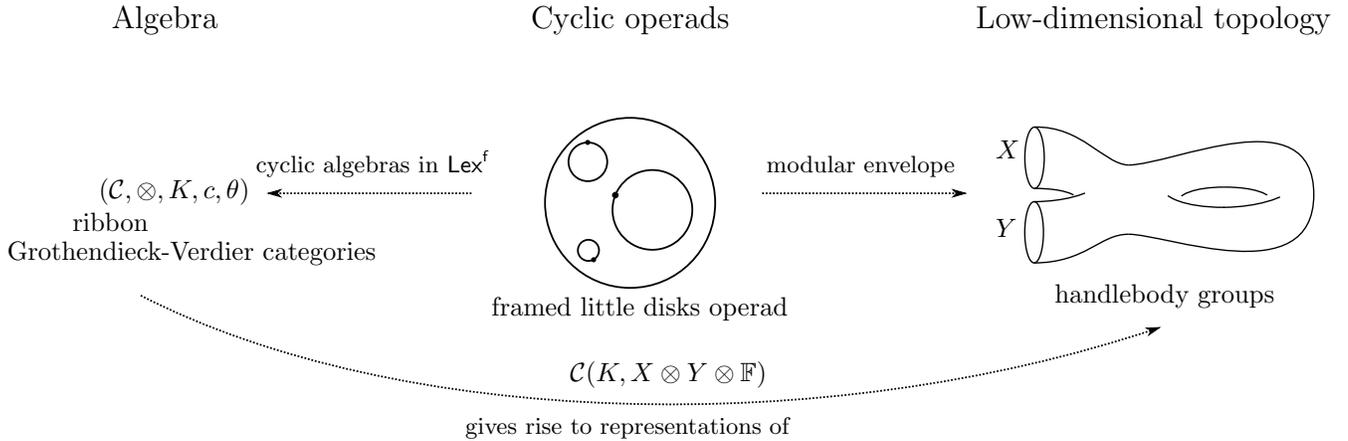
\begin{figure}[h]
\begin{center}
\begin{overpic}[scale=0.9]
{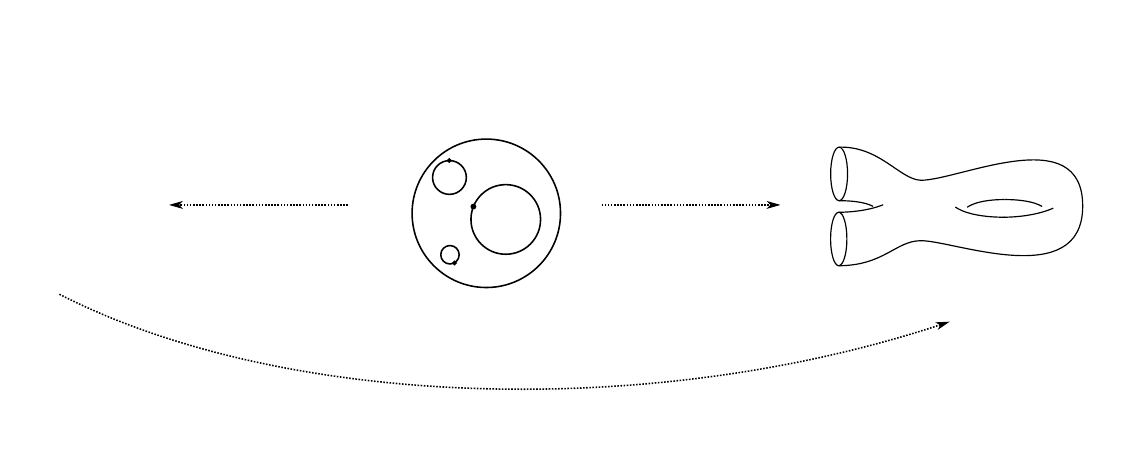}
\put(2,21){$(\cat{C},\otimes, K, c, \theta)$} 
\put(38,7){$\cat{C}(K, X\otimes Y \otimes \mathbb{F})$} 
\put(70.5,24){$X$}
\put(70.5,18){$Y$}
\put(32,12){$\text{framed little disks operad}$}
\put(0,18.5){$\text{ribbon} $}
\put(-5,16.3){$\text{Grothendieck-Verdier categories}$}
\put(75,13){$\text{handlebody groups}$}
\put(35,34){\large$\text{Cyclic operads}$}
\put(3,34){\large$\text{Algebra}$}
\put(69,34){\large$\text{Low-dimensional topology}$} 
\put(14,23){\small $\text{cyclic algebras in $\Fin$}$}
\put(53,23){\small $\text{modular envelope}$}
\put(30,3){\small $\text{gives rise to representations of}$}
\end{overpic}
\end{center}
\caption{A visualization of the different mathematical structures involved in our main result and their relations.}
\label{Fig:sketch results}
\end{figure}

In the special case that $\cat{C}$ is actually a modular category (meaning that the Grothendieck-Verdier duality actually comes from rigidity and that the braiding is non-degenerate; we review the terminology in Section~\ref{app-hbdy}), Theorem~\ref{thmmainbalancedbraidedcalc2} relates to classical constructions as follows: Under the far stronger assumption of modularity, $\cat{C}$ yields a modular functor, i.e.\ a consistent system of projective mapping class group representations by the Lyubashenko construction
\cite{lyubacmp,lyu,lyulex}. The restriction to handlebody group representations will then coincide with the handlebody group representations from Theorem~\ref{thmmainbalancedbraidedcalc2}.
Note however
that Theorem~\ref{thmmainbalancedbraidedcalc2} can be applied to way more general situations: Neither rigidity nor non-degeneracy of the braiding are needed. As a price to pay, one just finds handlebody group representations, but not mapping class group representations.

The relation to Lyubashenko's mapping class group representations is quite appealing on a conceptual level:
In the original papers \cite{lyubacmp,lyu,lyulex}, the mapping class group actions are established through tedious computations relying on a presentation of mapping class groups in terms of generators and relations; in \cite{jfcs} a description of these mapping class group actions through the combinatorial Lego-Teichmüller game of Bakalov and Kirillov 
\cite{bakifm}
is given. 
Our approach using the modular envelope and our characterization of cyclic $\framed$-algebras 
puts at least the handlebody part of Lyubashenko's construction
on purely topological grounds. 
In particular, it allows us to obtain from a topological construction the conformal blocks of a modular category (meaning the spaces \eqref{eqnvgn}) --- at least as vector spaces with handlebody group representations ---
without making any algebraic ad-hoc ansatz.

\subparagraph{\nameref{app-mf}.}
For many of the representation categories appearing in conformal and topological field theory,
\emph{duality results} have been established:

\begin{itemize}
	
	\item In \cite{huang} Huang gives a class of
	vertex operator algebras whose category of modules is rigid.
	
	\item In \cite{BDSPV15} Bartlett, Douglas, Schommer-Pries and Vicary prove that the value of an extended three-dimensional topological field theory on the circle is rigid.
	
	\item In \cite{baki} Bakalov and Kirillov prove under the assumption of semisimplicity, simplicity of the monoidal unit and a normalization axiom for the sphere that the category $\cat{C}$ obtained by evaluation of a modular functor on the circle is \emph{weakly rigid}, i.e.\ they conclude that there is an anti-equivalence $* : \cat{C}\to\cat{C}^\opp$ with natural isomorphisms
	\begin{align} \label{eqnrcatisos} \cat{C}(X,Y\otimes Z)\cong \cat{C}(Y^*\otimes X,Z)\cong \cat{C}(X\otimes Z^*,Y)\quad \text{for all}\quad X,Y,Z\in\cat{C} \ . \end{align}
	A similar result is given by Turaev in \cite[V.]{turaev} for so-called \emph{rational} modular functors.
	In \cite{turaev,baki} semisimplicity is directly imposed. Tillmann \cite[Section~3]{tillmann} presents related duality results by working in a different category of linear categories in which semisimplicity for the category on the circle is obtained as a consequence. We should emphasize here again that, in contrast to \cite{turaev,tillmann,baki}, we will not build in semisimplicity (neither directly or indirectly) into our definitions; in fact, going beyond the semisimple case is one of our main motivations.
	
	\end{itemize}
We refer to \cite[Section~1.3]{BDSPV15} for an excellent discussion of these (and more) duality results including more references to the literature. 

Our characterization of cyclic framed little disks algebras allows us to improve on the third result. Without the assumption of semisimplicity, simplicity of the monoidal unit and a normalization axiom for the sphere, we prove:
\begin{reptheorem}{thmcirclesectormf}
	The linear category extracted from a $\Fin$-valued modular functor 
	inherits a ribbon Grothendieck-Verdier structure. In particular, it is a pivotal Grothendieck-Verdier category.
\end{reptheorem}

For the notion of a modular functor, various slightly different definitions exists. We fix our notion in Definition~\ref{lemmainclhbdyc} following essentially \cite{jfcs,dmf}.

It might seem a little confusing that we find a Grothendieck-Verdier structure \emph{without} the requirement that the dualizing object actually coincides with the monoidal unit as one would assume after seeing \eqref{eqnrcatisos} (in language of \cite{bd}, such a Grothendieck-Verdier category would be called an \emph{r-category}; this is still strictly weaker than rigidity). 
However, we explain in Corollary~\ref{corrcategory} that once one imposes the (in fact very strong) requirements of semisimplicity, simplicity of the monoidal unit and a normalization axiom, the dualizing object will be forced to coincide with the monoidal unit. Therefore, Theorem~\ref{thmcirclesectormf} actually recovers the results of Tillmann and Bakalov-Kirillov as a special case.

		\subparagraph{Acknowledgments.} We are grateful to Christoph Schweigert for his constant interest in this project, helpful comments and for bringing us into contact with Grothendieck-Verdier structures.
		Theorem~\ref{thmcycle2} answers a question asked by Adrien Brochier on MathOverflow in 2015.
		We thank him not only for his question, but also for a lot of discussions 
		from which this project has greatly benefited. 
	We thank
	 Jürgen Fuchs,
	 Claudia Scheimbauer and
	 Nathalie Wahl
	for insightful discussions
	and helpful comments.
	To the authors of \cite{BDSPV15}
	--- Bruce Bartlett, Chris Schommer-Pries, Chris Douglas and Jamie Vicary ---
	we are grateful for a helpful correspondence on rigidity in context of modular functors.
Finally, we thank the anonymous referee whose comments have helped us to improve the manuscript significantly.

LM is supported by the Max Planck Institute for Mathematics in Bonn.
	LW gratefully acknowledges support by the RTG 1670 ``Mathematics inspired by String theory and Quantum
	Field Theory'' (DFG) at the University of Hamburg,
	the Danish National Research Foundation through the Copenhagen Centre for Geometry
	and Topology (DNRF151)
	and  by the European Research Council (ERC) under the European Union's Horizon 2020 research and innovation programme (grant agreement No. 772960).

\section{Cyclic and modular operads and algebras over them\label{seccycmodalg0}}
In this section, we recall the notions of cyclic and modular operads
introduced by Getzler and Kapranov in \cite{gk,gkmod}. 
The definition of \emph{algebras} over cyclic and modular operads requires the notion of an endomorphism operad that is defined by means of non-degenerate pairings. 

\subsection{Preliminaries on the definition of cyclic and modular operads via graphs\label{secrecallcos}}
In \cite{costello} Costello gives a very efficient description of operads, cyclic operads and modular operads based on different categories of graphs.
We will adopt this description and therefore briefly recall the most important definitions:
A \emph{graph} consists of a set $H$ of \emph{half edges} and a set $V$ of 
\emph{vertices} together with a map $H\longrightarrow V$ and an involution $\iota\colon H\longrightarrow H$
specifying how half edges are glued together. 
The orbits of the involution $\iota$ are the \emph{edges} of the graph.
Fixed points of $\iota$ are called \emph{external legs} (\emph{legs}, for short). We denote by $\Legs(\Gamma)$ the set of external legs of a graph $\Gamma$. 
We may realize a graph $\Gamma$ as a topological space $|\Gamma|$ with 
 the vertices of $\Gamma$ as the 0-cells and  the edges of $\Gamma$ as the 1-cells. 
A \emph{corolla} is
a graph with one vertex and only external legs. 
Often we will denote a graph as a pair $\Gamma=(V,H)$ of the set of vertices and the set of half edges suppressing all other parts of the structure in the notation.

Let $\Gamma=(V,H)$ and $\Gamma'=(V',H')$ be graphs. A \emph{morphism} of graphs consists of 
maps $V\longrightarrow V'$ and $H\longrightarrow H'$ which are compatible with the graph structure in 
the obvious way.

\begin{figure}[h]
\begin{center}
\begin{overpic}[scale=1]
{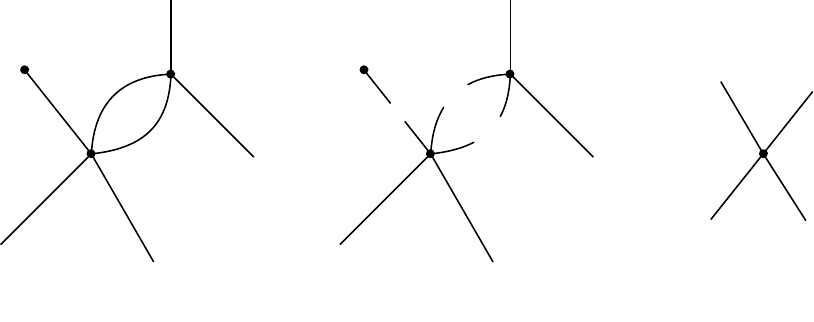}
\put(12,2){\large$\Gamma$}
\put(52,2){\large$\nu(\Gamma)$}
\put(91,2){\large$\pi_0(\Gamma)$}
\end{overpic}
\end{center}
\caption{A sketch for the graphs $\nu(\Gamma)$ and $\pi_0(\Gamma)$.}
\label{Fig:Graph operations}
\end{figure}    
Given a graph $\Gamma$ we can form a new graph $\nu (\Gamma)$ by cutting open all internal edges. 
Formally, this replaces the involution on the half edges by the identity map. 
We can also form a graph $\pi_0(\Gamma)$ by contracting all internal edges. In Figure~\ref{Fig:Graph operations} a pictorial presentation of these operations  
is given. 

The category $\Graphs$ has as objects graphs which are finite disjoint unions of corollas.
 A morphism 
$\gamma_1\longrightarrow \gamma_2$ is 
given by an equivalence class of a graph $\Gamma$ together with isomorphisms 
$\varphi_1\colon \gamma_1 \longrightarrow \nu (\Gamma)$ and $\varphi_2\colon \gamma_2 \longrightarrow \pi_0 (\Gamma)$; note that here $\varphi_1$ and $\varphi_2$ are morphisms of graphs in the above sense, but not morphisms in the category $\Graphs$ (which is named afters its morphisms). Two such triples $(\Gamma,\varphi_1,\varphi_2)$ and $(\Gamma',\varphi_1',\varphi_2')$ are equivalent if there exists an isomorphism $\psi \colon \Gamma\longrightarrow \Gamma'$ satisfying 
$\varphi_1'=\nu(\psi)\circ \varphi_1$ and $\varphi_2'=\pi_0(\psi)\circ \varphi_2$. The composition $\Gamma_2\circ \Gamma_1$ is defined by
replacing the vertices of $\Gamma_2$ by the graph $\Gamma_1$, see Figure~\ref{Fig: Composition}.
We denote by $\Forests$ the subcategory of $\Graphs$ whose objects are those objects of $\Graphs$ which do not contain a corolla with zero legs and whose morphisms are forests, i.e.\ disjoint unions of contractible graphs.
\begin{figure}[h]
\begin{center}
\begin{overpic}[scale=0.7]
{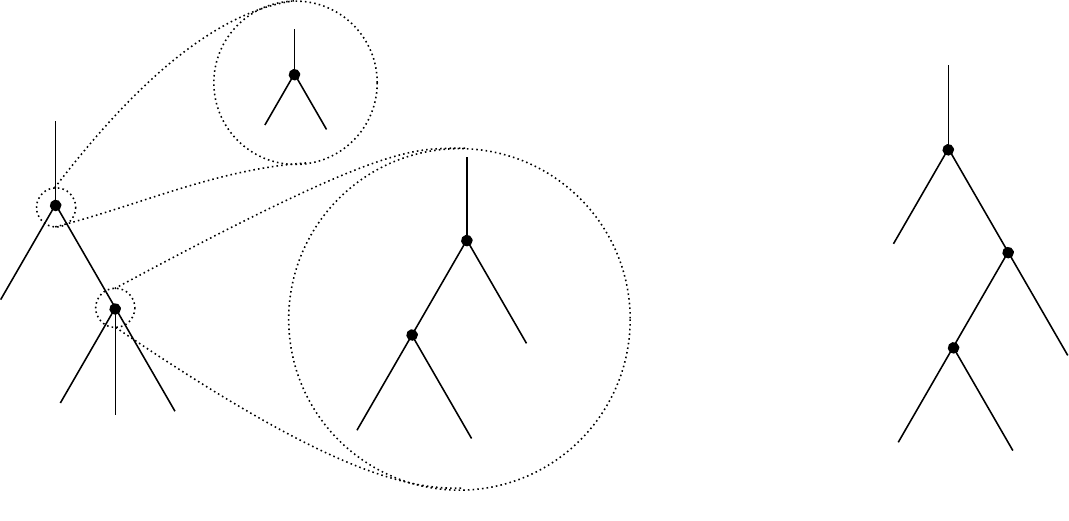}
\put(5,0){\large$\Gamma'$}
\put(35,0){\large$\Gamma$}
\put(65,20){\large$=$}
\put(85,0){\large$ \Gamma' \circ \Gamma$}
\end{overpic}
\end{center}
\caption{A sketch for the composition $\Gamma' \circ\Gamma $ of two morphisms $\Gamma
,\Gamma'$ in $\Graphs$.}
\label{Fig: Composition}
\end{figure}
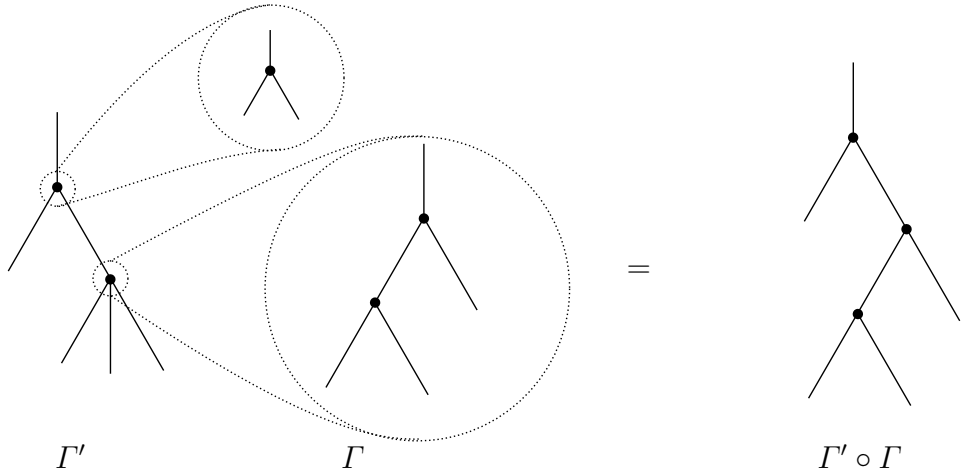

Finally, we define the  category $\RForests$ of \emph{rooted forests}: A \emph{rooted graph} 
is a graph $\Gamma$ equipped with a section $s\colon V(\pi_0(\Gamma))\longrightarrow \Legs(\Gamma)$ of the 
obvious map $\Legs(\Gamma)\longrightarrow V(\pi_0(\Gamma))$, i.e.\
in each component of $\Gamma$ we distinguish an external leg that we refer to as the \emph{root}. 
 Morphisms of rooted graphs are morphisms of
the underlying graph which are compatible with the specified sections. Note that a rooted forest $\Gamma$ induces 
the structure of a rooted graph on $\nu (\Gamma)$ by declaring for every vertex the edge in the direction
of the root of $\Gamma$ as the root of the vertex in $\nu(\Gamma)$.
The category $\RForests$ is defined just like $\Forests$ with objects and morphisms replaced by their rooted 
version. 
There is a functor $\RForests \longrightarrow \Forests$ which forgets the root.  
 \begin{remark}
 Our definitions slightly differ from the definitions used in \cite{costello} where  only  at least trivalent corollas are allowed as objects
 of $\Forests$ and  $\RForests$. In other words, in \cite{costello}  operads without arity zero and arity one operations are considered.   
 \end{remark}

The categories $\RForests, \Forests$ and $\Graphs$ are used in 
\cite{costello}
to give a definition of operads,
cyclic operads and modular operads with values in a symmetric monoidal (higher) category $\cat{S}$, see also \cite[Section~2]{giansiracusam}.
Before making this precise in the special case that we are interested in,
let us give the idea:
An ordinary/cyclic/modular operad in $\cat{S}$ is a symmetric monoidal functor \begin{align} \label{eqndefop}
\mathcal{O}:\RForests/\Forests/\Graphs \to \cat{S} \ . \end{align}
Here `symmetric monoidal' has to be understood in the appropriate weak sense (to be made precise momentarily in the bicategorical case). 
This very nicely allows us to define ordinary/cyclic/modular operad for which the associativity of operadic composition is relaxed up to coherent homotopy. 

In this paper,
the emphasis lies on aspherical topological operads, i.e.\ topological operads whose spaces of operations are aspherical, where aspherical means that all homotopy groups of degree two and higher are trivial. 
Famous examples include the associative operad and the (framed) little disks operad
that will be treated in Section~\ref{seccycas} and \ref{secfE2GV} of this article, respectively. 
Such operads can naturally be regarded as groupoid-valued or, more generally, category-valued. We will see below in Section~\ref{seccycmodalg} that for category-valued operads, one can naturally consider algebras in any symmetric monoidal bicategory. 
Therefore, we will develop the theory of ordinary/cyclic/modular operads and their algebras with values in a symmetric monoidal \emph{bi}category, 
which is also in line with our motivations coming from quantum topology.

We assume some familiarity with the theory of symmetric monoidal bicategories; we refer e.g.\ to \cite[Chapter~2]{schommerpries} for a detailed discussion and to \cite{leinster} for a short introduction (however without the treatment of monoidal structures). In particular, we rely on the following notions: By a \emph{bicategory} we mean a three-layered categorical structure with objects, 1-morphisms and 2-morphisms (sometimes also called 0-cells, 1-cells and 2-cells)
in the weak sense. A morphism between bicategories (sometimes also referred to as (weak) 2-functor) will just be called \emph{functor}. \emph{Symmetric monoidal functors} between \emph{symmetric monoidal bicategories}  are to be understood in a strong (not in any kind of lax) sense unless otherwise stated.  
In particular, a symmetric monoidal functor between symmetric monoidal bicategories comprises various sorts of coherence data subject to coherence conditions. 
We will briefly illustrate this after the next definition.

\begin{definition}\label{Def: Operads}
	Let $\cat{M}$ be a symmetric monoidal bicategory.
	An \emph{operad} in $\cat{M}$ is a symmetric monoidal functor $\cat{O}:\RForests \to \cat{M}$, where we 
	consider $\RForests$ as a symmetric monoidal bicategory with only trivial 2-morphisms. 
	We define \emph{cyclic} and \emph{modular operads} by replacing $\RForests$ with $\Forests$ and $\Graphs$, respectively.
\end{definition}

As explained before the definition, we understand the symmetric monoidal functors $\cat{O}$ out of $\RForests$, $\Forests$ and $\Graphs$ to $\cat{M}$ to be weak and not strict. 
This means in particular that for composable morphisms $T \ra{\Gamma} T' \ra{\Gamma'} T''$ the 1-morphisms $\cat{O}(\Gamma'\circ \Gamma)$ and $\cat{O}(\Gamma')\circ \cat{O}(\Gamma)$ are not necessarily equal, but related by a specific 2-isomorphism \begin{align} \label{eqncoherenceiso}\cat{O}(\Gamma'\circ \Gamma)\cong\cat{O}(\Gamma')\circ \cat{O}(\Gamma)
\end{align} that is part of the data. Furthermore, there are 1-equivalences $\cat{O}(T) \otimes \cat{O} (T') \ra{\simeq} \cat{O}(T\sqcup T')$.

Let us compare Definition~\ref{Def: Operads} to the `usual' definition of a (symmetric) operad $\mathcal{O}$ in $\cat{M}$ which is typically given in terms of the following data:

\begin{itemize}
	\item 
	Objects $(\mathcal{O}(n))_{n\ge 0}$ in $\cat{M}$ carrying a right action of the permutation group $\Sigma_n$ on $n$ letters. We interpret $\mathcal{O}(n)$ as the object of $n$-ary operations. 
	
	\item   Maps for all $1 \leq j \leq n$ and $m\geq 0$
	\begin{align}
	\circ_j \colon \mathcal{O}(n) \otimes \mathcal{O}(m) \to \mathcal{O} (n-1+m) \ . 
	\end{align}
	We think of these maps as partial composition maps.
\end{itemize}
Note that we consider a definition of an operad \emph{without} operadic identity.

The description of an operad just given in terms of objects of operations can be extracted from a symmetric monoidal functor $F: \RForests \to \cat{M}$. To see this, evaluate $F$ on the corolla $T_{n}$ with $n+1$ legs which we identify with the set $\{0,\dots , n \}$ with $0$ marked as the root. We then obtain the objects $\mathcal{O}(n):=F(T_{n})$. From a permutation $\sigma \in \Sigma_n$, we can construct an automorphism of $T_{n}$ with $T_n$ as underlying graph. The isomorphism $\varphi_2$  permutes the edges $\{1,\dots ,  n\}$ according to $\sigma$ and $\varphi_1$ is defined to be the identity of $T_n$. This gives us a $\Sigma_n$-action on $\mathcal{O}(n)$ up to coherent isomorphism.
The coherence isomorphisms for the $\Sigma_n$-action are a special case of the isomorphisms~\eqref{eqncoherenceiso} applied to permutation morphisms.
By evaluation of $F$ on morphisms in $\RForests$ and monoidality of $F$ we obtain the composition map for the objects $\mathcal{O}(n)$. The equivariance and associativity axioms follow from the functoriality of $F$
(again, these holds up to coherent isomorphism thanks to~\eqref{eqncoherenceiso}). This can be seen to provide an equivalence between the two descriptions which holds analogously for cyclic and modular operads, see also \cite{costello}. In the case of cyclic 
and modular operads,
the cyclic permutation of legs induces an isomorphism $\tau_n : T_n \cong T_n$ in $\Forests$ (or $\Graphs$)
 with $\tau_n^{n+1}=\id_{T_n}$. As above, the underlying graph of $\tau_n$ is $T_n$, $\varphi_1$ is the identity and $\varphi_2$ the map sending the edge $i$ to $i+1$ modulo $n$ 
(note that $\tau_n$ is \emph{not} a morphism in $\RForests$). These morphisms give rise to an action
of the cyclic group $\Z_{n+1}$ on $\cat{O}(T_n)$ which is part of the more traditional definition of cyclic (and modular)
 operads. 
 
Again, let us emphasize that this defines operads \emph{without} an operadic identity. However, all 
operads appearing in this paper have operadic unit, and it will be important to correctly keep track 
of those. To this end, we make the following definition:

\begin{definition}
Let $\mathcal{O}$ be an $\cat{M}$-valued modular operad. An \emph{operadic identity} for $\mathcal{O}$ consists of the following data:
\begin{itemize}
	\item 
A 1-morphism $1_\mathcal{O} \colon I \to  \mathcal{O}(T_1)$ which is a homotopy fixed point for the $\Z_2$-action on $\mathcal{O}(T_1)$, i.e.\ it comes with a 2-isomorphism
\begin{equation}
	\begin{tikzcd}[row sep=0.4cm] 
		& \mathcal{O}(T_1) \ar[dd, "\mathcal{O}(\tau_1)"]  \\ 
		I \ar[ru, "1_\mathcal{O}"] \ar[rd,"1_\mathcal{O}",swap,""{name=U}] & \\ 
		\  & \mathcal{O}(T_1)
		\ar[from=1-2, to=U, shorten <=5, shorten >=5, Leftarrow,"\cong"]
	\end{tikzcd}
\end{equation} squaring to the identity,
\item and for all $n\in \N$ and all morphisms $\Gamma \colon T_1 \sqcup T_n \to T_n$ in $\Graphs$,
2-isomorphisms 
\begin{equation}
\begin{tikzcd}
I\otimes \cat{O}(T_n) \ar[dd, "\simeq", swap] \ar[rrrdd, Rightarrow, "1_{\Gamma}"  near start, shorten <=5, shorten >=70, swap] \ar[rr, "1_\cat{O}\otimes \id_{\cat{O}(T_n)}"] & & \mathcal{O}(T_1) \otimes \cat{O}(T_n) \ar[r, "\simeq"] & \mathcal{O}(T_1 \sqcup T_n) \ar[llldd, "\cat{O}(\Gamma)", bend left=20] \\ 
 & & & \\
 \mathcal{O}(T_n) & & & \ 
\end{tikzcd}
\end{equation}
\end{itemize} 
such that
the application of the unitality morphisms $(1_\cat{O},1_\Gamma)$ 
is compatible with composition
(applying the unitality morphisms 
\emph{before} 
and \emph{after} composing
is required to yield the same result) and cyclic symmetry.   
In more detail:\enlargethispage*{1cm}
For all morphisms $\Omega \colon \sqcup_{i=1}^m T_{n_i} \to T_{n'}$, $\Gamma \colon T_1 \sqcup T_{n'} \to T_{n'}$ and $\Gamma'\colon  T_1 \sqcup T_{n_j} \to T_{n_j} $ in $\Graphs$ such that $ \Gamma \circ  ( \tau_1^{\varepsilon} \sqcup \Omega) = \Omega \circ (\id_{ \sqcup_{i\neq j} T_{n_i}} \sqcup \Gamma') $, where $\varepsilon$ is either $0$ or $1$,
\begin{equation}
\begin{tikzcd}
I\otimes \mathcal{O}(\sqcup_{i=1}^m T_{n_i}) \ar[dd, "\simeq", swap]  \ar[rr, "\id_I \otimes \mathcal{O}(\Omega)"] & & I \otimes \mathcal{O}(T_{n'})   \ar[rr, "1_\cat{O}\otimes \id_{\cat{O}(T_{n'})}"]  \ar[rrrdd, Rightarrow, "1_{\Gamma}"  near start, shorten <=5, shorten >=70, swap] \ar[dd, "\simeq"] & & \mathcal{O}(T_1) \otimes \cat{O}(T_{n'}) \ar[r, "\simeq"] & \mathcal{O}(T_1 \sqcup T_{n'}) 
 \ar[llldd, "\cat{O}(\Gamma)", bend left=20] \\ 
 & & & & & \\
 \mathcal{O}(\sqcup_{i=1}^m T_{n_i}) \ar[rr, " \mathcal{O}(\Omega)", swap] \ar[rruu, Rightarrow, "\cong" , shorten <=5, shorten >=5, swap]  & & \mathcal{O}(T_{n'}) & & & \ 
\end{tikzcd}
\end{equation}
\begin{equation}
=
\begin{tikzcd}
	I\otimes \mathcal{O}(\sqcup_{i=1}^m T_{n_i}) \ar[ddddd, "\simeq", swap] \ar[rrd, "1_\cat{O}\otimes \id_{\mathcal{O}(\sqcup_{i=1}^m T_{n_i})}"] \ar[rrdd, "1_\cat{O}\otimes \id_{\mathcal{O}(\sqcup_{i=1}^m T_{n_i})}", swap]  \ar[rr, "\id_I \otimes \mathcal{O}(\Omega)"]  & & I \otimes \mathcal{O}(T_{n'}) \ar[d, Rightarrow,  shorten <=2, shorten >=2, "\cong"]  \ar[rr, "1_\cat{O}\otimes \id_{\cat{O}(T_{n'})}"]  & & \mathcal{O}(T_1) \otimes \cat{O}(T_{n'}) \ar[r, "\simeq"] & \mathcal{O}(T_1 \sqcup T_{n'}) 
	\ar[lllddddd, "\cat{O}(\Gamma)", bend left=20] \\ 
 \ar[ddddrr, Rightarrow,  shorten <=20, shorten >=60, "1_{\Gamma'}"]	& \ \ar[r, Rightarrow, shorten <=17, shorten >=2,"\ \ \cong"] & \mathcal{O}(T_1)  \otimes \mathcal{O}(\sqcup_{i=1}^m T_{n_i}) \ar[d, "\mathcal{O}(\tau_1^{\varepsilon})\otimes \id_{\mathcal{O}(\sqcup_{i=1}^m T_{n_i})}"] \ar[rru, "\id_{\cat{O}(T_1)} \otimes \cat{O}(\Omega) ", swap]  & & & \\ 
 & & \mathcal{O}(T_1)  \otimes \mathcal{O}(\sqcup_{i=1}^m T_{n_i}) \ar[d, "\simeq"] & & & \\
	& &   \mathcal{O}(T_1\sqcup T_{n_j}) \otimes \cat{O}( \sqcup_{i\neq j} T_{n_i})  \ar[lldd, "\mathcal{O}(\Gamma')\otimes \id "] \ar[rr, Rightarrow,  shorten <=2, shorten >=20, "\cong"] & &  \ & \\
	& & & & & \\
	\mathcal{O}(\sqcup_{i=1}^m T_{n_i}) \ar[rr, "\mathcal{O}(\Omega)", swap]  &  \ & \mathcal{O}(T_{n'}) & & & \ 
\end{tikzcd}
\end{equation} 
holds, where the 1-equivalences labeled with $\simeq$ are the coherences of $\cat{M}$ and
$\cat{O}$ and where the unlabeled 2-morphisms come from the homotopy fixed point structure for the operadic unit in the case $\varepsilon=1$ and trivial for $\varepsilon=0$. Note that we suppress some coherence isomorphisms for readability in the diagram.  
We say that $\mathcal{O}$ is a \emph{unital operad} if there exists a unit in the sense just defined (we do not fix 
the unit).  
\end{definition}
The distinction of the two cases $\varepsilon=0,1$ is important to ensure the right compatibility with the morphisms
$\tau_n$.  

Units for cyclic and ordinary operads are defined analogously, where in the ordinary case one does 
not require it to be a homotopy fixed point.

\begin{remark}\label{Rem: Uniqunes of units}
One might expect the choice of a unit to be structure on a modular (or cyclic or ordinary) operad $\mathcal{O}$. However, if we 
have two choices of units $(1_{\mathcal{O}},1_{\Gamma})$ and $(1'_{\mathcal{O}},1'_{\Gamma})$,
 there exists
a unique 2-isomorphism
between $1_{\mathcal{O}}$ and $1'_{\mathcal{O}}$ compatible with 
the 2-isomorphisms $1_{\Gamma}$ and $1'_{\Gamma}$.
It is
 constructed by using the morphism $c \colon T_1\sqcup T_1 \to T_1$ that realizes the composition of arity one operations:
\begin{equation}
\begin{tikzcd}
\ar[rrdd, Rightarrow, "1'_c" near end,  shorten <=50, shorten >=2] & & \cat{O}(T_1) \\ 
 & & \\
I \ar[rruu, "1_\cat{O}"] \ar[rr, " 1_\cat{O} \otimes 1'_\cat{O}"] \ar[rrdd, "1'_\cat{O}",swap] & &  \cat{O}(T_1) \otimes \cat{O}(T_1) \ar[dd, "\cat{O}(c)"] \ar[uu, "\cat{O}(c)",swap] \\ 
 & & \\
 \ar[uurr, Leftarrow, "1_c^{-1}" near end,  shorten <=55, shorten >=2] & & \cat{O}(T_1)
\end{tikzcd}
\end{equation}
In this sense, the existence of an operadic unit is a property of $\mathcal{O}$ and not a structure.    
\end{remark}

\begin{example}\label{exsymmoncat}
	In this article, the following examples of symmetric monoidal bicategories will be relevant: \small 
	\begin{center}
	\begin{tabular}{l|p{0.3\textwidth}|p{0.15\textwidth}|p{0.2\textwidth}|p{0.15\textwidth}}
		 & objects & 1-morphisms & 2-morphisms & monoidal structure \\ \hline \hline
		$\Cat$ & categories & functors & natural transformations & Cartesian product \\ 
		$\Fin$ & finite categories over our fixed algebraically closed field $k$ \footnotesize (linear Abelian cate\-gories 
		with finite-dimensional morphism spaces, enough projective
		objects, finitely many isomorphism classes of simple objects such that every object has finite length) \normalsize & left exact functors & natural transformations & Kelly product (coincides here with the Deligne product)
	\end{tabular}
\end{center}
\normalsize
The symmetric monoidal bicategory $\Fin$ is defined dually to a finite version of $\Rex$ in \cite{bzbj}, see also \cite[Theorem~3.2]{fss}.
It is frequently used in many areas of representation theory, in particular in quantum algebra. 
Recall that a $k$-linear category is finite if and only if it is $k$-linearly equivalent to the category of finite-dimensional modules over some finite-dimensional $k$-algebra; we refer e.g.\ to \cite[Proposition~1.4]{dss} for this well-known statement. 
Note that describing a finite category as the finite-dimensional modules over a finite-dimensional algebra is rarely useful because we will mostly consider finite categories with additional structure like a monoidal product, and this additional structure cannot necessarily be described on the level of the algebra.
\end{example}

\begin{remark}[Graphical calculus in symmetric monoidal bicategories]\label{remgraphcalc}
	Symmetric monoidal bicategories admit a graphical calculus
	that we briefly recall now:
	Objects correspond to lines, and the monoidal product corresponds to juxtaposition.
	The juxtaposition of the empty collection of lines (i.e.\ no line)
	represents the monoidal unit.
	A 1-morphism $F: \bigotimes_{i=1}^m X_i \to \bigotimes_{j=1}^n Y_j$ will be written as a box 
	\begin{align}
	\input{graphical_calculus.tikz}
	\end{align}
	with $m$ ingoing legs labeled with the objects $X_1,\dots,X_m$ and $n$ outgoing legs labeled with $Y_1,\dots,Y_n$.
	The legs corresponding to the source objects and target objects will be attached to the bottom and the top of the box representing $F$, respectively. 
	The symmetric braiding will be represented by a crossing of lines (thanks to symmetry, overcrossing and undercrossing need not be distinguished).
	A 2-morphism $\alpha$ between 1-morphisms $F$ and $G$ with coinciding source and target object will be represented by an arrow
	\begin{align}
	\input{graphical_calculus_2.tikz}
	\end{align}
	allowing us to efficiently write commuting diagrams for 2-morphisms. We suppress coherence morphisms in the 
graphical calculus because these can be inserted in an essentially unique way (up to canonical higher isomorphism).
\end{remark}

\subsection{Non-degenerate pairings\label{seccycmodalg}}
For an operad $\mathcal{O}$ in $\Cat$, one can define an algebra over $\mathcal{O}$ in a
symmetric monoidal bicategory $\mathcal{M}$ as an operad map from $\mathcal{O}$ to the $\Cat$-valued endomorphism operad built from some object in $\mathcal{M}$ that is supposed to carry the algebra structure; we refer to \cite{markl,lodayvallette,FresseI} for an introduction to the theory of operads and the algebras over them.
To generalize the notion of an algebra over an operad to cyclic and modular operads, we need a cyclic and modular version of an endomorphism operad.
This is accomplished by considering (non-degenerate and symmetric) pairings on objects in $\mathcal{M}$ \cite{gk}. 
If we are dealing with vector spaces as our target category, then it is clear what is meant by a non-degenerate pairing. For higher categories, a little more care is required.

\begin{definition}\label{defpairing}
	Let $\cat{M}$ be a symmetric monoidal bicategory.
	A \emph{pairing} on $X \in \cat{M}$ is defined to be a morphism $\kappa : X\otimes X\to I$, where $I\in\cat{M}$ is the monoidal unit of $\cat{M}$.
A pairing $\kappa : X \otimes X \to I$ is called \emph{non-degenerate}
	if $\kappa$ exhibits $X$ as its own dual  in the homotopy category of $\cat{M}$ (by symmetry left and right dual coincide). 
		The pairing $\kappa$ exhibiting $X$ as its own dual means the following: There is a map $\Delta: I \to X\otimes X$, called \emph{coevaluation}, such that $\kappa$ and $\Delta$ satisfy the usual snake relations
	\begin{align}\label{eqnsnakeref}
		\input{snake.tikz}
	\end{align}	
	up to isomorphism, i.e.\ they lead to \emph{snake isomorphisms} instead of snake relations. 
\end{definition}

\begin{remark}\label{coevaluationrem}
The isomorphisms in~\eqref{eqnsnakeref} can always be chosen to be coherent~\cite[Section 2]{Pstragowski}, and our convention will be to always choose them that way, see also Remark~\ref{remsymconcrete} for more details. 	
\end{remark}

\begin{remark}\label{remadjunctions}
	If $\kappa : X \otimes X\to I$ is a non-degenerate pairing on an object $X$ in a symmetric monoidal bicategory,
	we obtain adjunctions $-\otimes X\dashv -\otimes X$ and $X\otimes - \dashv X\otimes -$, i.e.\
	natural equivalences
	\begin{align}
	\cat{M}(Y\otimes X,Z)&\simeq \cat{M}(Y,Z\otimes X) \ , \label{eqnfirstequivdual} \\
	\cat{M}(X\otimes Y,Z)&\simeq \cat{M}(Y,X\otimes Z) \ 
	\end{align}for $Y,Z\in\cat{M}$.
	These equivalences can be expressed explicitly in terms of the pairing and its coevaluation $\Delta$. For example, \eqref{eqnfirstequivdual} sends $f:Y\otimes X\to Z$ to
	\begin{align}
Y \ra{\id_Y \otimes \Delta} Y\otimes X\otimes X\ra{f\otimes \id_X} Z\otimes X \ . 
\end{align}
	\end{remark}

As motivated above, we can use the notion of a pairing
to define cyclic and modular endomorphism operads.
It will be crucial to consider symmetric pairings:

\begin{definition}\label{defsymmetricstructure}
	Let $\cat{M}$ be a symmetric monoidal bicategory.
	For $X \in \cat{M}$, consider the $\mathbb{Z}_2$-action on $X\otimes X$ coming from the symmetric braiding on $\cat{M}$ and the induced $\mathbb{Z}_2$-action on the morphism category $\cat{M}(X\otimes X,I)$.
	We define a \emph{symmetric pairing} on $X$ as a homotopy fixed point of the $\mathbb{Z}_2$-action on $\cat{M}(X\otimes X,I)$.
	We call a symmetric pairing \emph{non-degenerate} if the underlying pairing is non-degenerate.
\end{definition}

\begin{remark}\label{remsymconcrete}
Concretely,
the structure of a homotopy fixed point on a pairing $\kappa : X \otimes X \to I$ consists of a natural isomorphism 
$\Sigma_\kappa \colon \kappa  \cong \kappa  \tau$, where $\tau$ is the braiding $X\otimes X\to X\otimes X$, such that the composition $\kappa \cong\kappa \tau \cong \kappa \tau^2 =\kappa$ is the identity transformation. The symmetry can dually be  described by an 
isomorphism $ \Sigma_\Delta \colon \Delta \cong \tau  \Delta$ for the coevaluation $\Delta \colon I \to X\otimes X$  via
\begin{align}
\input{symmetrycopairing.tikz}
\end{align}
where the unlabeled isomorphisms are the snake isomorphisms. Again, we find that the composition $  \Delta \cong 
\tau \Delta \cong \tau^2 \Delta = \Delta$ is the identity. 
The coherence relations for the snake isomorphism imply that various
diagrams one can form using the symmetry isomorphisms $\Sigma_\kappa$ and $\Sigma_\Delta$ commute.
For example, the diagram 
\begin{align}\label{Eq:compatability symmetry}
\input{diagram_symmetry.tikz}
\end{align}
commutes. 
\end{remark}

\subsection{Cyclic and modular endomorphism operads\label{seccyclicmodalg}}
Having defined non-degenerate symmetric pairings,
we are now in a position to define cyclic and modular endomorphism operads:
As a first step, let us define for a corolla $T\in \Graphs$ and an object $X\in\cat{M}$ in a symmetric monoidal bicategory $\cat{M}$ the \emph{unordered monoidal product} $X^{\otimes \Legs (T)}$ (we could make this definition for any finite set instead of the set $\Legs(T)$ of legs of $T$).
The definition uses that $\cat{M}$ is symmetric monoidal and can be given as follows:
Let $\cat{L}$ be the category whose objects are orders of the set $\Legs(T)$ and whose morphisms are 
bijections compatible with the orders.
Consider the natural functor $\cat{L}\longrightarrow \cat{M}$ sending an object
$\{\ell_1<\ell_2<\dots < \ell_n\}$ to $\otimes_{j=1}^n X$ and a morphism to the corresponding permutation of tensor factors. We define $X^{\otimes \Legs(T)}$ as the 
2-colimit of this functor.
After choosing an order for $\Legs(T)$, the unordered monoidal product is equivalent to $X^{\otimes |\Legs(T)|}$, but the unordered monoidal product has the advantage that it can be defined \emph{without choosing an order}. 

Now for a symmetric pairing $\kappa : X \otimes X \to I$ on an object $X$ in a symmetric monoidal bicategory $\cat{M}$ and a corolla $T$, we set \begin{align} \End_\kappa^X (T) :=  \cat{M}\left(X^{\otimes \Legs(T)},I\right) \end{align} 
and extend monoidally.
For a morphism $(\Gamma, \varphi_1,\varphi_2)\colon \sqcup_{i=1}^{k_1} T^{(i)} \longrightarrow \sqcup_{j=1}^{k_2} {T'}^{(j)}$ in $\Graphs$, we need to construct a functor  
\begin{align}\label{eqnendonmorph}
\End_\kappa^X(\Gamma, \varphi_1,\varphi_2) :
 \prod_{i=1}^{k_1} \cat{M}\left( X^{\otimes \Legs(T^{(i)})},I \right) \longrightarrow \prod_{j=1}^{k_2}  \cat{M}\left(    X^{\otimes \Legs({T'}^{(j)})},I \right) \  
\end{align}
(later we will again suppress $\varphi_1$ and $\varphi_2$ in the notation and write $\End_\kappa^X(\Gamma)$ instead of $\End_\kappa^X(\Gamma, \varphi_1,\varphi_2)$).
For the definition of~\eqref{eqnendonmorph},
we can concentrate on the case that $\Gamma$ is connected.
Afterwards, $\End_\kappa^X$ can be extended monoidally to non-connected graphs.
In the connected case, $k_2=1$ holds, and we write $T'={T'}^{(1)}$. 
As a first step, we define maps
\begin{align}\label{axufamilymapseqn0}
 {X^{\otimes |\Legs(T')|}}\to  \bigotimes_{i=1}^{k_1} {X^{\otimes |\Legs(T^{(i)})|}}
\end{align}
for every ordering of 
 $\Legs(T^{(1)}),\dots ,\Legs(T^{(k_1)})$ and $\Legs(T')$ (an order is chosen for \emph{each} of these sets separately). These are defined by the commutativity of the square
\begin{equation}
\begin{tikzcd}
 {X^{\otimes |\Legs(T')|}}\ar[swap]{dd}{}  \ar[]{rr}{\simeq } & &   {X^{\otimes |\Legs(T')|}}   
\otimes
\left(  \substack{\otimes\\ \text{internal} \\  \text{edges of $\Gamma$}  } \, I  \right)  \ar{dd}{       \id \otimes       \left(  \substack{\otimes\\ \text{internal} \\  \text{edges of $\Gamma$}  }\ \  \Delta  \right)    }   \\
& & \\
\otimes_{i=1}^{k_1} {X^{\otimes |\Legs(T^{(i)})|}} & &  {X^{\otimes |\Legs(T')|}}   
\otimes
\left(  \substack{\otimes\\ \text{internal} \\  \text{edges of $\Gamma$}  } \, X\otimes X  \right) \ar{ll}{\varphi_1,\varphi_2}
\end{tikzcd} 
\end{equation}
Here the right vertical arrow uses the coevaluation,  and the lower horizontal arrow uses the identifications induced by 
$\varphi_1$ and $\varphi_2$. In order to construct the latter map, recall that by definition $\varphi_1$ and $\varphi_2$ tell us how to identify the legs of $ T'$ and the pairs of edges obtained by cutting at the internal edges of $\Gamma$ with the legs of $\sqcup_{i=1}^{k_1} T^{(i)}$. 

The symmetry of $\Delta$ ensures that this family \eqref{axufamilymapseqn0}
of functors is compatible in the sense that they descend to the 2-colimit used to define unordered monoidal products.
Hence, they provide maps
\begin{align} \label{eqnpreresultunordered0}  {X^{\otimes \Legs(T') }}  \longrightarrow   \bigotimes_{i=1}^{k_1} X^{\otimes \Legs(T^{(i)})} \ .  \end{align}
We now define the functor $\End_\kappa^X(\Gamma, \varphi_1,\varphi_2)$ 
by
\begin{align}
\End_\kappa^X(\Gamma, \varphi_1,\varphi_2) :  \prod_{i=1}^{k_1} \cat{M}\left( X^{\otimes \Legs(T^{(i)})},I \right) \to \cat{M} \left(   \bigotimes_{i=1}^{k_1} X^{\otimes \Legs(T^{(i)})} ,  I   \right) \to     \cat{M}(   X^{\otimes \Legs(T') },I  ) \ ,
\end{align}
where the first map is induced by the tensor product in $\cat{M}$, and the second is the precomposition
with~\eqref{eqnpreresultunordered0}.

As one verifies directly, these assignments define the desired  modular endomorphism operad in the sense that the following statement holds:
\begin{proposition}\label{propendoperad}
	Let $\cat{M}$ be a  symmetric monoidal category and $\kappa$ be a non-degenerate symmetric pairing on $X\in\cat{M}$. Then $\End_\kappa^X : \Graphs \to \Cat$ is a symmetric monoidal functor, i.e.\ a modular operad in $\Cat$. We call this modular operad the modular endomorphism operad of $(X,\kappa)$.
	Furthermore, $\End_\kappa^X$ admits an operadic unit given by the pairing $\kappa \in \End_\kappa^X(T_1)$.
\end{proposition}

By restricting $\End_\kappa^X$ to $\Forests$ 
we get the \emph{cyclic endomorphism operad} of $(X,\kappa)$.
By further pulling back to $\RForests$ we find the endomorphism 
operad of $(X,\kappa)$.
It is important to note that for $\End_\kappa^X$
to be a symmetric monoidal functor, non-degeneracy of $\kappa$ is not needed. It is needed, however, to identify the restriction of $\End_\kappa^X$ with the endomorphism operad in the `traditional sense'. Indeed, the natural map 
		\begin{align} 
		\input{relation_usual_endomorphism.tikz} \ ,
		\end{align}  
		induced by a choice of root provides an equivalence between $\End_\kappa^X(T_{n})$ and $\mathcal{M}(X^{\otimes n},X)$. Therefore, we insist on non-degeneracy in Proposition~\ref{propendoperad}.

\subsection{Cyclic and modular algebras\label{seccycmodalg}}
Using the cyclic and modular endomorphism operad, one defines cyclic and modular algebras, respectively:
		
\begin{definition}\label{defmodalgM}
	Let $\mathcal{O}$ be a modular operad in $\Cat$ and $\cat{M}$ any symmetric monoidal bicategory.
	A \emph{modular $\mathcal{O}$-algebra} $A$ in  $\cat{M}$  
	is an object 
	$X\in\cat{M}$ together with the choice of a non-degenerate symmetric pairing $\kappa$ on $X$ and
	a symmetric monoidal transformation $A: \mathcal{O} \to \End_\kappa^X$ of symmetric monoidal 
	functors $\Graphs \to \Cat$.
	We denote the components of a modular algebra $A$ by $A_T:\mathcal{O}(T)\to\End_\kappa^X(T)$ for $T\in \Graphs$ and its naturality isomorphisms
	for every morphism $\Gamma:T\to T'$ by
	\begin{equation}
		\begin{tikzcd}
			\mathcal{O}(T) \ar[dd,swap,"\mathcal{O}(\Gamma)"]  \ar[]{rr}{A_T } & &   \End_\kappa^X(T)  \ar{dd}{     \End_\kappa^X(\Gamma)  }  \ar[lldd, Rightarrow,"A_\Gamma"]  \\
			& & \\
			\mathcal{O}(T') \ar[swap]{rr}{A_{T'}} & &  \End_\kappa^X(T') \ .
		\end{tikzcd} 
	\end{equation}
	 If $\cat{O}$ admits a unit $(1_\mathcal{O},1_\Gamma)$, we additionally require $A$ to be equipped with compatibility data in the form of a
	natural isomorphism
	\begin{equation}
		\begin{tikzcd}[row sep=0.4cm] 
			& \mathcal{O}(T_1) \ar[dd, "A_{T_1}"]  \\ 
			\star \ar[ru, "1_\mathcal{O}"] \ar[rd,"1_{\End_\kappa^X}=\kappa",swap,""{name=U}] & \\ 
			\  & \End_\kappa^X(T_1)
			\ar[from=1-2, to=U, shorten <=5, shorten >=5, Rightarrow, "1_A"]
		\end{tikzcd}
	\end{equation}
	such that 
	\begin{equation} 
		\begin{tikzcd}[row sep=0.4cm] 
			& \mathcal{O}(T_1) \ar[dd, "A_{T_1}"]  \ar[r,"\mathcal{O}(\tau_1)"] & \mathcal{O}(T_1) \ar[ldd, Rightarrow, shorten <=5, shorten >=5,"A_{\tau_1}"] \ar[dd, "A_{T_1}"] \\ 
			\star \ar[rru, bend left =65, ""{name=S}, "1_\mathcal{O}"] \ar[ru, "1_\mathcal{O}"] \ar[rd,"1_{\End_\kappa^X}",swap,""{name=U}] & \\ 
			\  & \End_\kappa^X(T_1) \ar[r,"\End_\kappa^X(\tau_1)",swap] & \End_\kappa^X(T_1)  
			\ar[from=1-2, to=U, shorten <=5, shorten >=5, Rightarrow,"1_A"]
			\ar[from=S, to=1-2, shorten <=3, shorten >=3, Rightarrow,swap,"\cong\ \ "]
		\end{tikzcd}
		= 
		\begin{tikzcd}[row sep=0.4cm]
			& \mathcal{O}(1) \ar[dd, "A_{T_1}"]  \\ 
			\star\ar[dd, "1_{\End_\kappa^X}",swap] \ar[ru, "1_\mathcal{O}"] \ar[rd,"1_{\End_\kappa^X}",swap,""{name=U}] & \\ 
			\  & \End_\kappa^X(1)\ 
			\ar[from=1-2, to=U, shorten <=5, shorten >=5, Rightarrow] \\
			\End_\kappa^X(1) \ar[ru,swap, "\End_\kappa^X(\tau_1)"] &
			\ar[from=U, to=4-1, shorten <=12, shorten >=5, Rightarrow]
		\end{tikzcd}
	\end{equation}
and for all morphisms $\Gamma \colon T_1\sqcup T_n \to T_n$ in $\Graphs$
\begin{equation}
\begin{tikzcd}
 & \ar[d, Rightarrow, "1_\Gamma"] & \\
\star\times \cat{O}(T_n) \ar[d, swap, "\star\times A_{T_n}"] \ar[r, "1_{\cat{O}}"] \ar[rrr, bend left=25, "\simeq"] & \mathcal{O}(T_1)\times \cat{O}(T_n) \ar[ld, Rightarrow, shorten <=5, shorten >=5, "1_A\times A", swap]  \ar[rr, "\cat{O}(\Gamma)"] \ar[d,  "A_{T_1}\times A_{T_n}"] & &	\cat{O}(T_n) \ar[d,  "A_{T_n}"] \ar[lld, Rightarrow, shorten <=5, shorten >=5, "A_\Gamma", swap] \\ 
\star\times \End_\kappa^X(T_n) \ar[r, "1_{\End_\kappa^X}", swap] \ar[rrr, bend right=25, "\simeq", swap] & \End_\kappa^X(T_1)\times \End_\kappa^X(T_n) \ar[rr, "\End_\kappa^X(\Gamma)",swap] &	& \End_\kappa^X(T_n) \\ 
& \ar[u, Leftarrow, "(1^{\End_\kappa^X}_\Gamma)^{-1}", swap]  &
\end{tikzcd}
= \begin{tikzcd}
\star\times \cat{O}(T_n) \ar[d, swap, "1\times A_{T_n} "] \ar[r, "\simeq"]  &  \cat{O}(T_n) \ar[ld, Rightarrow, shorten <=5, shorten >=5, "\cong", swap]  \ar[d,  "A_{T_n}"] \\ 
\star\times \End_\kappa^X(T_n) \ar[r, "\simeq", swap]  &  \End_\kappa^X(T_n) \ . 
\end{tikzcd} 
\end{equation}
Again, cyclic algebras over cyclic operads are defined analogously by replacing $\Graphs$ with $\Forests$.
\end{definition}

\begin{remark}\label{remarkphi2}
The compatibility condition in Definition~\ref{defmodalgM} with respect to operadic units might seem to
depend on a specific choice of a unit for $\mathcal{O}$ and to consist of additional structure. However, 
note that for another choice of unitality data $(1'_\cat{O},1'_\Gamma, 1'_A)$, we have
\begin{equation}
	\begin{tikzcd}[row sep=0.4cm] 
	& \mathcal{O}(T_1) \ar[dd, "A_{T_1}"]  \\ 
	\star \ar[ru, "1'_\mathcal{O}", bend left = 90, ""{name=A}] \ar[ru, "1_\mathcal{O}", ""{name=B}] \ar[rd,"1_{\End_\kappa^X}",swap,""{name=U}] & \\ 
	\  & \End_\kappa^X(T_1)
	\ar[from=1-2, to=U, shorten <=5, shorten >=5, Rightarrow, "1_A"]
	\ar[from=A, to=B, shorten <=5, shorten >=5, Rightarrow,"\cong"]
\end{tikzcd} =
	\begin{tikzcd}[row sep=0.4cm] 
	& \mathcal{O}(T_1) \ar[dd, "A_{T_1}"]  \\ 
	\star \ar[ru, "1'_\mathcal{O}"] \ar[rd,"1_{\End_\kappa^X}",swap,""{name=U}] & \\ 
	\  & \End_\kappa^X(T_1) \ , 
	\ar[from=1-2, to=U, shorten <=5, shorten >=5, Rightarrow, "1'_A"]
\end{tikzcd}
\end{equation}
where the unlabeled 2-isomorphism is the one mentioned in Remark~\ref{Rem: Uniqunes of units}. In this sense, being 
compatible with operadic units is a property of a symmetric monoidal natural transformation $A\colon \cat{O} \to \End_\kappa^X$.
\end{remark}

\begin{remark}\label{remcatdual}
	The symmetric monoidal bicategory $\cat{M}$ in Definition~\ref{defmodalgM} can be arbitrary. The case $\cat{M}=\Cat$ (with the `usual' symmetric monoidal structure from Example~\ref{exsymmoncat}) can, of course, be considered, but it is not interesting because a non-degenerate symmetric pairing $\cat{C}\times \cat{C}\to \star$ exists if and only if $\cat{C}\simeq \star$.
	The situation will be significantly more interesting in the example $\cat{M}=\Fin$, see Section~\ref{secfinpairing}.
	\end{remark}

In the approach to operads chosen in this paper, the definition of morphisms between algebras requires some care. Na\" ively, one might try to define them
as symmetric monoidal modifications. However, this does not relate algebra
structures on \emph{different} underlying objects in $\cat{M}$. 
We sketch a construction of the bicategory $\ModAlg(\mathcal{O})$ for a modular
operad $\mathcal{O}$ using the (symmetric monoidal) arrow category $\operatorname{ar}(\cat{M}) \coloneqq [ 0\to 1, \cat{M}]$, i.e.\
the  category 
of functors from the interval category $0\to 1$ with two objects $0$ and $1$ and one non-identity morphism to $\cat{M}$.  Since the bicategory of algebras will not play an essential role throughout the paper, we will limit the exposition to the very essential points and omit some details. The construction we give will just be spelled out for modular operads, but can be easily transferred to 
 cyclic and ordinary operads.

\begin{definition}
Let $(X,\kappa)$ and $(X',\kappa')$ be objects of a symmetric monoidal bicategory $\cat{M}$ equipped with 
non-degenerate symmetric pairings. 
A \emph{morphism compatible with the pairings} $f\colon (X,\kappa) \longrightarrow (X',\kappa') $ is a morphism $f\colon X \longrightarrow X'$
plus the structure required to make 
\begin{equation}\label{eqnpairinginarrow}
\begin{tikzcd}
X\otimes X \ar[r,"f\otimes f"] \ar[d,"\kappa",swap] & X'\otimes X' \ar[d, "\kappa'"] \\
I \ar[r,"\id_I", swap] & I
\end{tikzcd}
\end{equation}
into a symmetric pairing in $\operatorname{ar}(\cat{M})$ on $f$. 
\end{definition}

Now let $f\colon (X,\kappa) \longrightarrow (X',\kappa')$
be a morphism compatible with non-degenerate symmetric pairings
$\kappa$ and $\kappa'$ on $X$ and $X'$, respectively.
 Then we can associate the endomorphism operad $\End^f\colon \Graphs \longrightarrow \Cat$
 to the pairing \eqref{eqnpairinginarrow} in the arrow category. It fits into
a span
\begin{equation}
\begin{tikzcd}
 & \End^f \ar[rd] \ar[ld] & \\
 \End^X_\kappa & & \End^{X'}_{\kappa'}  
\end{tikzcd}
\end{equation}
of $\Cat$-valued modular operads.

\begin{definition}
	For a $\Cat$-valued modular operad $\mathcal{O}$,
	the bicategory $\ModAlg(\mathcal{O})$ of modular $\mathcal{O}$-algebras in a symmetric monoidal bicategory $\cat{M}$ is defined as follows:
	Objects are modular $\mathcal{O}$-algebras in $\cat{M}$ in the sense of Definition~\ref{defmodalgM}.
A \emph{1-morphism of modular algebras} $f\colon A \longrightarrow B$ consists of a morphism $f\colon (A,\kappa_A) \longrightarrow (B,\kappa_B)$ together with a symmetric monoidal transformation $\mathcal{O}\to\End^f$ and a filling of the diagram
\begin{equation}
\begin{tikzcd}
&& \mathcal{O} \ar[lddd, "A" , swap] \ar[rddd,"B"]  \ar[dd, dotted]& & & \\ \\ 
 & & \End^f \ar[rd] \ar[ld] & \\
 & \End^A_{\kappa_A} & & \End^{B}_{\kappa_B} 
\end{tikzcd}
\end{equation} 
with natural symmetric monoidal isomorphisms.
We define 2-morphisms of modular algebras 
similarly using spans of spans of modular 
operads constructed from pairing preserving 2-morphisms in $\cat{M}$ and 
the endomorphism operad in $\operatorname{ar}(\operatorname{ar}(\cat{M}))$.  
\end{definition}

Morphisms between algebras over cyclic and modular operads are always equivalences:
\begin{proposition}
Let $\mathcal{O}$ be a modular operad. The bicategory $\ModAlg \mathcal{O}$ is a 2-groupoid. Similarly, the bicategory of 
cyclic algebras over a cyclic operad is a 2-groupoid. 
\end{proposition}
\begin{proof}
We  give the proof for the modular case. The cyclic case works completely analogously.  
Let $f\colon (A,\kappa_A) \to (B,\kappa_B) $ be a morphism of modular $\mathcal{O}$-algebras. We define an inverse to $f$ by 
\begin{align}
\input{Inverse_cyclic_morphism.tikz} 
\end{align}
It is straightforward to verify that this defines a weak inverse for $f$ in $\cat{M}$
 using the snake isomorphisms and the 2-isomorphisms $ \kappa_B \circ 
(f\otimes f) \cong \kappa_A $ and $(f\otimes f) \circ \Delta_A \cong \Delta_B $. These isomorphisms exist by definition because $f$ 
is compatible with the pairings. The inverse can be equipped with the additional structure needed to make it into a 
morphism of modular algebras. This shows that every 1-morphism is weakly invertible. 

Similarly, for a 2-morphism $\eta \colon f\to g$ between 1-morphisms $f,g \colon (A,\kappa_A) \to (B,\kappa_B) $ of modular algebras, we can construct an inverse as follows:
\begin{align}
	\input{Inverse_cyclic_2_morphism.tikz} 
\end{align}
\end{proof}

We call a morphism of modular operads an \emph{equivalence} if it is an equivalence arity-wise (the same definition is made for cyclic and ordinary operads).
It is not clear that equivalent modular operads give rise to equivalent categories of  modular algebras.
For ordinary operads, such results are proven in \cite[Theorem~4.1]{bm} under the name \emph{Comparison Theorem}.
We will establish such a result for modular algebras over modular operads with values in bicategories. 
To this end, we will need the following, probably well-known result:

\begin{proposition}[Whitehead's Theorem for symmetric monoidal functors between symmetric monoidal bicategories]\label{propwhitehead}
	Let $F , G : \cat{C} \to \cat{D}$ be symmetric monoidal functors between symmetric monoidal bicategories.
	Then every symmetric monoidal equivalence $\alpha$ from $F$ to $G$ has a weak inverse, i.e.\
	there is a symmetric monoidal transformation $\beta$ from $G$ to $F$ such that $\alpha \beta$ and $\beta \alpha$ are the identity transformation of $G$ and $F$, respectively, up to invertible monoidal modification. 
\end{proposition}

In lack of a reference, we at least sketch the proof:

\begin{proof}[\textsl{Sketch of the proof}] 
	The symmetric monoidal equivalence $\alpha$ consists of components $\alpha_c : F(c)\to G(c)$ for $c\in \cat{C}$ plus coherence data:
	\begin{enumerate}[label=(\roman*)]
	\item An invertible 2-morphism $\alpha_f \colon G(f)\alpha_c\to \alpha_{c'}F(f)$ for all 1-morphisms $f:c\to c'$ in $\cat{C}$.\label{cohtypei}
	\item A 2-isomorphism \label{cohtypeii}
		\begin{equation}
		\begin{tikzcd}
		& F(I_\cat{C}) \ar[dd, "\alpha_{I_\cat{C}}"] \\
		I_\cat{D}\ar[ru,"\simeq"] \ar[rd,swap,"\simeq"] \ar[r, Leftarrow, "\alpha_{I_\cat{D}}"] & \ \\ 
		& G(I_\cat{C}) \ , 
	\end{tikzcd}
		\end{equation}
	where $I_\cat{C}$ and $I_\cat{D}$ is the monoidal unit of $\cat{C}$ and $\cat{D}$, respectively, and the unlabeled 1-equivalences are part of the coherence data for $F$ and $G$.
	\item 2-Isomorphisms 
		\begin{equation}
		\begin{tikzcd}
			& F(c\otimes_\cat{C} c')  \ar[rd, "\alpha_{c\otimes_\cat{C} c'}"] &  \\
		F(c)\otimes_\cat{D} F(c') \ar[d, " \alpha_c \otimes_\cat{D} \id_{F(c')} ", swap] \ar[ru,"\simeq"] & & G(c\otimes_\cat{C} c') \\ 
		G(c)\otimes_\cat{D} F(c') \ar[rr, "\id_{G(c)}\otimes_\cat{D} \alpha_{c'} ",swap] & \ar[uu, Rightarrow, "{\alpha_{c,c'}}", shorten <=5, shorten >=5] & G(c)\otimes_\cat{D} G(c')  \ar[u,swap,"\simeq"]
		\end{tikzcd}
	\end{equation}
	for all objects $c,c' \in \cat{C}$. The unlabeled arrows are again the coherence morphisms for $F$ and $G$.\label{cohtypeiii}
\end{enumerate}
This data is subject to the coherence conditions that are given in~\cite[Definition 2.7]{schommerpries}.

By assumption $\alpha_c$ is an equivalence for every $c\in\cat{C}$. Hence, we find a weak inverse 
$\beta_c : G(c)\to F(c)$ that we may choose to be additionally adjoint to $\alpha$. The unit and counit 
of this adjunction for varying $c\in \cat{C}$ can be used to equip $\beta$ with the necessary coherence 
data (of the type \ref{cohtypei}-\ref{cohtypeiii} as above for $\alpha$): 
\begin{enumerate}[label=(\roman*)]
	\item For a morphism $f\colon c \to c'$, the needed invertible 2-morphism $\beta_f:F(f)\beta_c \to \beta_{c'}G(f)$ is given by 
\begin{equation}
\begin{tikzcd}
G(c) \ar[rr, "\beta_c"] \ar[ddd, "G(f)", swap] \ar[rd, "\id", swap] & & F(c) \ar[ld, "\alpha_c"] \ar[d, "F(f)"] \ar[ddd, "F(f)", bend left=60] \ar[dll, Rightarrow,   shorten <=15, shorten >=50] \\
\ & \ar[ld, Rightarrow,   shorten <=10, shorten >=10] G(c) \ar[d,"G(f)", swap] & \ar[l, Rightarrow, "\alpha_f^{-1}"] F(c') \ar[dd, "\id_{F(c)}"]  \ar[ld, "\alpha_{c'}"]    & \ar[ld, Rightarrow,   shorten <=15, shorten >=15]              \\
\ & G(c') \ar[d, Rightarrow,   shorten <=5, shorten >=9] \ar[rd, "\beta_{c'}"]&     \    \ar[l, Rightarrow,   shorten <=5, shorten >=5]            \\    
G(c') \ar[ru, "\id"] \ar[rr,swap, "\beta_{c'}"] & \ & F(c') \ , 
\end{tikzcd}
\end{equation} 
where the unlabeled 2-morphisms are
all 2-isomorphisms and arise as
 coherence data
or through
the unit and counit of the pair $(\alpha_{I_\cat{C}} , \beta_{I_\cat{D}})$
of adjoint equivalences
(all 2-morphisms in this proof will be invertible, but we will omit the symbol `$\cong$' in the diagrams to save space).
	\item The coherence  data 
	of type~\ref{cohtypeii}
	is given by
	\begin{equation}
		\begin{tikzcd}
		& G(I_\cat{C}) \ar[rdd, "\beta_{1_\cat{C}}"] & \ar[ldd, Rightarrow,   shorten <=45, shorten >=5] \\
		I_\cat{D}\ar[ru,"\simeq"] \ar[rrd, bend right =90,swap,"\simeq"] \ar[rd,"\simeq"]  \ar[r, Leftarrow, "\alpha_{I_\cat{D}}"] & \  \\ 
		& F(I_\cat{C}) \ar[r, "\id"] \ar[uu, "\alpha_{I_\cat{C}}", swap] \ar[r]  & F(I_\cat{C}) \ ,  \\
		& \ar[u, Leftarrow,   shorten <=2, shorten >=2] &
	\end{tikzcd}
	\end{equation}
	where the unlabeled 1-equivalences are part of the coherence data for $F$ and $G$ and the
	unlabeled 2-isomorphisms are the counit of the adjoint equivalence and the coherence
	isomorphisms $f\circ \id\cong f$ for the composition with identities in the bicategory $\cat{C}$.

	\item The 2-isomorphisms of type~\ref{cohtypeiii} are given by (we suppress tensoring with identities in the notation for 1-morphisms and 2-morphisms to ensure readability)
	\begin{equation}
		\begin{tikzcd} 
		& \ar[dd, Leftarrow,   shorten <=25, shorten >=5] & G(c\otimes_\cat{C} c')\ar[rrd, "\beta_{c\otimes c'}"] & & \\
		 G(c) \otimes_\cat{D} G(c') \ar[rru,"\simeq"] \ar[rd, "\id"] \ar[ddd, "\beta_{c} ", swap] \ar[ddr, " \beta_{c'}", swap, bend right=20] & &\ & & F(c\otimes_\cat{C} c') \ar[ll, Rightarrow,   shorten <=5, shorten >=95] \\
		  \ &  G(c) \otimes_\cat{D} G(c') \ar[l, Leftarrow,   shorten <=5, shorten >=25] \ar[uur,"\simeq"] &\ & \ & F(c\otimes_\cat{C} c') \ar[lluu, "\alpha_{c\otimes c'}"] \ar[u, "\id"] \ar[lll, Rightarrow, "\alpha_{c,c'}^{-1}", swap,  shorten <=10, shorten >=10] \\ 
		 \ & G(c) \otimes_\cat{D} F(c') \ar[ru, Leftarrow,   shorten <=5, shorten >=30] \ar[ld, Leftarrow,   shorten <=5, shorten >=5] \ar[rr, "\id"] \ar[u, "\alpha_{c'}"] \ar[rrrd, "\beta_{c}", swap] & \ & G(c) \otimes_\cat{D} F(c') \ar[llu, "\alpha_{c'}", swap] & \ & \ar[l, Rightarrow,   shorten <=30, shorten >=2] \\
		  F(c) \otimes_\cat{D} G(c') \ar[rrrr, "\beta_{c'}",swap] & & & & F(c)\otimes_\cat{D} F(c') \ ,  \ar[llu, Rightarrow,   shorten <=35, shorten >=35] \ar[lu, "\alpha_c", swap] \ar[uu,"\simeq"] \ar[uuu, bend right=60]
		\end{tikzcd}
	\end{equation}
	where the unlabeled 2-isomorphisms are coherence data of $F$ and $G$
	or the unit and counit of the adjoint equivalences $\alpha$ and $\beta$ evaluated at various objects.
\end{enumerate}
Now we consider the composition $ \beta \circ \alpha$. The adjunction data combines into an invertible 
symmetric monoidal modification (see~\cite[Definition 2.8]{schommerpries} for a definition)
 $\beta\circ \alpha \to \id_F$, i.e.\
coherent isomorphisms $\beta_c \circ \alpha_c \to \id_c $ for all $c\in \cat{C}$. Similarly, we get 
an invertible symmetric monoidal modification $\id_G \to \alpha \circ \beta$. 
\end{proof}

If we are given  a map $\Phi : \mathcal{O} \to \P$ of $\Cat$-valued modular operads, we obtain a functor $\Phi^*:\ModAlg \P\to\ModAlg \mathcal{O}$ for the category of modular algebras with values in a symmetric monoidal bicategory $\cat{M}$, namely by precomposition with $\Phi$.
If $\Phi$ is an equivalence, we may find 
thanks to Proposition~\ref{propwhitehead}
a weak inverse $\Psi : \P\to \mathcal{O}$. This gives us a functor $\Psi^* : \ModAlg \mathcal{O}\to\ModAlg \P$ --- again by precomposition --- that can easily be seen to be a weak inverse for $\Phi^*$. This leads to the desired \emph{Comparison Theorem} for modular and cyclic operads (in a bicategorical context):

\begin{theorem}[Comparison Theorem]\label{comparisonthm}
	Any equivalence $\Phi : \mathcal{O} \to \P$ of modular $\Cat$-valued operads 
	induces by precomposition an equivalence $\Phi^*:\ModAlg \P \to \ModAlg \mathcal{O}$ between the 2-groupoids of modular algebras with values in any symmetric monoidal bicategory $\cat{M}$. An analogous statement holds for cyclic operads and ordinary operads.
\end{theorem}

\subsection{Non-degenerate pairings in $\Fin$\label{secfinpairing}}
In Section~\ref{seccycmodalg} we have defined
cyclic and modular algebras over a $\Cat$-valued operad.
The algebras take values in an arbitrary symmetric monoidal bicategory.
When characterizing associative and framed little disks algebras, we will also allow general symmetric monoidal bicategories for the algebras to take values in.
Only afterwards, we will specialize to a specific example of a symmetric monoidal bicategory that allows us to study applications of our results in quantum algebra, namely the symmetric monoidal bicategory $\Fin$ 
formed by finite categories, left exact functors and natural transformations (Example~\ref{exsymmoncat}). In this situation, there is an intimate relation between non-degenerate symmetric pairings and the morphism spaces that we will exploit.

A key tool for the investigation of pairings on finite categories will be the \emph{categorical Eilenberg-Watts Theorem} stated in terms of coends (we refer to \cite{fss} for an account on coends in finite categories). For the formulation of the result,
we use that $\Fin$ is enriched over itself; we denote the internal hom by $\Fin[-,-]$.

	\begin{theorem}[Fuchs-Schaumann-Schweigert \text{\cite[Theorem~3.2]{fss}}]\label{thmfromfss}
		For  finite categories $\cat{C}$ and $\cat{D}$, the functors
		\begin{align}
		\Psi : \Fin[\cat{C},\cat{D}] &\to \cat{C}^\opp \boxtimes \cat{D} \\
		F&\mapsto \int^{X \in \cat{C}^\opp} X \boxtimes F(X) \ , \\
		\Phi :  \cat{C}^\opp \boxtimes \cat{D} &\to \Fin[\cat{C},\cat{D}] \\
		X \boxtimes Y &\mapsto \cat{C}(X,-)\otimes Y \ \ ,
		\end{align}
		where we denote by a slight abuse of notation the $\Vect$-tensoring of $\cat{C}$ by $\otimes$,
		provide a pair of adjoint equivalences
		\begin{flalign}\label{quillenequiv}
		\xymatrix{
			\Psi \,:\, \Fin[\cat{C},\cat{D}] ~\ar@<0.8ex>[r]_-{\sim} & \ar@<0.8ex>[l]~ \cat{C}^\opp \boxtimes \cat{D} \,:\, \Phi \ . 
		}
		\end{flalign}
	\end{theorem}

	\begin{definition}\label{defkappaopp}
		For any pairing $\kappa$ on a finite category $\cat{C}\in\Fin$, i.e.\ a left exact functor $\kappa :\cat{C}\boxtimes\cat{C}\to\FinVect$,
		 we define the functor $-^\kappa : \cat{C} \to \cat{C}^\opp$ as the composition
		\begin{align}
		-^\kappa : \cat{C} \xrightarrow{ \   X \mapsto \kappa (X,-)    \ } \widehat{\cat{C}}:=\Fin[\cat{C},\FinVect] \xrightarrow{ \  \Psi   \ }    \cat{C}^\opp \ ,
		\end{align} i.e.\ we set
		\begin{align}
		X^\kappa := \int^{Y \in \cat{C}^\opp} Y \otimes \kappa(X,Y) \quad \text{for}\quad X \in \cat{C}\ . 
		\end{align}
	\end{definition}

		By Theorem~\ref{thmfromfss} there is a canonical natural isomorphism
	$ \cat{C}(X^\kappa , -)= \Phi \Psi (\kappa(X,-))\cong \kappa(X,-)$,
	namely the inverse of the unit of the adjunction $\Psi \dashv \Phi$. This implies directly the following basic, but important fact:

	\begin{lemma}\label{lemmakappastar}
		Let $\kappa$ be a pairing on a finite category $\cat{C} \in \Fin$. Then there is a canonical natural isomorphism $\Lambda_{X,Y}: \cat{C}(X^\kappa , Y) \to \kappa(X,Y)$ for $X,Y\in\cat{C}$
	\end{lemma}

	\begin{remark}\label{remsympairinghom}
		Recall from Definition~\ref{defsymmetricstructure}
		that a pairing $\kappa \colon \cat{C} \boxtimes \cat{C} \longrightarrow \FinVect $ is symmetric if it is equipped with a homotopy fixed point structure with 
	respect to the natural $\Z_2$-action, i.e.\ a natural isomorphism $\Sigma_{X,Y} \colon \kappa(X,Y) \longrightarrow \kappa(Y,X)$ squaring to 
	the identity. From Lemma~\ref{lemmakappastar}, we now get a natural isomorphism 
	\begin{align}
	\cat{C}(X^\kappa,Y) \ra{\Lambda_{X,Y}} \kappa(X,Y) \ra{\Sigma_{X,Y}}\kappa(Y,X)\ra{\Lambda_{Y,X}^{-1}} \cat{C}(Y^\kappa,X) 
	\ra{-^\kappa} \cat{C}(X^\kappa,Y^{2\kappa}) \ ,  \ \text{where} \ Y^{2\kappa}:= \left(Y^\kappa \right)^\kappa \ .
	\end{align}
Provided that $-^\kappa$ is an equivalence, the 
 Yoneda Lemma leads to a natural isomorphism $ -^{2\kappa}
	\cong \id_{\cat{C}}$. 
	\end{remark}

	\begin{remark}\label{remsnakecoev}
According to Definition~\ref{defpairing}, the fact that
	the pairing $\kappa$ exhibits $\cat{C}$ as its own dual means that there is a functor $\Delta: \FinVect \to \cat{C}\boxtimes\cat{C}$, called coevaluation, such that $\kappa$ and $\Delta$ satisfy the usual snake relations up to natural isomorphism.	
	Since we are working in $\Fin$, the functor $\Delta$ is determined by its value on the ground field $k$
	that we also denote by $\Delta \in \cat{C}\boxtimes \cat{C}$ and will refer to as the \emph{coevaluation object}.
	It will be convenient to write $\Delta = \Delta' \boxtimes \Delta''$. This notation is inspired by the Sweedler notation in the theory of Hopf algebras \cite[Notation~1.6]{kassel} and should not be understood in the sense that $\Delta$ is actually a `pure tensor'. The $\Delta'$ and the $\Delta''$ are merely placeholders for the different factors of $\cat{C}$. These may help to write the snake isomorphisms. For example, the composition
	\begin{align}
	\cat{C}\boxtimes \cat{C}  \xrightarrow{ \ \Delta \boxtimes \id_\cat{C}\boxtimes \id_\cat{C}  \  } \cat{C}^{\boxtimes 4} \xrightarrow{\   \id_\cat{C} \boxtimes \kappa \boxtimes \id_\cat{C}     \ } \cat{C}\boxtimes \cat{C}\xrightarrow{\ \kappa \ } \FinVect 
	\end{align} is naturally isomorphic to $\kappa$ by a snake isomorphism.  The component of this natural isomorphism at $X\boxtimes Y \in \cat{C}\boxtimes \cat{C}$ can now be written as
	\begin{align}
	\kappa(\Delta',Y)\otimes \kappa(\Delta'',X)\cong \kappa(X,Y) \ . 
	\end{align}
\end{remark}
	
	\begin{proposition}\label{propkappaisequiv}
		A pairing $\kappa$ on a finite category $\cat{C}\in\Fin$ is non-degenerate if and only if $-^\kappa :\cat{C}\to \cat{C}^\opp$ from Definition~\ref{defkappaopp}
		is an equivalence.
		In that case, the coevaluation object is given by the 
		coend $\Delta = \int^{ X \in \cat{C}   }   X \boxtimes X^{-\kappa} \in \cat{C}\boxtimes \cat{C}$, where $X^{-\kappa}$ is the image of $X$ under the weak inverse of $-^\kappa$.
		If $\kappa$ is symmetric, then $-^{-\kappa}\cong -^\kappa$ by a canonical isomorphism, so that the coevaluation object is canonically isomorphic to the coend $\int^{X\in\cat{C}} X \boxtimes X^\kappa\in\cat{C}\boxtimes\cat{C}$. 
	\end{proposition}
	
	\begin{proof} 
		Suppose $\kappa$ is non-degenerate and denote by $\Delta$ the coevaluation object. 
		Sending $\alpha \in \widehat{\cat{C}}:=\Fin[\cat{C},\FinVect]$ to $(\alpha \boxtimes \id_\cat{C})      (\Delta)$ yields a weak inverse for the functor $\cat{C}\to\widehat{\cat{C}}$ sending $X$ to $\kappa(X,-)$. By definition this implies that $-^\kappa$ is also an equivalence.
		
		Conversely, let 
		$-^\kappa : \cat{C} \to \cat{C}^\opp    $
		be an equivalence. 
		Since by \cite[Section~3.7]{fss} $\cat{C}^\opp$ is dual to $\cat{C}$ with duality pairing
		$\cat{C}(-,-):\cat{C}^\opp \boxtimes \cat{C} \to \FinVect$, we see that $\cat{C}$ is self-dual with duality pairing
		\begin{align}
		\cat{C} \boxtimes \cat{C} \xrightarrow{\       -^\kappa \boxtimes  \id_{\cat{C}}  \    }     \cat{C}^\opp \boxtimes \cat{C} \xrightarrow{\ \cat{C}(-,-)    \    } \FinVect\ . 
		\end{align}
		This composition is canonically isomorphic to $\kappa$ by Lemma~\ref{lemmakappastar}. 
		As we can also extract from \cite[Section~3.7]{fss}, the coevaluation object for the duality of $\cat{C}$ and $\cat{C}^\opp$ is the coend $\int^{ X \in \cat{C}   }   X \boxtimes X \in \cat{C} \boxtimes \cat{C}^\opp $. Hence, $\int^{ X \in \cat{C}   }   X \boxtimes X^{-\kappa}$ is the coevaluation object for the self-duality of $\cat{C}$. 
		
		The additional statement
		on the symmetric case
		follows from Remark~\ref{remsympairinghom}.
	\end{proof}

In order to prove excision results later on in Section~\ref{calculusconstruction}, 
we need to relate the composition in the endomorphism operad to left exact coends studied in \cite{lyulex}, see also the treatment in \cite{fss} and additionally \cite{dva} for the relation to homotopy coends and derived traces.
The notion of a left exact coend is relevant in the following situation: If we are given a left exact functor $G:\cat{C}\boxtimes\cat{C}^\opp\boxtimes\cat{D}\to\cat{A}$, where $\cat{A},\cat{C}$ and $\cat{D}$ are finite categories, the functor $\cat{D}\ni Y \mapsto \int^{X\in\cat{C}} F(X\boxtimes X\boxtimes Y)$ will be linear, but not necessarily left exact. Hence, it does not belong to $\Fin$.
As a remedy, we see $G$ as a left exact functor $\cat{C}\boxtimes\cat{C}^\opp\to \Fin[\cat{D},\cat{A}]$. The coend of this functor exists; it will by construction give us a left exact functor $\cat{D}\to\cat{A}$ that one refers to as left exact coend and denotes by $\oint^{X\in\cat{C}}G(X\boxtimes X\boxtimes -)$. 
To make the connection to the endomorphism operad, we first establish a relation between the left exact coend and dualizability in $\Fin$ by proving that the left exact coend can be described through evaluation on the coevaluation object:

\begin{lemma}\label{lemmaleftexactcoend}
	Let $\cat{C}\in \Fin$ have a  non-degenerate symmetric pairing $\kappa : \cat{C}\boxtimes \cat{C}\to\FinVect$.
	Then for any finite category $\cat{D}$ and any left exact functor $F:\cat{C}\boxtimes\cat{C}\boxtimes\cat{D}\to\FinVect$, there is a canonical isomorphism
	\begin{align}
	\oint^{X\in\cat{C}} F(X\boxtimes X^{\kappa}\boxtimes -) \cong F\left(     \left(   \int^{X\in\cat{C}} X \boxtimes X^{\kappa} \right)  \boxtimes -      \right) \ , 
	\end{align}
	where $-^\kappa : \cat{C}^\opp\to\cat{C}$ is the equivalence induced by $\kappa$. 
	\end{lemma}

\begin{proof}
	By \cite[Proposition~1.7 \& Corollary~1.10]{dss} $F$ is representable, i.e.\ there is an object $L\in \cat{C}\boxtimes\cat{C}\boxtimes\cat{D}$ such that $F$ can be written as the hom functor $(\cat{C}\boxtimes\cat{C}\boxtimes\cat{D})(L,-)$. Using for $Y\in\cat{C}\boxtimes \cat{C}$ the contraction
	\begin{align}
	\langle L,Y\rangle := (\cat{C}\boxtimes\cat{C})(L',Y)\otimes L''\in\cat{D} \quad \text{with Sweedler notation}\quad L=L'\boxtimes L'' \in (\cat{C}\boxtimes\cat{C})\boxtimes \cat{D}
	\end{align} of $L$ and $Y$ via the morphism spaces, we find \begin{align} F(X\boxtimes X^{\kappa}\boxtimes-)&=\cat{D}(\langle L,X\boxtimes X^{\kappa} \rangle,-) \ , \\ F\left(     \left(   \int^{X\in\cat{C}} X \boxtimes X^{\kappa} \right)  \boxtimes -      \right)&=\cat{D} \left(     \left\langle L,\int^{X\in\cat{C}} X \boxtimes X^{\kappa}\right\rangle ,-    \right) \ . \label{eqnthesecondfunctor} \end{align}
	It remains to prove that the dinatural family
	\begin{align}
	\cat{D}(\langle L,X\boxtimes X^{\kappa}\rangle,-) \to \cat{D} \left(     \left\langle L,\int^{X\in\cat{C}} X \boxtimes X^{\kappa}\right\rangle ,-    \right)
	\end{align} induced by the dinatural family $X\boxtimes X^{\kappa}\to \int^{X\in\cat{C}} X \boxtimes X^{\kappa}$ is universal because this exhibits \eqref{eqnthesecondfunctor} as the coend $\oint^{X\in\cat{C}} F(X\boxtimes X^{\kappa}\boxtimes -)$ and hence proves the assertion. For this, we need to show that any dinatural family $	\cat{D}(\langle L,X\boxtimes X^{\kappa}\rangle,-)\to G$ for some left exact functor $G:\cat{D}\to\FinVect$ descends uniquely to $\cat{D} \left(     \left\langle L,\int^{X\in\cat{C}} X \boxtimes X^{\kappa}\right\rangle ,-    \right)$. But by invoking again the representability statement from \cite[Proposition~1.7 \& Corollary~1.10]{dss}, we may write $G=\cat{D}(M,-)$ for some $M\in\cat{D}$, which by the Yoneda Lemma implies that it suffices to prove that the dinatural family
	\begin{align}
	\langle L,X\boxtimes X^{\kappa}\rangle\to \left\langle L,\int^{X\in\cat{C}} X \boxtimes X^{\kappa}\right\rangle
	\end{align} is universal or, in other words, that there is a canonical isomorphism
	\begin{align}
\int^{X\in\cat{C}}	\langle L,X\boxtimes X^{\kappa}\rangle\cong \left\langle L,\int^{X\in\cat{C}} X \boxtimes X^{\kappa}\right\rangle \ , 
	\end{align}
	which in fact exists by \cite[Proposition~3.4]{fss}. 
	\end{proof}

As a consequence of this Lemma, we find that the composition operation in the endomorphism operad corresponding to a graph is a left exact coend with one dummy variable for each internal edge of the graph:

\begin{proposition}\label{propointcoend}
	Let $\Gamma : T\to T'$ be a morphism in $\Graphs$ between objects $T$ and $T'$.
	We assume that $T'$ is connected and denote by $T=\sqcup_{i=1}^m T^{(i)}$ the decomposition of $T$ into connected components.
	Then for any $\cat{C}\in\Fin$ with non-degenerate symmetric pairing $\kappa : \cat{C}\boxtimes\cat{C}\to\FinVect$
	and $F=(F_i)_{1\le i\le m} \in\prod_{i=1}^m \Fin(\cat{C}^{\boxtimes \Legs(T^{(i)})},\Vect)$ in $\Fin$, we have a canonical isomorphism
	\begin{align}
	\End_\kappa^\cat{C}(\Gamma)F\cong \oint^{X_1,\dots,X_r\in\cat{C}} F^\boxtimes(\dots,X_j,\dots,X_j^{\kappa},\dots) \quad \text{with}\quad F^\boxtimes = F_1 \boxtimes \dots \boxtimes F_m:\cat{C}^{\boxtimes	\Legs(T)}\to\Vect \ . 
	\end{align} 
	The left exact coends $\oint$ runs over the variables $X_1,\dots,X_r$ corresponding to internal edges of $\Gamma$.
	(The integrand of the coend is just a mnemonic notation for the insertion of the variables corresponding to internal edges into the correct arguments of $F^\boxtimes$ according to the identification of $\nu(\Gamma)$ with $T$.)
	\end{proposition}

\begin{proof}
	Since we can decompose $\Gamma$ into morphisms contracting one internal edge at a time, we may assume that $\Gamma$ has just one internal edge.
	Then
	\begin{align}
	\End_\kappa^\cat{C}(\Gamma)F\cong F^\boxtimes(\dots,\Delta',\dots,\Delta'',\dots)
	\end{align}
	with the Sweedler notation for the coevaluation object $\Delta \in \cat{C}\boxtimes\cat{C}$ 
	discussed in Remark~\ref{remsnakecoev}.
	This coevaluation object is given by $\Delta = \int^{X\in\cat{C}}   X \boxtimes X^{\kappa}$ by Proposition~\ref{propkappaisequiv}.
	Now the assertion follows from Lemma~\ref{lemmaleftexactcoend}.
	\end{proof}

\section{The Lifting Theorem}
Pulling back a cyclic operad $\mathcal{O}$ along the functor $\RForests \longrightarrow \Forests$ yields
an ordinary operad $\overline{\mathcal{O}}$, the underlying ordinary operad of $\cat{O}$. It is a natural question whether the structure of an algebra  over $\overline{\mathcal{O}}$ on some object $X$
(in some (higher) symmetric monoidal category --- depending on the context that one is working in)
 can be lifted
to a \emph{cyclic} algebra over $\mathcal{O}$ on $(X,\kappa)$. 
More precisely: What kind of additional data and/or properties are needed?

When the target category is given by the category of vector spaces, there is a very classical answer \cite{gk,markl}: The structure of a cyclic $\mathcal{O}$-algebra on a vector space $V$ with a pairing $\kappa$ is equivalent to the structure of an $\overline{\mathcal{O}}$-algebra on $V$ such that the pairing satisfies an invariance property. 

In this section we provide an answer in the context of symmetric monoidal bicategories.
In this framework, we have to take further coherence data into account making the situation slightly more subtle.  Still the principle that an $\mathcal{O}$-algebra is an $\overline{\mathcal{O}}$-algebra plus an invariant pairing can be generalized to this setting. 
However, the invariance of the pairing will amount to additional \emph{structure} and will not be just a property. More precisely, for each operation, there will be a cyclic invariance isomorphism.
The isomorphisms for different operations will generally not be independent, but related in a way prescribed by the operad. This leads to a significantly richer algebraic structure. 	

	\subsection{General version}
In order to state the Lifting Theorem,
 we need to make some preliminary observations:
Consider the corollas  $T_n$ (these have $H=\{0,\dots,n\}$ as the set of edges and one vertex) 
and their disjoint unions. Whenever needed, we may regard $T_n$ as a rooted corolla with $0$ as the root. 
Recall from Section~\ref{secrecallcos} that  	
the cyclic permutation of legs induces an isomorphism $\tau_n : T_n \cong T_n$ in $\Forests$ with $\tau_n^{n+1}=\id_{T_n}$.

\begin{figure}[h]
	\begin{center}
		\begin{overpic}[scale=2]
			{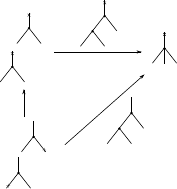}
			\put(34,85){$\Gamma'=$}
			\put(-8,42){$ \tau_{2} \sqcup \tau_{2}^2$}
			\put(50,30){$\Gamma=$}
		\end{overpic}
	\end{center}
	\caption{A sketch for the factorization of a morphism $\Gamma\in \Forests$ into a cyclic permutation
		and $\Gamma'\in \RForests$ The roots are marked by $\times$.}
	\label{Fig:Decomposition}
\end{figure}   

The graph underlying a morphism $\Gamma\colon \gamma_1 = \sqcup_{i} T_{n_i}\longrightarrow \sqcup_{j} T_{n_j}=\gamma_2$ in $\Forests$ between finite disjoint unions of corollas can be 
equipped with the structure of a rooted forest using the roots for the $T_{n_j}$ (by convention the root for any of the $T_{n_j}$ is the zeroth leg). Hence, we can interpret
it as a morphism $\Gamma'\colon \gamma_1\to \gamma_2$ in $\RForests$, where the identification 
$\varphi_1'\colon \gamma_1 \longrightarrow \nu(\Gamma')$ differs from the original one $\varphi_1: \gamma_1 \to \nu(\Gamma)$ by a cyclic 
permutation of the corollas. This allows us to write $\Gamma = \Gamma' \circ \sqcup_{i} \tau_{n_i}^{k_i} $ 
with unique $k_i\in \Z_{n_i}$, see Figure~\ref{Fig:Decomposition} for an illustration. 
We will refer to this factorization of $\Gamma$ as the \emph{standard factorization} and will denote it (as above) by $\Gamma'$.

Now let $\Omega\colon \sqcup_{i}T_{n_i}\to T_{n}$ be a morphism in $\RForests$.
Then $\tau_n \circ \Omega$ can be seen as morphism in $\Forests$, and hence, we can consider its standard factorization \begin{align}
\tau_n \circ \Omega =  \Omega^{(n)} \circ \sqcup_{i} \tau_{n_i}^{k_i}\ , \quad \text{where}\quad \Omega^{(n)} := (\tau_n \circ \Omega)' \ .   \label{eqnfactorization} \end{align} 
With this notation,
we can now formulate the coherent version of the principle that a cyclic algebra over a cyclic operad is an non-cyclic algebra over the underlying operad plus an invariant pairing.

\begin{theorem}[Lifting Theorem]\label{thmlifting}
	Let $\cat{M}$ be a symmetric monoidal bicategory, $\kappa$  a non-degenerate symmetric  pairing on $X\in\cat{M}$
	and $\mathcal{O}$ a $\Cat$-valued operad. Then the structure of an $\cat{M}$-valued cyclic algebra over $\mathcal{O}$ on $(X,\kappa)$ can equivalently be described as an algebra $A$ over the underlying non-cyclic operad $\overline{\mathcal{O}}$ together with natural isomorphisms $\phi_n$
	\begin{equation}\label{isophineqn0}
	\begin{tikzcd}
	\mathcal{O}(T_n) \ar{rr}{\mathcal{O}(\tau_n) } \ar[swap]{dd}{A_{T_n}} & & \mathcal{O}(T_n)  \ar{dd}{A_{T_n}} \ar[Rightarrow, shorten >= 10, shorten <= 10]{ddll}{\phi_n} \\
	& & \\
	\End_\kappa ^X(T_n) \ar[swap]{rr}{\End_\kappa^X(\tau_n) } & & \End_\kappa ^X(T_n)
	\end{tikzcd} 
	\end{equation} 
	for $n \ge 1$
	subject to the following coherence conditions:
	
	\begin{enumerate}
		
		\item[{\normalfont (C1)}]
		
		The $n+1$-fold composition of $\phi_n$ is the identity, i.e.\
		we have the following equality of natural isomorphisms:
		\begin{equation}
		\begin{tikzcd}
		& & & \ar[Rightarrow, shorten >= 10, shorten <= -10]{d} & & & \\
		\mathcal{O}(T_n) \ar[bend left=30]{rrrrrr}{\mathcal{O}(\tau_n^{n+1}=\id_{T_n})} \ar{rr}{\mathcal{O}(\tau_n) } \ar[swap]{dd}{A_{T_n}} & & \mathcal{O}(T_n)  \ar{dd}{A_{T_n}} \ar[Rightarrow, shorten >= 10, shorten <= 10]{ddll}{\phi_n} & \dots & \mathcal{O}(T_n) \ar{rr}{\mathcal{O}(\tau_n) } \ar[swap]{dd}{A_{T_n}} & & \mathcal{O}(T_n)  \ar{dd}{A_{T_n}} \ar[Rightarrow, shorten >= 10, shorten <= 10]{ddll}{\phi_n} \\
		& & & & \\
		\End_\kappa ^X(T_n) \ar{rr}{\End_\kappa^X(\tau_n) } \ar[bend right=30, swap]{rrrrrr}{\End_\kappa^X(\tau_n^{n+1}=\id_{T_n})} & & \End_\kappa ^X(T_n) & \dots \ar[Rightarrow, shorten >= -10, shorten <= 10]{d} & \End_\kappa ^X(T_n) \ar{rr}{\End_\kappa^X(\tau_n) } &  & \End_\kappa ^X(T_n) \\
		& & & \ \  & & & 
		\end{tikzcd} 
		\end{equation}
		\begin{equation}
		=
		\begin{tikzcd}
		\mathcal{O}(T_n) \ar{rr}{\mathcal{O}(\id_{T_n}) } \ar[swap]{dd}{A_{T_n}} & & \mathcal{O}(T_n)  \ar{dd}{A_{T_n}} \ar[Rightarrow, shorten >= 10, shorten <= 10]{ddll}{A_{\id_{T_n}}} \\
		& & \\
		\End_\kappa ^X(T_n) \ar[swap]{rr}{\End_\kappa^X(\id_{T_n}) } & & \End_\kappa ^X(T_n) \ . 
		\end{tikzcd} 
		\end{equation} \label{itemC1}

		\item[{\normalfont (C2)}] 
		The isomorphisms $\phi$ intertwine with the $\mathcal{O}$-action, i.e.\
		for all rooted morphisms $\Omega\colon \sqcup_{i}T_{n_i}\to T_{n}$, the equality
		\begin{equation}
		\begin{tikzcd}
		& & \  \ar[Rightarrow]{d} & \\
		\mathcal{O}(\sqcup_{i}T_{n_i}) \ar{rr}{\mathcal{O}(\Omega)}\ar[swap]{dd}{A_{\sqcup_{i}T_{n_i}}} \ar[bend left=30]{rrrr}{\mathcal{O}(\tau_n\circ \Omega)} & & \mathcal{O}(T_n) \ar{dd}{A_{T_{n}}} \ar{rr}{\mathcal{O}(\tau_n)} \ar[Rightarrow]{ddll}{A_{\Omega}}  & & \mathcal{O}(T_n)\ar{dd}{A_{T_{n}}} \ar[Rightarrow]{ddll}{\phi_n} 
		\\ 
		& & & 
		\\
		\End_\kappa ^X(\sqcup_{i}T_{n_i}) \ar{rr}{\End_\kappa ^X(\Omega)}  \ar[bend right=30, swap]{rrrr}{\End_\kappa ^X(\tau_n\circ \Omega)} & & \End_\kappa ^X(T_n) \ar{rr}{\End_\kappa ^X(\tau_n)}  \ar[Rightarrow]{d} & & \End_\kappa ^X(T_n)  \\ 
		& & \  & &  
		\end{tikzcd} 
		\end{equation}
		\begin{equation}
		=
		\begin{tikzcd}
		& & \  \ar[Rightarrow]{d} & \\
		\mathcal{O}(\sqcup_{i}T_{n_i}) \ar{rr}{\mathcal{O}(\sqcup \tau_{n_i}^{k_i})}\ar{dd}[swap]{A_{\sqcup_{i}T_{n_i}}} \ar[bend left=30]{rrrr}{\mathcal{O}(\tau_n\circ \Omega)} & & \mathcal{O}(\sqcup_{i}T_{n_i}) \ar{dd}{A_{T_{n}}} \ar{rr}{\mathcal{O}(\Omega^{(n)})} \ar[Rightarrow]{ddll}{\sqcup \phi_{n_i}^{k_i}}  & & \mathcal{O}(T_n)\ar{dd}{A_{T_{n}}} \ar[Rightarrow]{ddll}{A_{\Omega^{(n)}}} 
		\\ 
		& & & 
		\\
		\End_\kappa ^X(\sqcup_{i}T_{n_i}) \ar{rr}{\End_\kappa ^X(\sqcup \tau_{n_i}^{k_i})}  \ar[bend right=30, swap]{rrrr}{\End_\kappa ^X(\tau_n\circ \Omega)} & & \End_\kappa ^X(T_n) \ar{rr}{\End_\kappa ^X(\Omega^{(n)})}  \ar[Rightarrow]{d} & & \End_\kappa ^X(T_n)  \\ 
		& & \  & &  
		\end{tikzcd}\label{isophineqn2}
		\end{equation}
		of natural isomorphisms holds, where $\tau_n\circ \Omega = \Omega^{(n)} \circ \sqcup_{i} \tau_{n_i}^{k_i}$ is the standard factorization  of the morphism $\tau_n\circ \Omega$ in $\Forests$ 
		given in \eqref{eqnfactorization}. \label{itemC2}
	\end{enumerate}
\end{theorem}

\begin{proof}
	It is clear that we can extract from the structure of a cyclic algebra the natural isomorphisms \eqref{isophineqn0} through $\phi_n:=A_{\tau_n}$.  By definition these have to satisfy the coherence conditions (C1) and (C2) given above.

	Conversely, suppose we are given an $\overline{\mathcal{O}}$-algebra $A$.
	To extend the structure of $A$ to a cyclic algebra, we have to define natural
	transformations 
	\begin{equation}
	\begin{tikzcd}
	\mathcal{O}(\sqcup_{i}T_{n_i}) \ar{rr}{\mathcal{O}(\Gamma) } \ar[swap]{dd}{A_{\sqcup_{i}T_{n_i}}} & & \mathcal{O}(T_n)  \ar{dd}{A_{T_n}} \ar[Rightarrow, shorten >= 10, shorten <= 10]{ddll}{A_{\Gamma}} \\
	& & \\
	\End_\kappa ^X(\sqcup_{i}T_{n_i}) \ar{rr}{\End_\kappa^X(\Gamma) } & & \End_\kappa ^X(T_n)
	\end{tikzcd} 
	\end{equation} 
	for all morphisms $\Gamma\colon  \sqcup_{i}T_{n_i}\to T_{n} $ in $\Forests$. We define them to agree with the ones provided by $A$ on rooted morphisms and by $\phi_n$ on the morphisms $\tau_n$.
	Finally, for an arbitrary morphism $\Gamma$ in $\Forests$,
	we consider its factorization
	$ \Gamma = \Gamma' \circ \sqcup_{i} \tau_{n_i}^{k_i}$
	from \eqref{eqnfactorization}
	and define $A_\Gamma$ by
	\begin{equation}
	\begin{tikzcd}
	\mathcal{O}(\sqcup_{i}T_{n_i}) \ar{rr}{\mathcal{O}(\Gamma) } \ar[swap]{dd}{A_{\sqcup_{i}T_{n_i}}} & & \mathcal{O}(T_n)  \ar{dd}{A_{T_n}} \ar[Rightarrow, shorten >= 10, shorten <= 10]{ddll}{A_{\Gamma}} \\
	& & \\
	\End_\kappa ^X(\sqcup_{i}T_{n_i}) \ar{rr}{\End_\kappa^X(\Gamma) } & & \End_\kappa ^X(T_n)
	\end{tikzcd}
	:= \begin{tikzcd}
	& & \  \ar[Rightarrow]{d} & \\
	\mathcal{O}(\sqcup_{i}T_{n_i}) \ar{rr}{\mathcal{O}(\sqcup \tau_{n_i}^{k_i})}\ar[swap]{dd}{A_{\sqcup_{i}T_{n_i}}} \ar[bend left=30]{rrrr}{\mathcal{O}( \Gamma)} & & \mathcal{O}(\sqcup_{i}T_{n_i}) \ar{dd}{A_{T_{n}}} \ar{rr}{\mathcal{O}(\Gamma')} \ar[Rightarrow]{ddll}{\sqcup \phi_{n_i}^{k_i}}  & & \mathcal{O}(T_n)\ar{dd}{A_{T_{n}}} \ar[Rightarrow]{ddll}{A_{\Gamma'}} 
	\\ 
	& & & 
	\\
	\End_\kappa ^X(\sqcup_{i}T_{n_i}) \ar{rr}{\End_\kappa ^X(\sqcup \tau_{n_i}^{k_i})}  \ar[bend right=30, swap]{rrrr}{\End_\kappa ^X( \Gamma)} & & \End_\kappa ^X(T_n) \ar{rr}{\End_\kappa ^X(\Gamma')}  \ar[Rightarrow]{d} & & \End_\kappa ^X(T_n) \ .  \\ 
	& & \  & &  
	\end{tikzcd}
	\end{equation}	
	Now the conditions (C1) and (C2) imply that this defines a cyclic 
	algebra. 
	
	Both constructions are inverse to each other.
\end{proof}

\begin{remark}[]\label{remliftingpairing}
	By the definition of the endomorphism operad, any operation  $o \in \mathcal{O}(T_n)$ acts as a 1-morphism
	$o : A^{\otimes (n+1)}\to I$,
	and the component $\phi_n^o$ of the natural isomorphism $\phi_n$
	at $o$ is an isomorphism in $\cat{M}(A^{\otimes (n+1)},I)$
	which in the graphical calculus (Remark~\ref{remgraphcalc}) we write as
	\begin{align}
	\input{lt1.tikz}
	\end{align} 
	Here $\tau.o$ is the image of $o$ under the action with $\tau$.
	The coherence condition (C1) implies that applying this isomorphism $n+1$ times yields the identity. 
	The coherence condition (C2) ensures that 
these isomorphism are compatible with the composition of operations.
\end{remark}

\subsection{Adaption to a presentation in terms of generators and relations}
Many operads, including the ones appearing in this paper, have a description in terms of generators and 
relations. 
For us, the case of $\Cat$-valued operads is most important. Although $\Cat$ is a symmetric 
monoidal bicategory, we will treat it as an ordinary symmetric 
monoidal category by forgetting the 2-morphisms in order to build operads from generators and relations.
 In this case, the operadic composition is strictly associative. This has the practical advantage that we can rely on the
well-developed theory 
of generators and relations
in this setting, see e.g.\ \cite[Section~1.2.5]{FresseI}. Specifically for the category-valued 
case, we make a few conventional adjustments following~\cite[Section~4.1]{littlebundles}
that we now also briefly recall: The generators are a sequence $( G_{\operatorname{Ob}}(n) )_{n\in \N }$ of sets of generating objects and another sequence $(  G_{\operatorname{Mor}}(n) )_{n\in \N }$ of sets of
generating morphisms. They come with maps $s_n,t_n\colon  G_{\operatorname{Mor}}(n) \to G_{\operatorname{Ob}}(n)$ specifying the source and target for generating morphisms.  When listing generators for a $\Cat$-valued operad,
 we only list them as a set and complete
them freely to a symmetric sequence in categories. This means for example if we specify a generating object $\mu$ in
arity two, we also add freely an operation $\sigma_{1,2} \mu $ corresponding to the application of
 the permutation $\sigma_{1,2}$ to $\mu$. Furthermore, we freely add the composition of generating morphisms.
 
If $\mathcal{O}$ is a $\Cat$-valued cyclic operad and if the underlying operad $\overline{\mathcal{O}}$ is presented by generating objects and generating morphisms subject to relations in the above-explained sense, we can give the following more explicit version of the Lifting Theorem that we will need in the next section:

\begin{corollary}[Lifting Theorem in terms of generators and relations]\label{remgenrel}
	Let $\mathcal{O}$ be a $\Cat$-valued cyclic operad with a fixed presentation in terms of generators and relations for the underlying operad $\overline{\mathcal{O}}$.
	For a given $\overline{\mathcal{O}}$-algebra in a symmetric monoidal bicategory $\cat{M}$, 
	 a lift of this $\overline{\mathcal{O}}$-algebra structure to a cyclic $\mathcal{O}$-algebra structure amounts
	precisely to the choice of a non-degenerate symmetric pairing $\kappa:X\otimes X\to I$ on the underling object $X$ of the $\overline{\mathcal{O}}$-algebra and isomorphisms
	\begin{equation}\label{eqnphio}
\input{lt1.tikz}
\end{equation} for every generating object $o \in \mathcal{O}(T_n)$ with $n\neq 0$
	subject to the following relations:
	\begin{enumerate}
		
		\item[(C)] The $n+1$-fold composition of $\phi_n^o$ with itself is the identity. 
		
		\item[(R)] Suppose we have a relation $  \mathcal{O}(\Omega)(o_1,\dots , o_\ell)=o=\mathcal{O}(\widetilde \Omega)(\widetilde o_1,\dots , \widetilde o_k)$ between generating objects of $\mathcal{O}$ (here $\Omega$ and $\widetilde \Omega$ are morphisms in $\RForests$, and $o$ is an operation in arity $n$).  Then we obtain two natural isomorphisms
			\begin{equation}\label{eqnphio2}
		\input{lt2.tikz}
		\end{equation}
		induced by \eqref{eqnphio}
		and the standard decomposition of $\tau_n\circ \Omega$ and $\tau_n\circ \widetilde \Omega$, respectively. We need to impose the relation that these two isomorphisms agree, thereby allowing us to extend the definition of the isomorphisms \eqref{eqnphio} consistently from generating operations to all operations.

		\item[(M)] For every generating morphism, $r: o \to o'$
		the square

			\begin{equation}\label{eqnphio3}
		\input{lt3.tikz}
		\end{equation}
		commutes.
		Here the vertical arrows are induced by $\tau.r$ and $r$, respectively.
	\end{enumerate}
\end{corollary}

\begin{proof}
Assume we are given a cyclic algebra $A$ over $\mathcal{O}$.
Then the evaluation of the isomorphisms~\eqref{isophineqn0} on a generating operation $o$ gives rise to the 2-isomorphism $\phi_n^o$. Their naturality with respect to generating morphisms amounts exactly to~(M).
Condition~(C1) and~(C2) in Theorem~\ref{thmlifting} above imply~(C) and~(R), respectively.

Conversely, let us assume that we are given 
a non-cyclic algebra $A$ together with
isomorphism $\phi_n^o$ of the form~\eqref{eqnphio}
 for all generating objects
 subject to (C), (R) and~(M).
We need 
to define $\phi_n$ 
\begin{equation}
\begin{tikzcd}
\mathcal{O}(T_n) \ar{rr}{\mathcal{O}(\tau_n) } \ar[swap]{dd}{A_{T_n}} & & \mathcal{O}(T_n)  \ar{dd}{A_{T_n}} \ar[Rightarrow, shorten >= 10, shorten <= 10]{ddll}{\phi_n} \\
& & \\
\End_\kappa ^X(T_n) \ar[swap]{rr}{\End_\kappa^X(\tau_n) } & & \End_\kappa ^X(T_n)
\end{tikzcd} 
\end{equation} 
at an arbitrary object $O\in \cat{O}(T_n)$. We can write $O$ non-uniquely as $O=\cat{O}(\Omega)(o_1,\dots , o_m)$ for a rooted tree $\Omega \colon \sqcup_{i=1}^m T^{(i)} \to T_n $ and generating objects $o_i\in \cat{O}(T^{(i)})$. Now the standard factorization $\tau_n \circ \Omega = \Omega' \circ (\sqcup \tau_{i}^{k_i})$ allows us to define $\phi_n^O$ in terms of the $\tau_n^{o_i}$ through the right hand side of point~(C2)
in Theorem~\ref{thmlifting}. Condition~(R) is equivalent 
 to this being well-defined. The collection of the $\phi_n^O$ assemble into a natural transformation $\phi_n$ by condition~(M)
 since every morphism can be written as a combination of generating morphisms. We need to verify that $\phi_n$ satisfies condition (C1) and (C2). Condition~(C1) follows directly from~(C). Moreover, the extension of $\phi_n^o$ (for generating objects) to $\phi_n^O$
 (for arbitrary objects)
  is precisely constructed in such a way that~(C2) holds as soon as it is well-defined.      
	\end{proof}

	\section{Cyclic associative algebras in a symmetric monoidal bicategory and Grothendieck-Verdier structures\label{seccycas}}
The 
cyclic associative operad $\As : \Forests \to \Set$ sends a corolla with edges $E$ to the set of cyclic orders on $E$. 
By means of the symmetric monoidal functor $\Set \to \Cat$ forming the discrete category for a given set, we consider $\As$ as a category-valued cyclic operad.
In this section, we characterize cyclic associative algebras in an arbitrary symmetric monoidal bicategory and afterwards specialize to the symmetric monoidal bicategory $\Fin$ to find a connection to Grothendieck-Verdier duality in the sense of Boyarchenko-Drinfeld \cite{bd}.

\subsection{A characterization of cyclic associative algebras in a symmetric monoidal bicategory}
	First we present the characterization of cyclic associative algebra in an \emph{arbitrary} symmetric monoidal bicategory.
	In order to do this in a compact way, we introduce the following notion:
	
	\begin{definition}\label{def-self-dual-algebra}
		A \emph{self-dual algebra} in a symmetric monoidal bicategory $\cat{M}$ is an object $X$ with the following structure:
			\begin{itemize}
			
			\item[(M)] The object $X$ is endowed with the structure of an associative algebra up to coherent isomorphism
			whose product, unit, associator and unitors we denote as follows:
			\begin{equation}\label{eqnmonoidal}
				\input{monoidal.tikz}
			\end{equation}
			(The term `up to coherent isomorphism' includes the standard coherence conditions on the associators and unitors \cite[Section~VII.1]{maclane}, in particular the pentagon axiom.)
			\item[(P)] The object $X$ is endowed with a non-degenerate symmetric pairing $\kappa : X\otimes X\to I$
			whose  coevaluation   we  denote by $\Delta:I\to X\otimes X$.
			We denote  by
			\begin{equation}\label{eqnsymmetryiso}
				\input{symmetry.tikz} 
			\end{equation}
			the symmetry isomorphisms.

			\item[(Z)] The product $\mu$
			and the pairing $\kappa$ come with an isomorphism
			\begin{equation}\label{eqnOmegaiso0}
				\input{cyclicity2.tikz} \ . 
			\end{equation}
		\end{itemize}
		This data is subject to the following relations:
		\begin{itemize}
			\item[(H1)]	The  isomorphism 
			\begin{align}\label{eqndefpsi}
				\input{psi.tikz} 
			\end{align}
			makes the
			hexagon
			\begin{align}
				\input{psirel.tikz}
			\end{align}
			commute.
			
			\item[(H2)] The isomorphism
			\begin{equation}
				\input{condgamma.tikz}
			\end{equation}
			induced by $\gamma$ agrees with the one induced by the left unitor $\ell$.

		\end{itemize}
		\end{definition}

	\begin{theorem}\label{thmstrucasalg}
		The structure of a cyclic associative algebra on an object $X$
		in a symmetric monoidal bicategory $\cat{M}$
		 amounts precisely to the structure of a self-dual algebra on $X$.
	\end{theorem}

	\begin{proof}
		By Definition~\ref{defmodalgM} a cyclic associative algebra in a symmetric monoidal bicategory $\cat{M}$ is an object $X\in\cat{M}$ and a morphism $A: \As \to \End_\kappa^X$ 
		of cyclic operads from the cyclic associative operad to the cyclic endomorphism operad formed by a non-degenerate symmetric pairing $\kappa$ on $X$.
		This already gives us the non-degenerate symmetric pairing mentioned in point (P).

		We conclude now from the Lifting Theorem~\ref{thmlifting} that a cyclic algebra $A: \As \to \End_\kappa^X$ on $(X,\kappa)$ precisely amounts to the following:
		\begin{itemize}
			\item[($*$)] A \emph{non-cyclic} $\cat{M}$-valued algebra over $\As$, i.e.\ an associative algebra in $\cat{M}$ up to coherent isomorphism, thereby giving us precisely part~(M) in Definition~\ref{def-self-dual-algebra}. The fact that an associative algebra up to coherent isomorphism yields the algebraic structure listed under point~(M) with the coherence conditions described in \cite[Section~VII.1]{maclane}, most notably the pentagon axiom, is  well-known, but a concrete reference to a proof that, without further non-trivial comparisons, is valid in our framework is not easy to give. We can, however, extract it from the literature as follows: Denote by $\As^{\otimes}$ the symmetric monoidal enveloping category of $\As$ \cite[Definition~1.7]{horel}. Then associative algebras in $\cat{M}$
			up to coherent isomorphism
			 are symmetric monoidal functors $\As^{\otimes}\to\cat{M}$ up to coherent isomorphism. In this picture, we can see that they are exactly algebras over the $\infty$-operad $\As$, or equivalently $E_1$, in the sense of Lurie~\cite[Chapter~3]{lurieHA}. The needed statement is then given in \cite[Example~5.1.2.4]{lurieHA}. A detailed derivation of the coherence conditions can be extracted by restricting the statements of \cite[Section~4]{timw} from $E_2$ to $E_1$. 
			(It might be helpful to mention that \cite[Section~6]{timw} contains also a detailed unpacking of the notion of $\infty$-operads and their algebras from 
			 \cite{lurieHA} with a special focus on the bicategorical situation.)

			\item[$(**)$]  Natural isomorphisms as given by the Lifting Theorem~\ref{thmlifting} subject to the coherence conditions (C1) and (C2) also given there.
		\end{itemize}
		The proof will proceed in two steps: In step~\ref{proofAs1}, we use the Lifting Theorem in terms of generators and relations (Corollary~\ref{remgenrel}) to explicitly give the natural isomorphisms from $(**)$ and their coherence conditions. In step~\ref{proofAs2}, we prove that these are equivalent to (Z), (H1) and (H2) in the statement of the Theorem.
		
	\begin{enumerate}[label=(\roman*)]
	
	\item\label{proofAs1}        In order to explicitly describe the natural isomorphisms appearing in $(**)$, we use Corollary~\ref{remgenrel} where we have spelled out the Lifting Theorem for the situation that the underlying operad is given in terms of generators and relations.

			For the associative operad, a very easy presentation in terms of generators and relations is available: There are three generating operations
				\begin{align}
			\input{asop.tikz}\label{eqngenas}
			\end{align}
			corresponding to the operadic identity (our operads do not have an operadic identity by default, so it has to be included into the list of generators), the monoidal unit and the monoidal product, respectively,
			subject to the relations
				\begin{align}
			\input{asop2.tikz}
			\end{align}
			(corresponding to associativity and unitality, respectively) and the rather trivial relations for the operadic identity which we do not include here. All these relations hold strictly, but
			we still always obtain algebras up to coherent isomorphism
			because by the conventions set in Section~\ref{seccycmodalg0} all functors are weak. 
			Let us recall that a cyclic permutation acts trivially on the generating operations. Note that this does not imply that it acts trivially on all composed operations.

			We now spell out Corollary~\ref{remgenrel} for the associative operad: From \eqref{eqnphio},
			 we obtain an isomorphism for each non-nullary generator:
		\begin{itemize}
			
			\item The isomorphism that the operadic identity gives rise to agrees with the symmetry isomorphism of the pairing.
			Point (C) in Corollary~\ref{remgenrel} tells us that the square of the first of these isomorphism is the identity,
			but this already holds  because of the symmetry of the pairing.
			
			\item The product generator $\mu$ gives an isomorphism 
	\begin{align}\label{eqncyclicity}
	\input{cyclicity.tikz}
	\end{align}
	whose threefold composition is the identity.

	\end{itemize}
Point (M) in Corollary~\ref{remgenrel} is not relevant 
for the associative operad because the latter is discrete;
it has no morphisms between operations.
According to point (R) in Corollary~\ref{remgenrel}, there is a relation for each of the three non-trivial relations (As), (LU) and (RU). 
The relations corresponding to the operadic unit do not induce any additional 
conditions due to our conventions laid out in Definition~\ref{defmodalgM} which allow us to assume without loss of generality that the  operadic unit acts as the identity of $X$.

	Let us first derive the relation coming from associativity (As).
	 For notational convenience, we introduce  the operation $m : X\otimes X\otimes X\to I$ encoding the combination of $\kappa$ and $\mu$:
				\begin{align}
			\input{mandmu.tikz}
			\end{align}
	Then the associativity relation induces an isomorphism
	\begin{align}\label{Eq: alpha}
	\input{alpha.tikz}
	\end{align}
	build from the associator $\alpha$ and the snake isomorphism.
	The relation coming from associativity (As) 
	now corresponds to the equality of the two natural isomorphisms 
	\begin{align}
	\input{2maps.tikz}
	\end{align}  
	we can build from $\alpha'$ and $\Omega$. Concretely, this condition can be
	rewritten in terms of the commutativity of the diagram
			\begin{align}
			\input{coherent_cyclic.tikz}
			\end{align}
	When reformulated in terms of $\mu$ and $\alpha$,  this amounts to the commutativity of
	\begin{equation}\label{eqnr30}
				\input{r3.tikz}
				\end{equation} 
					Similarly, the left and right unitality relation (LU) and (RU)
				give us the following commuting diagrams:

				\begin{equation}\label{eqnr10}
				\input{r2.tikz}
				\end{equation}

				\begin{equation}\label{eqnr20}
				\input{r1.tikz}
				\end{equation}
				In the relation corresponding to (LU), $\Omega^2$ appears since acting with 
				the generating cyclic permutations on the composed operation on the left corresponds to acting twice with the generator of the cyclic permutations on $\mu$.
				This concludes the derivation of all the  structure and conditions coming from the Lifting Theorem.

	\item\label{proofAs2}
		Let us summarize step~\ref{proofAs1}: A cyclic associative algebra on $(X,\kappa)$ amounts precisely to the structure of an associative algebra
		up to coherent isomorphism
		 on $X$
		together with an isomorphism $\Omega$ from \eqref{eqncyclicity} whose three-fold composition is the identity and which makes \eqref{eqnr30}, \eqref{eqnr10} and \eqref{eqnr20} commute. 
		It remains to prove that there is the following 1:1 correspondence of structure and relations:
		\begin{align}
		  \text{$\Omega$ as in \eqref{eqncyclicity} subject to $\Omega^3=\id$ and \eqref{eqnr30}, \eqref{eqnr10} and \eqref{eqnr20}}               \ \xleftrightarrow{\ \ \ 1:1 \ \ \ } \  \text{$\gamma$ as in (Z) subject to (H1) and (H2)} \ .  
		\end{align} 
		We establish the two directions of this correspondence separately:
		\begin{itemize}
		\item[$(\longrightarrow)$] We send $\Omega$ to the isomorphism
		\begin{align}\label{eqndefgamma}
		\input{Def_gamma.tikz} \ . 
		\end{align}
		In order to prove that $\gamma$ actually satisfies (H1), we need the following preliminary consideration: We insert  
	 in \eqref{eqnr30}  the monoidal unit into the first argument from the left and obtain the following  diagram in which the outer diagram commutes thanks to \eqref{eqnr30} (we suppress the unitors):\enlargethispage*{1cm}
		\begin{align}
		\input{r3_id.tikz}
		\end{align}
	Now we observe:	The upper triangle commutes thanks to the symmetry requirements on $\Sigma$, the left middle square commutes since
		the horizontal composition of 2-morphisms applied to \emph{different} 1-morphisms 
		is commutative, and the right middle square commutes since the snake isomorphisms
		are chosen to be coherent. 
		Therefore, the outer diagram and all inner diagrams except for the lower pentagon commute.  From this we conclude that
		 the lower pentagon commutes as well, and this tells us that $\Omega$ can be written as the following composition:
			\begin{align}\label{eqndefomega}
				\input{Def_Omega.tikz}
				\end{align}
		One can now directly verify that $\Omega^3=\id$ implies the property (H1) for $\gamma$ (or rather for $\psi$ that is used to state (H1)).
		Property~(H2) follows 
		the diagram
		\begin{align}
			\input{r1-1.tikz}
		\end{align}
		in which the square commutes by \eqref{eqnr20}
		and the triangle is one of the standard coherence conditions for any non-cyclic associative algebra in a symmetric monoidal bicategory.
		By definition the composition of the lower horizontal arrows is $\gamma_{u,-}$. Since $\Sigma$ squares to the identity, (H2) follows.

	\item[$(\longleftarrow)$] Starting from the isomorphism $\gamma$ subject to (H1) and (H2), we can define $\Omega$ via 
		\begin{align}\label{eqndefomega2}
		\input{Def_Omega2.tikz}
	\end{align}
After all, we already know that $\Omega$ must be of this form. 
It remains to prove that $\Omega$ --- if defined that way --- satisfies $\Omega^3=\id$ and \eqref{eqnr30}, \eqref{eqnr10} and \eqref{eqnr20}. 
This boils down to writing out the corresponding diagrams and filling them in 
with smaller squares and triangles using the fact that the vertical 
composition of 2-morphisms applied at different 1-morphisms is 
commutative and the fact that the snake isomorphisms are chosen to be 
coherent. 
Let us give the details: \begin{itemize}

	\item For the proof of~\eqref{eqnr10},
	consider the following diagram:
	\begin{equation}
		\input{4-5.tikz}
	\end{equation}
	The outer diagram is really~\eqref{eqnr10} by definition.
	The commutativity follows from the commutativity of all the subdiagrams. The commutativity of the subdiagrams is clear (or holds by definition) for all except~$(*)$. The commutativity of~$(*)$ follows from (H1) and the symmetry property that $\psi$ inherits from $\Sigma$.

	\item Similarly to~\eqref{eqnr10},
	the commutativity of~\eqref{eqnr20} follows from the commutativity of the following diagram:
	\begin{equation}
		\input{r10_check.tikz}
	\end{equation}

	\item The relation $\Omega^3 =\id$ follows from~(H1).

	\item The proof of~\eqref{eqnr30} follows from a computation similar to the ones leading to~\eqref{eqnr10} and~\eqref{eqnr20}.

\end{itemize}

	\end{itemize}
The assignments $(\longrightarrow)$ and $(\longleftarrow)$ are inverse to each other:
Suppose we start with $\Omega$ and define $\gamma$ via~\eqref{eqndefgamma} and define $\Omega'$ using $\gamma$ via \eqref{eqndefomega2}. We need to show $\Omega'=\Omega$. But this is a consequence of~\eqref{eqndefomega}.
Next suppose we start with $\gamma$ and define $\Omega$ via~\eqref{eqndefomega} and use the so-defined $\Omega$ to define $\gamma'$ via~\eqref{eqndefgamma}. By definition $\gamma'$ is then composition in clockwise direction in the following diagram:
\begin{equation}
	\input{inverseproof.tikz}
\end{equation}
The triangle on the left commutes by (H2). The two squares commute because the horizontal composition of 2-morphisms applied to different 1-morphisms is commutative. Since $\Sigma$ squares to the identity, $\gamma'=\gamma$.
This proves	 that the assignments $(\longrightarrow)$ and $(\longleftarrow)$ are inverse to each other and completes step~\ref{proofAs2} and hence the proof. \enlargethispage*{1cm}
\end{enumerate}
			\end{proof}

	\subsection{Relation between cyclic associative algebras and Grothendieck-Verdier categories}
	In order to phrase Theorem~\ref{thmstrucasalg} in terms of Grothendieck-Verdier duality, we recall
	the definition of the latter from \cite{bd} (Grothendieck-Verdier categories have
	 been considered earlier in \cite{barr} under the name \emph{$\star$-autonomous categories}). We will use conventions dual to the ones in \cite{bd}. This is merely for convenience and does not make an essential difference, see Remark~\ref{remgvalt}.

	\begin{definition}\label{defGV}
		A \emph{Grothendieck-Verdier category} is a monoidal category $\cat{C}$ together with an 
		object $K \in \cat{C}$ such that
		$\cat{C}(K,X \otimes -)$
		is representable for every $X\in\cat{C}$
		and such that the functor $D:\cat{C}\to\cat{C}^\opp$
		sending $X$ to a representing object $DX$ for $\cat{C}(K,X \otimes -)$ is an equivalence.
				 The object $K$ is referred to as \emph{dualizing object}. The functor $D$ is referred to as \emph{duality functor}.
	\end{definition}

In more detail, the functor $D$ is defined by \emph{choosing} an object $DX \in \cat{C}$ and a natural isomorphism $\cat{C}(K,X \otimes -)\cong \cat{C}(DX,-)$ for every $X \in \cat{C}$.
The assignment $X\mapsto DX$ extends to a functor by the Yoneda Lemma.
Note that representability of $\cat{C}(K,X \otimes -)$ is a property; 
the pair of the choice of $DX$ as representing object and the isomorphism 
$\cat{C}(K,X \otimes -)\cong \cat{C}(DX,-)$, however, is only 
unique up to canonical isomorphism.
While $D$ is only essentially unique in the sense just explained, 
the requirement that $D$ is an equivalence 
does not depend on the choice involved in the definition of $D$.

	\begin{remark}\label{remgvalt}
		Definition~\ref{defGV} means that for some distinguished object $K \in \cat{C}$ we have fixed isomorphisms
		\begin{align}
		\cat{C}(K,X\otimes Y) \cong \cat{C}(DX,Y)
		\end{align}
		natural in $X$ and $Y$. Since $I\otimes -$ is isomorphic to the identity functor, we obtain \begin{align} K\cong DI  \label{eqnKandI}
			\end{align}
		by a canonical isomorphism.
		In \cite{bd}
		natural isomorphisms
		$\cat{C}(X\otimes Y,K) \cong \cat{C}(X,DY)$ are used instead.
		Both definitions are equivalent via categorical duality
		in the following sense:
		If a category $\cat{C}$ with monoidal product $\otimes$ has a Grothendieck-Verdier structure with duality $D$ according to Definition~\ref{defGV}, then $\cat{C}^\opp$ equipped with $\otimes^\opp$ (here, the `$\opp$' on the monoidal product means that we pass to the opposite category \emph{and} flip the tensor factors) is a Grothendieck-Verdier category with duality $D^\opp : \cat{C}^\opp \to \cat{C}$ in the sense of  \cite{bd}.
	\end{remark}
	
	\begin{remark}[Normalized duality]\label{normalizedrem}
		We have seen in~\eqref{eqnKandI} that, regardless of \emph{how} the duality functor is defined, there is always a canonical isomorphism $K\cong DI$.
		In fact, it is most convenient to have $\cat{C}(K,I\otimes -)$ represented through $\cat{C}(K,I\otimes -)\stackrel{\ell}{\cong} \cat{C}(K,-)$, where $\ell$ is the left unitor. Then $DI=K$ strictly on object level. We will refer to any duality functor that extends these choices made for the monoidal unit as \emph{normalized}. By the essential uniqueness of duality functors, it can always be assumed without loss of generality that the duality functor is normalized.
		\end{remark}

	\begin{example}\label{Ex:rigid}
		Every (right) rigid monoidal category is an example of a Grothendieck-Verdier category.
		Recall that a monoidal category $(\cat{C},\otimes, I )$ is \emph{(right) rigid} if every object $X\in \cat{C}$
		admits a right dual $X^{\vee}$. This is an object $X^{\vee}\in \cat{C}$
		together with an evaluation map $\ev_X \colon X^\vee \otimes X \longrightarrow  I $ and 
		coevaluation map $\operatorname{coev}_X \colon  I  \longrightarrow X\otimes X^\vee$ which satisfy the usual snake relations. We can define a
		Grothendieck-Verdier structure on the monoidal category $\cat{C}$ that consists of the object $K= I $ 
		and the natural isomorphisms
		\begin{align}
		\cat{C}( I ,X\otimes Y)& \longrightarrow \cat{C}(X^\vee,Y) \\ 
		\left( f\colon  I  \rightarrow X\otimes Y \right) &\longmapsto \left( X^\vee \xrightarrow{\id_{X^\vee}\otimes f} X^\vee \otimes X\otimes Y  \xrightarrow{\ev_X\otimes \id_Y} Y \right) \ \ 
		\end{align}
		for all $X,Y\in \cat{C}$. 
	\end{example}

\begin{example}\label{examplecomplement} The following example is well-known:
	For a set $X$, denote by $\wp(X)$ the category of subsets of $X$ with inclusions as morphisms. The union provides a monoidal structure on $\wp(X)$ with monoidal unit $\emptyset\in\wp(X)$. For $U\in \wp(X)$, denote by $\comp(U)\in\wp(X)$ the complement. The canonical isomorphisms
	\begin{align} \wp(X)(X,U\cup - )\cong \wp(X)(\comp (U) , - ) \end{align}
	endow $(\wp(X),\cup)$ with a Grothendieck-Verdier structure with dualizing object $X$ and duality $\comp$. If $X$ is not the empty set, this provides us with an example of a Grothendieck-Verdier category 
	which does not come from a rigid monoidal category 
	in the sense of Example~\ref{Ex:rigid}.
\end{example}

\begin{definition}
	A Grothendieck-Verdier category whose dualizing object coincides up to isomorphism
	 with the monoidal unit is called an \emph{r-category}.
	\end{definition}

By Example~\ref{Ex:rigid} every rigid category can be seen as an r-category, but by \cite[Example~0.9]{bd} it is false that conversely every r-category 
comes from a rigid monoidal structure.

		Since a Grothendieck-Verdier structure is a weakening of rigidity, one might hope that there is also a notion of a pivotal structure. Boyarchenko and Drinfeld \cite[Definition~5.1]{bd}
	propose the following (again, we present the dualized version, see Remark~\ref{remgvalt}):
	
	\begin{definition}\label{defpivbd}
		A \emph{pivotal structure} on a Grothendieck-Verdier category $\cat{C}$ with dualizing object $K$ and duality $D$ is the choice of an isomorphism
		\begin{align}
		\psi_{X,Y} : \cat{C}(K, X\otimes Y)\to\cat{C}(K,Y\otimes X)\    \label{eqnpiviso}
		\end{align}
		natural in $X,Y \in \cat{C}$ satisfying
		\begin{align}
		\label{eqncohpivotal1}	
		\psi_{X,Y}=\psi_{Y,X} ^{-1}
		\end{align} 
		and making the diagram
		\begin{equation}
		\begin{tikzcd}[column sep=0.02in]
		\cat{C}(K,(X\otimes Y)\otimes Z) \ar{rr}{\psi_{X\otimes Y,Z} } & & \cat{C}(K,Z\otimes (X\otimes Y) ) \ar{rr}{\cat{C}(K,   \alpha_{Z,X,Y} ) } && \cat{C}(K,(Z\otimes X)\otimes Y) \ar{dl}{\psi_{Z\otimes X,Y} } \\ 
		&  \cat{C}(K,X\otimes (Y\otimes Z)) \ar{lu}{\cat{C}(K,\alpha_{X,Y,Z}) }  && \cat{C}(K,Y \otimes (Z\otimes X)) \ar{ld}{\cat{C}(K,   \alpha_{Y,Z,X}   ) } \\ && \cat{C}(K,(Y\otimes Z)\otimes X) \ar{lu}{\psi_{Y\otimes Z, X} }
		\end{tikzcd} \label{eqncyclictriangle}
		\end{equation}
		commute
		for $X,Y,Z\in \cat{C}$. Here $\alpha$ is the associator of the monoidal category $\cat{C}$. 
	\end{definition}

	\begin{remark}\label{Rem:D^2=1}
		By \cite[Proposition~5.7]{bd},
		a pivotal structure amounts precisely to a natural monoidal isomorphism $D^2 \cong \id_\cat{C}$ whose component at the unit $I$ is the canonical isomorphism $D^2I\cong I$. 
	\end{remark}

	\begin{example}
		The notion of a pivotal structure on a Grothendieck-Verdier category generalizes the notion of a pivotal structure on a rigid monoidal category: Recall that a \emph{pivotal structure} on a rigid monoidal category $(\cat{C},\otimes, I )$
		is a monoidal natural isomorphism $\omega \colon -^{\vee \vee}\Longrightarrow \id_{\cat{C}}$. This induces a pivotal structure for the corresponding rigid Grothendieck-Verdier structure (see~Example~\ref{Ex:rigid}) by sending a morphism $f: I \to X\otimes Y$ to 
		\begin{align}
		I \ra{\operatorname{coev}_{X^\vee}} &\ X^\vee \otimes X^{\vee \vee}\cong X^\vee \otimes (I\otimes X^{\vee\vee})\\ 
		\ra{\id_{X^\vee}\otimes( f \otimes \id_{X^{\vee \vee}})} &\  X^\vee \otimes(( X \otimes Y) \otimes X^{\vee \vee})           \cong ((X^\vee \otimes X) \otimes Y) \otimes X^{\vee \vee}              \\ \ra{(\ev_{X}\otimes \id_Y) \otimes \omega_X }   &\ (I\otimes  Y)\otimes X\cong Y\otimes X \  .
		\end{align}  
	\end{example}

By a Grothendieck-Verdier category in $\Fin$ we mean an object $\cat{C}\in\Fin$ together not only 
with a Grothendieck-Verdier structure on the underlying category, but actually a lift of all the 
structure to structure living inside $\Fin$. This means in particular that the monoidal product 
will be left exact by construction.

Non-degenerate symmetric pairings
in $\Fin$
are intimately related to the morphism spaces by Lemma~\ref{lemmakappastar} (this was a consequence of the Eilenberg-Watts calculus for finite categories).
Now a careful comparison of the characterization of cyclic associative algebras in Theorem~\ref{thmstrucasalg} and the axioms
of a pivotal Grothendieck-Verdier category leads to the following specialization of Theorem~\ref{thmstrucasalg}:

\begin{theorem}\label{thmcyclas}
The structure of a cyclic associative algebra in $\Fin$ 
amounts precisely  to a pivotal Grothendieck-Verdier category in $\Fin$.
\end{theorem}

\begin{proof}
	The result is a straightforward specialization of Theorem~\ref{thmstrucasalg} once we take the description of
	non-degenerate symmetric pairings in $\Fin$ given in Lemma~\ref{lemmakappastar} into account.
	More precisely, a  non-degenerate symmetric pairing $\kappa : \cat{C}\boxtimes\cat{C}\to\FinVect$ can be equivalently written as $\kappa(X,Y)=\cat{C}(DX,Y)$, where $D:\cat{C}\to\cat{C}^\opp$ is an equivalence.
	
	Taking this into consideration, Theorem~\ref{thmstrucasalg} allows us to conclude that a cyclic associative algebra  $\cat{C}\in\Fin$ is precisely the following structure: \begin{itemize}
		\item By point~(M) we obtain a monoidal structure on $\cat{C}$.
		
		\item By point~(P) --- without exploiting the symmetry yet --- $\cat{C}$ comes with a non-degenerate pairing $\kappa$ which can equivalently be described by an equivalence $D:\cat{C}\to\cat{C}^\opp$. 
		
		\item By point~(Z) we obtain natural isomorphisms $\cat{C}(K,X\otimes Y)\cong \cat{C}(DX,Y)$ with $K=DI$. For $X=I$, this isomorphism is induced by the left unitor thanks to (H2).
		
		\item The structure in the three preceding points describes precisely a Grothendieck-Verdier category in $\Fin$
		with normalized duality (which one can always assume without loss of generality by Remark~\ref{normalizedrem}).
		The symmetry part of~(P) and condition~(H1) now amount precisely to a pivotal structure on this Grothendieck-Verdier category. 
		In more detail, the fact that the symmetry isomorphisms of $\kappa$ square to the identity gives us \eqref{eqncohpivotal1}, and the commutativity of the diagram in~(H1) gives us~\eqref{eqncyclictriangle}.
		\end{itemize}
	\end{proof}

\begin{remark}\label{remalttolex}
	The characterization of cyclic associative algebras in Theorem~\ref{thmstrucasalg} works in an \emph{arbitrary} symmetric monoidal bicategory.
	But the translation to pivotal Grothendieck-Verdier categories (Theorem~\ref{thmcyclas})
	makes use of the specialization to the symmetric monoidal bicategory $\Fin$.
	For example, Theorem~\ref{thmcyclas} fails in $\Cat$ because for any cyclic associative algebra in $\Cat$ the underlying category has to be dualizable in $\Cat$. But by Remark~\ref{remcatdual} this actually implies that this category is equivalent to the one-point category. Clearly, this is not the case for all Grothendieck-Verdier structures in $\Cat$.
\end{remark}

\begin{remark}\label{remgvfa}
Cyclic associative algebras in the symmetric monoidal category of vector spaces are symmetric Frobenius algebras (if one, as in this article, considers \emph{symmetric} pairings).	
Motivated by this fact, one might refer to the algebraic structure that in Theorem~\ref{thmstrucasalg} (and the specialization in Theorem~\ref{thmcyclas}) we prove to be equivalent to a cyclic algebra over the associative operad (in a symmetric monoidal bicategory) as a \emph{symmetric Frobenius algebra up to coherent isomorphism}.
The pairing $\kappa :X\otimes X\to I$ for the underlying \emph{non-symmetric} Frobenius structure (i.e.\ the structure in Theorem~\ref{thmstrucasalg} without the isomorphism $\Sigma$ and  (H1) and (H2)) can be described in terms of a trace $\varepsilon : X \to I$. This follows from arguments given by Street in \cite[Proposition~3.2]{street}.
While this allows us to describe parts of the algebraic structure found in Theorem~\ref{thmstrucasalg} in an equivalent way, it does, as far as we see, not provide a shortcut in the proof of the relation to \emph{operadically} defined cyclic associative algebras --- and it is the latter that we need for applications in low-dimensional topology.
	\end{remark}

\section{Categorical framed little disks algebras and ribbon Groth\-en\-dieck-Verdier structures\label{secfE2GV}}
The \emph{operad $E_2$ of little disks} is the topological operad whose space $E_2(r)$ of $r$-ary operations is given by the space of embeddings $\left(\disk^2 \right)^{\sqcup r} \to \disk^2$ of the disjoint union of $r$ disks into another disk that are composed of a translation and rescaling, see \cite{bv68,bv73} and additionally \cite{FresseI} for a textbook introduction. 
There is a well-known extension of the little disks operad, the \emph{operad $\framed$ of framed little disks}
which allows, 
in addition to translations and rescalings, also rotations, i.e.\
\begin{align}
\framed(r)=E_2(r)\times (\mathbb{S}^1)^{\times r}      \label{eqnformulaframed}
\end{align} as spaces, where the $r$ factors of
 $\mathbb{S}^1$ encode the rotation parameter. 
The operadic composition is given by composition of maps and can be obtained 
from the semidirect product construction in \cite{salvatorewahl}.
In \cite{budney} it is proven that $\framed$ is equivalent to the operad of conformal balls, and that the latter comes with a cyclic structure. 
Therefore, $\framed$ is equivalent to a cyclic operad. 

This section is devoted to the characterization of cyclic algebras  over $\framed$ with values in an arbitrary symmetric monoidal bicategory.

\subsection{A groupoid model for the cyclic operad of framed little disks}
From \eqref{eqnformulaframed} it follows that the framed little disks operad is aspherical. A groupoid model can be given in terms of ribbon braids \cite{salvatorewahl}. The purpose of this subsection is to also describe the cyclic structure on the level of this groupoid model.
We do this by defining a cyclic structure on the ribbon braid operad. 
In~\cite{CIW} 
Campos, Idrissi and Willwacher give a 
cyclic action on a parenthesized version of $\RBr$. The action defined 
in this section is the 
 analogue of this cyclic structure on $\RBr$. 
In Section~\ref{secszeroframed}, we prove that it corresponds to the \emph{geometric} cyclic structure.

We first recall the operad of ribbon braids:
Denote by $\pi \colon RB_n \longrightarrow \Sigma_n$ the canonical map from
the ribbon braid group on $n$ strands to the symmetric group on $n$ letters. This map defines an $RB_n$-action on $\Sigma_n$. The corresponding action groupoids $\Sigma_n \DS RB_n$ for varying $n\ge 0$
provide a groupoid model $\RBr$ for the framed little disks operad, see \cite{WahlThesis} and \cite[Section~7]{salvatorewahl} for details.
Concretely, the groupoid of 
arity $n$-operations is given by $ \RBr(n) \coloneqq \Sigma_n \DS B_n\times (\star \DS \Z)^n$,
where $\star \DS \mathbb{Z}$ is the groupoid with one object with automorphism group $\mathbb{Z}$.
We refer to \cite[Section~1.2]{WahlThesis} for the details on the operad structure;
we will also momentarily give a presentation in terms of generators and relations.
The operad $\RBr$ being a groupoid model for $\framed$ means that there is an equivalence
\begin{align}
\RBr \xrightarrow{\ \simeq  \ } \Pi \framed \label{BBrframedequiv1}
\end{align}
of operads. 
On the level of objects, the category-valued operad $\RBr$ 
coincides with the associative operad $\As$ meaning that it has a binary associative and unital operation $\mu$
(the generators and relations were listed explicitly on page~\pageref{eqngenas}).
As a consequence, there is a canonical operad map
$\As \longrightarrow \RBr $. 
However, $\RBr$ has non-trivial morphisms
 in the groupoids of operations that are generated under operadic composition by 
 the braiding $c\colon \mu \longrightarrow  \mu^\opp$ and
the balancing $\theta\in \RBr(1)$.
In order to formulate the relations for these morphisms, we will use a graphical notation: We denote the braiding by an overcrossing and the balancing by a filled disk (or ellipse depending on the number of strands it has to be stretched over). 
The inverse braiding $\mu^\opp \to \mu$ is denoted by an undercrossing and the inverse balancing by a white disk with  boundary.
This graphical calculus for the \emph{morphisms} in the groupoids of operations is different from the graphical calculus for the \emph{objects} of operations that we have already used. In order to avoid confusion, we will print the graphical computations for the morphisms with fat lines (this is merely for convenience; it is logically not needed). 
The generators that $\RBr$ has in addition to those of $\As$ and their relations are now as follows (when listing generating morphisms for groupoid-valued operads, we do not list their inverses explicitly):
	\begin{align}
\input{fE2.tikz}
\end{align}
The relations (B1) and (B2) require some explanation and are only well-defined after the relations from the associative operad are imposed. For 
example, the equation (B1) has to be understood as the commutativity of the 
diagram
\begin{equation}\label{Eq: H2}\tag{B1}
		\input{H2.tikz}
\end{equation}
in $\RBr(3)$. We have used here the usual notation $\circ_i$ for the partial 
composition of operations. The relation (B2) has to be interpreted 
analogously. Clearly, these relations just encode the usual hexagon relations
for braided monoidal categories.

We need to elaborate on a notational subtlety:
In the graphical calculus, we have the braiding $c:\mu \to \mu^\opp$ and the inverse braiding $c^{-1} : \mu^\opp \to \mu$. But by symmetry, there is also a braiding $\mu^\opp \to \mu$ and its inverse $\mu \to \mu^\opp$. In other words, we have four types of braidings if we take into account not only the type of crossing, but also the source and target operation. Very often this distinction is suppressed in the notation. In the sequel, this distinction will be relevant. For this reason, we will (whenever needed) indicate the source and target operation with numbers as follows:
	\begin{align}\label{eqnallthebraidings}
\input{otherbraidings.tikz}
\end{align}

In order to equip $\RBr$ with a cyclic structure, we first exhibit the cyclic action on the  groupoids of operations $\RBr(n)$ for $n=1,2$:

\begin{itemize}
	\item The $\Z_2$-action on 
	$\RBr(1)$ is defined to be trivial. 
	
	\item The generator $\tau_3$ of $\mathbb{Z}_3$ acts by the functor $Z(\tau_3)\colon \RBr(2)\longrightarrow \RBr(2)$ which is defined to be the identity on objects. On morphisms it is given as follows:
	\begin{itemize}
		\item The action $Z(\tau_3)$ 
		on the twist on one of two strands or both strands is actually already fixed by the definition on $\RBr(1)$ and the requirements for a cyclic structure:
			\begin{align}
		\input{action_on_twists.tikz}
		\end{align}

		\item     The definition $Z(\tau_3)$
		on the braiding is given by:
			\begin{align}\label{eqncyclicactiononbraiding}
			\hspace{-\leftmargin}
		\input{cyclic_action_on_fE2.tikz}
		\end{align}
		Note that one of these assignments fixes the remaining three.

	\item The cyclic action on the other generators is trivial. 
	\end{itemize}
\end{itemize}

\begin{lemma}\label{lemmacycliclowdim}
	$Z(\tau_2):\RBr(1)\to\RBr(1)$ and $Z(\tau_3):\RBr(2)\to\RBr(2)$  are functors that
	extend the permutation actions of the operad structure on $\RBr$. More precisely,
	$Z(\tau_2)$ yields a $\Sigma_2$-action on $\RBr(1)$, and $Z(\tau_3)$ extends the $\Sigma_2$-action on $\RBr(2)$ to a $\Sigma_3$-action.
	\end{lemma}

\begin{proof}
	For $Z(\tau_2)$, the statement is clear.

	In order to prove that $Z(\tau_3)$ is a functor, we need to verify that it preserves 
	the relation (T1). This is indeed the case:
		\begin{align}
	\input{checkT1.tikz}
	\end{align}

	In order to see that $Z(\tau_3)$ extends the $\Sigma_2$-action on $\RBr(2)$ to a $\Sigma_3$-action,
	we need to verify two things:
	
	\begin{itemize}
		
		\item The relation $\tau_3 \sigma_{1,2}=\sigma_{1,2} \tau_3^2$ with the transposition $\sigma_{1,2}$ is respected. Indeed:
			\begin{align}
		\input{comp_sigma_tau.tikz}
		\end{align}
		
		\item The functor $Z(\tau_3)$ triples to the identity. For this, it suffices to check the following:
			\begin{align}
		\input{triple.tikz}
		\end{align}
		\end{itemize}  
	\end{proof}

\begin{proposition}\label{Prop: Combinatorial model for cyclic structure}
	The cyclic action from Lemma~\ref{lemmacycliclowdim} naturally extends to the structure of a $\Cat$-valued cyclic operad on $\RBr$.
\end{proposition}

\begin{proof}
	We can extend the cyclic action
	from from Lemma~\ref{lemmacycliclowdim}
	to all of $\RBr$
 by using the following general formula for the cyclic 
	action on composed operations~\cite[Remark~5.3]{markl}:
	\begin{align}\label{Eq: action on composition}
	\tau(a\circ_1 b) &= \tau(b)\circ_{m} \tau(a) \\ 
	\tau(a\circ_i b) &  = \tau(a)\circ_{i-1} b \label{Eq: action on composition2}
	\end{align}
	Here $a$ and $b$ are arbitrary operations of arity $n$ and $m$ larger than $0$, respectively, and $1<i\leq n$. When $m=0$ and $i\neq 1$, we can use the same formula. 
	However, if $i=1$, we need to use $\tau(a\circ_1 b) = \tau(a)^2\circ_{m-1} b$ instead. 
	For this to be well-defined,
	we need to verify the compatibility with the relations (B1) and (B2) (the compatibility with (T1) and (T2) is already a part of Lemma~\ref{lemmacycliclowdim}):
 We  use \eqref{Eq: action on composition}
 and \eqref{Eq: action on composition2} to apply the cyclic permutation to the diagram~\eqref{Eq: H2} and get
	\begin{align}\label{eqncyclicB1}
	\input{tH2.tikz}
	\end{align}
For the upper horizontal map, no cyclic permutation is applied to $c$ in accordance with \eqref{Eq: action on composition2}.
When evaluating $\tau$ on the operation which is the result of applying $\sigma_{12}$, we use the relation $\tau \circ \sigma_{23}= \sigma_{12}\circ \tau$ to evaluate one 
	of the actions.
	 We now need to verify that \eqref{eqncyclicB1} commutes. This amounts precisely to the identity
	 	\begin{align}
	 \input{verification_braiding.tikz}
	 \end{align} 
	 which can be easily seen to hold.
	A similar computation shows that (B2) is respected. This completes the proof.
	\end{proof}

\subsection{Equivalence of the cyclic structure on $\RBr$ to the one on $\framed$\label{secszeroframed}}
In this section, we show that the cyclic structure defined on $\RBr$ above 
agrees with the cyclic structure on $\framed$ induced by identifying
it with the operad of genus zero surfaces (here, surfaces will always be compact and oriented) and hence also with the cyclic structure found by Budney in \cite{budney}.
 To this end, we will, similarly to \cite{budney}, exhibit an equivalence from $\framed$ to a topological operad $\szero$ built from disk configurations in the 2-sphere (it is essentially the operad of genus zero surfaces). The operad $\szero$ has an obvious cyclic structure, and we prove 
that there is an equivalence $\RBr \simeq \Pi \szero$ of cyclic operads, where the cyclic structure on $\RBr$ is the one given above.

Before defining $\szero$ we need to recall some basic properties of the 
stereographic projection: For every point $p\in \mathbb{S}^2$, the 
stereographic projection is a diffeomorphism $\operatorname{st} \colon \mathbb{S}^2\setminus \{ p \} \longrightarrow T_{-p}\mathbb{S}^2$
from $\mathbb{S}^2\setminus \{ p \}$ to the tangent space to $\mathbb{S}^2$ at $-p$.  We consider $\mathbb{S}^2$ here as 
a submanifold of $\R^3$ and hence can identify the tangent space $T_{-p}\mathbb{S}^2$ with an affine plane in $\R^3$. The canonical scalar product on $\R^3$ 
induces a natural scalar product on $T_{-p}\mathbb{S}^2$.
A vector $x\in T_{-p}\mathbb{S}^2$ on the unit circle in $T_{-p}\mathbb{S}^2$ can be uniquely extended to a positively oriented orthonormal basis and thereby induces a linear isomorphism $T_{-p}\mathbb{S}^2\cong \R^2$.

Let us now define the operad $\szero$: For $n\ge 0$, a map $\sqcup_{j=0}^n\mathbb{D}_j^2 \longrightarrow \mathbb{S}^2$ is called \emph{admissible} if \begin{itemize}
	\item the images of the interiors are pairwise disjoint,
	\item 
for every map $f_j \colon \mathbb{D}_j^2 \longrightarrow \mathbb{S}^2$, the
point $-f_j(0)\in \mathbb{S}^2$ is not contained in the image of $f_j$, 
\item and
the composition $\mathbb{D}_j^2 \ra{f_j} \mathbb{S}^2\setminus \{-f_j(0)\} \ra{\operatorname{st}} T_{f_j(0)}\mathbb{S}^2\cong \mathbb{R}^2 $ is given by a rescaling,
where the last isomorphism is induced, as explained above,  by the choice of a point on the  unit circle of $T_{f_j(0)}\mathbb{S}^2$, namely the intersection of the  unit circle of $T_{f_j(0)}\mathbb{S}^2$ with the line in $T_{f_j(0)}\mathbb{S}^2$ from $0$ to the point $\operatorname{st}\circ f_j\left( \begin{array}{c}
	1 \\ 
	0
\end{array} \right)$.

\end{itemize}
We equip the space of 
admissible maps with the topology induced by the compact-open 
topology on the mapping space.
There is an $\SO(3)$-action on the space of all admissible maps $\sqcup_{i=0}^n\mathbb{D}_i^2 \longrightarrow \mathbb{S}^2$ coming from the $\SO(3)$-action on $\mathbb{S}^2$. 
An operation in $\szero (n)$ is an $\SO(3)$-orbit of admissible maps  $\sqcup_{j=0}^n\mathbb{D}_i^2 \longrightarrow \mathbb{S}^2$.
The partial composition $\underline{f} \circ_\ell \underline{f'}$ of two operations $\underline{f}:\sqcup_{i=0}^n\mathbb{D}_i^2 \xrightarrow{\sqcup f_i} \mathbb{S}^2$ and $\underline{f'}:\sqcup_{j=0}^{n'}\mathbb{D}_j^2 \xrightarrow{\sqcup f'_j} \mathbb{S}^2 $ can be defined as follows:
For $\underline{f'}$, perform a stereographic projection from the center of the zeroth disk. This yields a configuration of $n'$ disks inside a bigger disk in the plane. We insert these $n'$ disks into the $\ell$-th disk in $\underline{f}$ to get the partial composition $\underline{f} \circ_\ell \underline{f'}$.

The permutation group on $n+1$ letters acts on $\szero (n)$ and thereby endows $\szero$ with the structure of a cyclic operad. 
There is a canonical
equivalence of operads
\begin{align} 
\framed \xrightarrow{\ \simeq  \ } \szero \label{eqnframedsurfequiv}
\end{align} 
which sends a disk configuration in the standard disk to a disk configuration in the lower hemisphere  via (the  inverse of) the
stereographic projection and adds the upper hemisphere as the zeroth disk.

The next result shows that the cyclic structure on $\szero$ can be combinatorially described by the cyclic structure on $\RBr$ from Proposition~\ref{Prop: Combinatorial model for cyclic structure}.

\begin{proposition}\label{Thm: Relation ribbon braids and E2}
	There is an equivalence $\RBr \to \Pi \szero$ of  groupoid-valued cyclic operads.
\end{proposition}

\begin{proof}
	We start by constructing an equivalence $\Phi \colon \RBr \longrightarrow \Pi \szero$ of operads.
	As recalled in \eqref{BBrframedequiv1}, there is an equivalence of operads $\RBr \longrightarrow \Pi \framed$ sending the generator $\mu$ to the disk configuration in the lower left corner
	of Figure~\ref{fig:b}, the braiding to the homotopy which slides disk
	one over disk two (see lower part of Figure~\ref{fig:b}) and the twist to 
	the rotation of a disk by 360 degrees against its orientation. This map is not 
a strict map of operads (associativity is not respected strictly), but 
	only up to coherent isomorphism
	(which in the bicategorical framework is allowed). We can obtain the map as follows: There
	is a zigzag of strict equivalences (where `strict' means that the coherence data is trivial)
	 between $\Cat$-valued operads $\RBr \longleftarrow \catf{PaRB} \to \Pi (\framed)$, where 
	 $\catf{PaRB}$ is the operad of parenthesized braids \cite[Section~3.1]{CIW}, and
	 the first map forgets the parentheses. Now by Proposition~\ref{propwhitehead} we can invert the first map with a non-strict morphism of $\Cat$-valued operads. The composition of this inverse with the second map gives us the desired morphism.   
The equivalence $\Phi \colon \RBr \to \Pi \szero$ is the composition of the equivalence from \eqref{BBrframedequiv1} with the one from~\eqref{eqnframedsurfequiv}. 
	
It remains to show that the equivalence $\Phi$ is compatible with the cyclic structure --- possibly up to coherent isomorphism (again, let us emphasize that our notion of cyclic operads and their morphisms naturally accounts for that). 
For an object $o\in \RBr(n)$,
this compatibility means that the cyclic action on $\Phi(o)$ is trivial up to coherent isomorphism because the cyclic action on $\RBr$ is trivial on the object level.
Indeed, acting with a cyclic permutation on $\Phi(o)$ changes only the sizes of the disks (since $\szero$-operations are $\SO(3)$-orbits,
the configuration stays the same otherwise). 
We denote by $h_{n+1}$ the corresponding rescaling homotopy.  
The rescaling homotopies are coherent in the sense that they are compatible with composition. The reason for this is that the rescalings are unique up to higher homotopy, i.e.\ unique in $\Pi \szero$.

It remains to show that these homotopies form \emph{natural transformations}
	\begin{equation}
	\begin{tikzcd}
	\RBr(n) \ar[rr,"\RBr(\tau_{n+1})"]  \ar[dd,"\Phi"]& & \RBr(n) \ar[dd,"\Phi"]  \\
	& & \\
	\Pi\szero(n) \ar[rr,"{\Pi(\szero)(\tau_{n+1})}",swap] \ar[rruu, Rightarrow, "h_{n+1}"]& & \Pi\szero(n) \ , 
	\end{tikzcd}
	\end{equation}
i.e.\ are compatible with the morphisms in the groupoids of operations.
	It is enough to verify this on the generating morphisms $c$ in arity two and $\theta$ in arity one (because, as argued above, the homotopies are compatible with composition):

	\begin{itemize}
		\item 
The generator $\theta$ is mapped to the automorphism of the configuration in 
Figure~\ref{fig:a} which rotates the first disk $\mathbb{D}^2_1$ by
$-360$ degrees. Applying the cyclic structure maps this to the automorphism of the lower
configuration in Figure~\ref{fig:a} which rotates zeroth disk $\mathbb{D}^2_0$ by $-360$ degrees. 
After applying an element in $\SO(3)$, this agrees with the automorphism of the original 
configuration rotating $\mathbb{D}_0^2$ by $360$ degrees. But this is equivalent 
to rotating $\mathbb{D}_1^2$ by $-360$ degrees. This shows that $h_2$ (which is actually the identity) is natural
with respect to $\theta$. 
		
		\item 
		The image of the braiding under the equivalence $\Phi$ has the following
		 description (using the freedom of choosing representatives in
		the $\SO(3)$-orbits): We rotate $\mathbb{D}_0^2$ by $-180$ degrees,  $\mathbb{D}_1^2$ by 180 degrees and  $\mathbb{D}_2^2$ by 180 degrees, see Figure~\ref{fig:b} for a sketch.   
		Up to a rescaling, which is absorbed by the homotopy $h_3$, the cyclic action produces the morphism rotating $\mathbb{D}_0^2$ by 180 degrees, rotating $\mathbb{D}_1^2$ by $180$ degrees and rotating $\mathbb{D}_2^2$ by $-180$ degrees. Figure~\ref{fig:c} explains that this agrees 
		(after rescaling) with the image of $ \bar{c}\circ_1 \theta^{-1}$  
		under $\Phi$ --- as it should be by the definition of the cyclic action on $\RBr$.
	\end{itemize}
This shows that $\Phi$ is compatible with the cyclic structure and finishes the proof.
	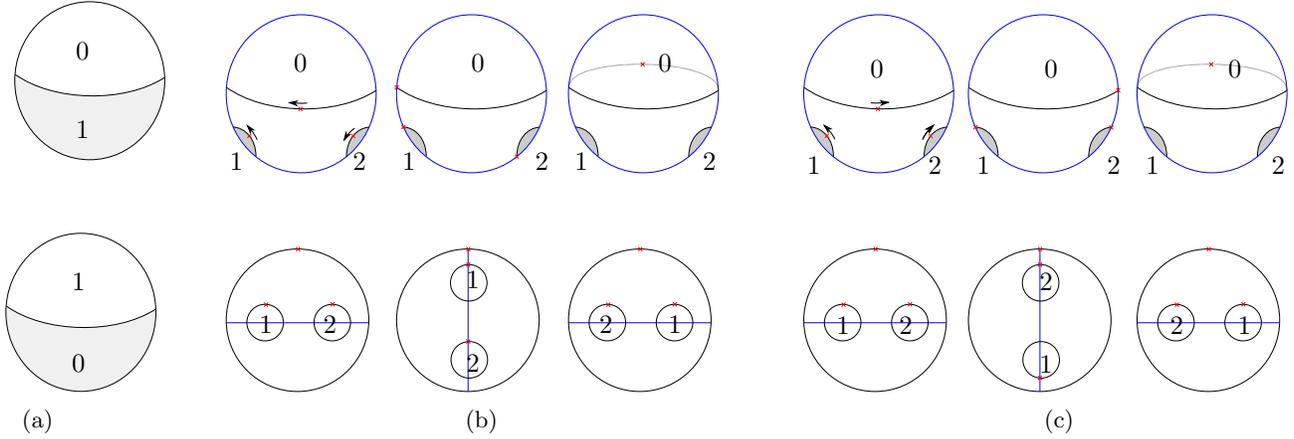
\begin{figure}[h]
	\begin{subfigure}[b]{0.05\textwidth}
		\begin{overpic}[scale=0.18]
			{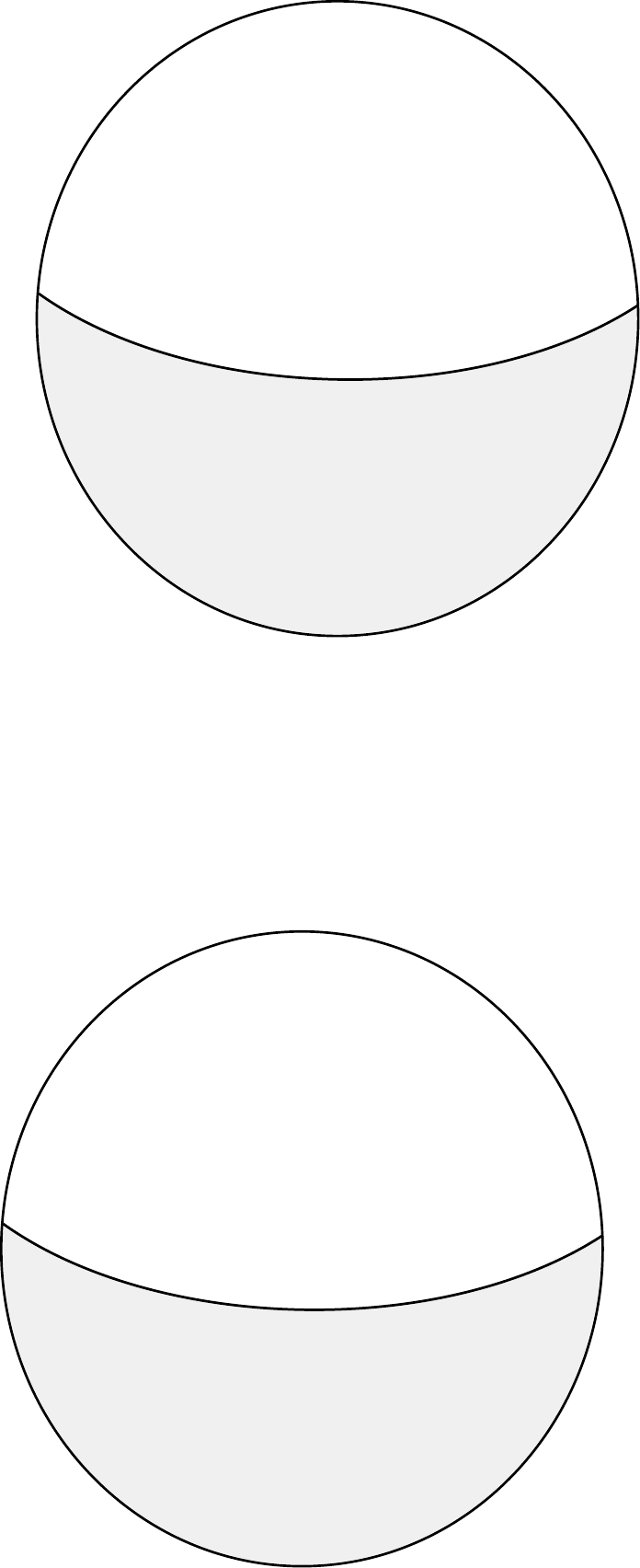}
			\put(18,85){$0$}
			\put(18,65){$1$}
			\put(17,5){$0$}
			\put(17,26){$1$}
		\end{overpic}
		\caption{}
		\label{fig:a}
	\end{subfigure}\hspace{2cm}
	\begin{subfigure}[b]{0.4\textwidth}
		\begin{overpic}[scale=0.18]
			{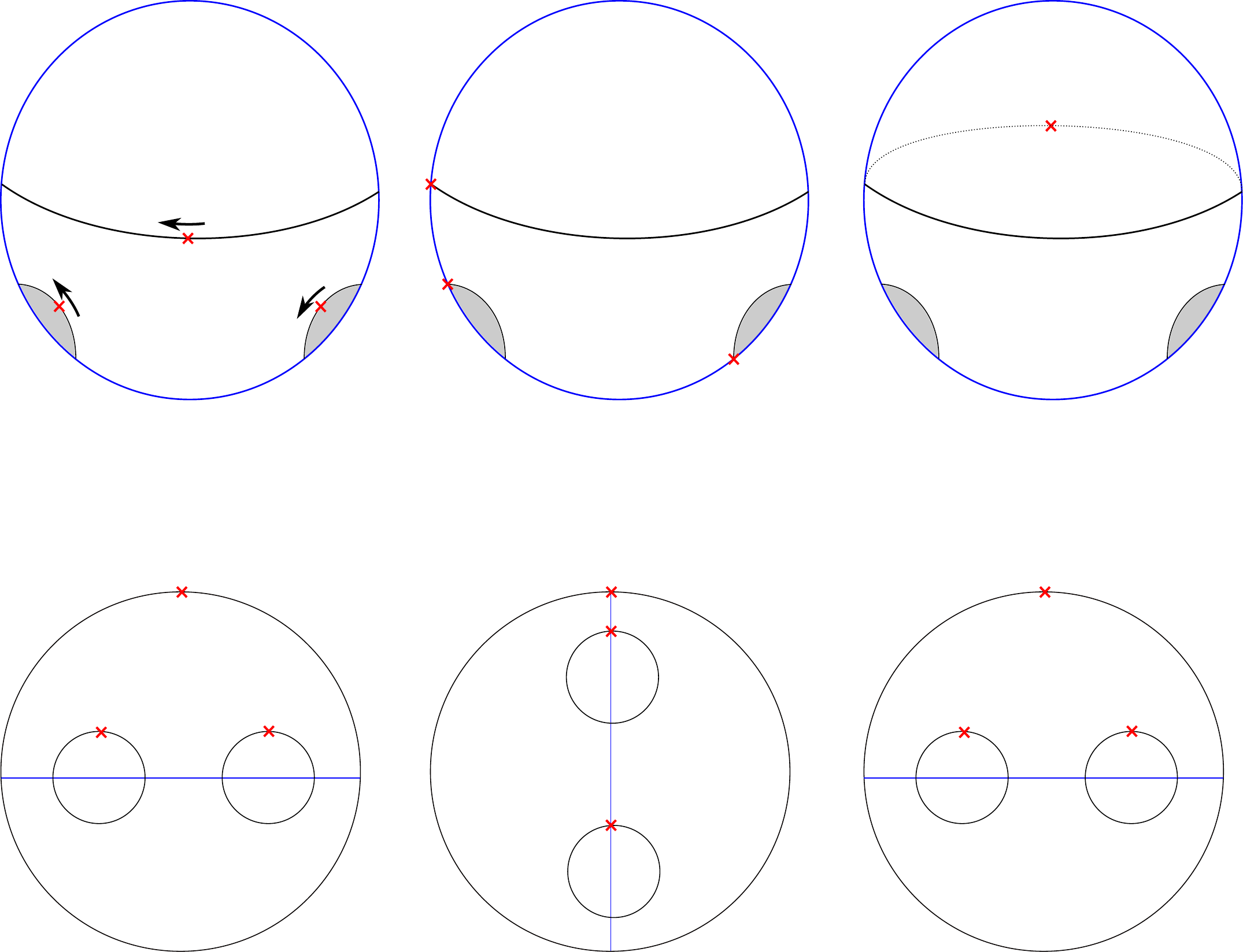}
			\put(12,65){ $0$}
			\put(-1,45){ $1$}
			\put(24,45){ $2$}
			\put(5,12){ $1$}
			\put(18,12){ $2$}
			\put(48,65){ $0$}
			\put(34,45){ $1$}
			\put(61,45){ $2$}
			\put(47,21){ $1$}
			\put(47,4){ $2$}
			\put(86,65){ $0$}
			\put(69,45){ $1$}
			\put(95,45){ $2$}
			\put(74,12){ $2$}
			\put(88,12){ $1$}
		\end{overpic}
		\caption{}
		\label{fig:b}
	\end{subfigure}\hspace{0.8cm}
	\begin{subfigure}[b]{0.4\textwidth}
		\begin{overpic}[scale=0.18]
			{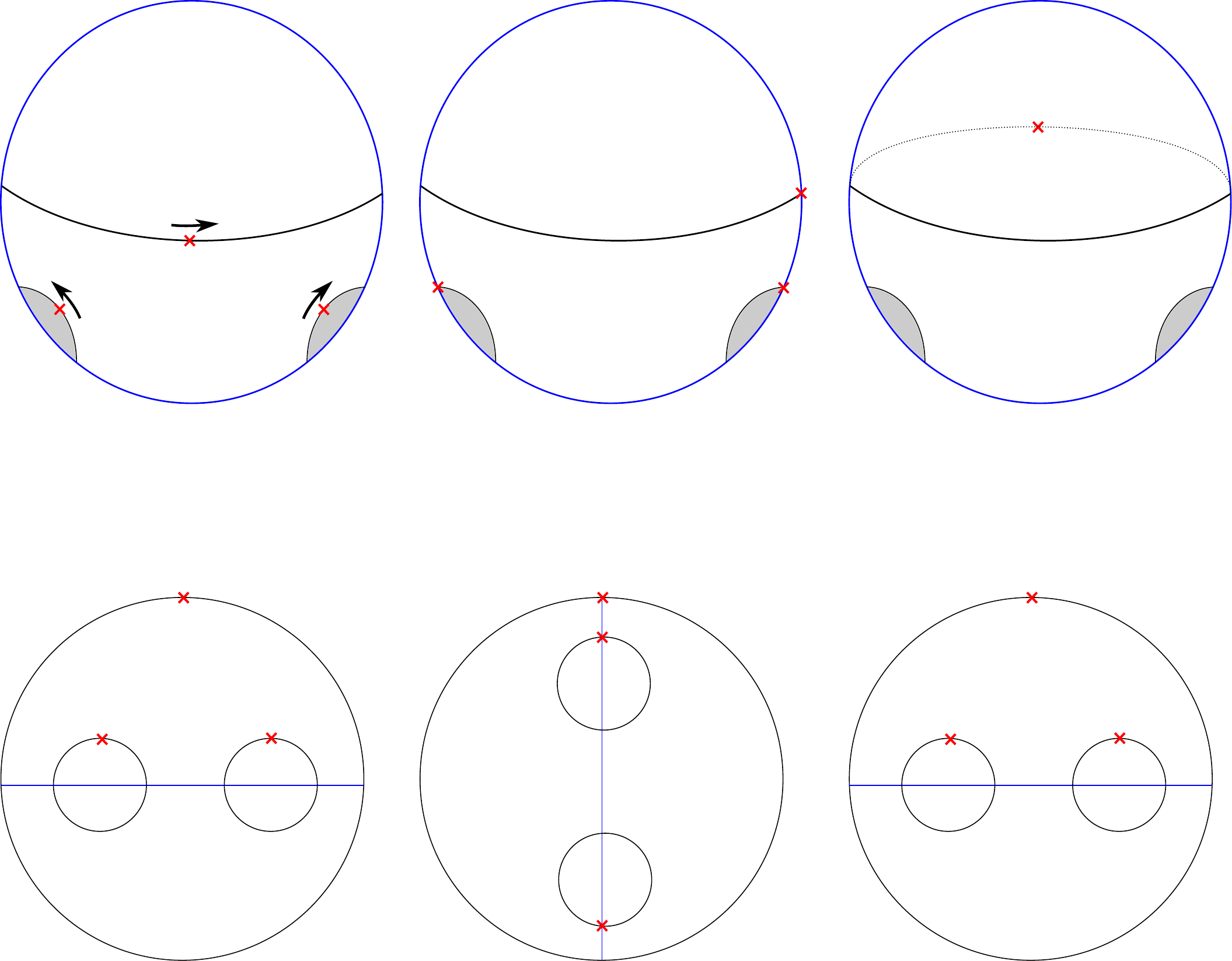}
			\put(12,65){ $0$}
			\put(-1,45){ $1$}
			\put(24,45){ $2$}
			\put(5,12){ $1$}
			\put(18,12){ $2$}
			\put(48,65){ $0$}
			\put(34,45){ $1$}
			\put(61,45){ $2$}
			\put(47,21){ $2$}
			\put(47,4){ $1$}
			\put(86,65){ $0$}
			\put(69,45){ $1$}
			\put(95,45){ $2$}
			\put(74,12){ $2$}
			\put(88,12){ $1$}
		\end{overpic}
		\caption{}
		\label{fig:c}
	\end{subfigure}
	\caption{(a)  The top figure shows the image of the operadic identity in $\RBr$ under the equivalence from 
		Theorem~\ref{Thm: Relation ribbon braids and E2}. The bottom figure shows the same, but \emph{after} the action with a cyclic permutation. The two configurations can be transformed into each other by acting with an element
		of $\SO(3)$. \\
		(b) The top figure shows the path in $\szero(2)$ corresponding to the braiding. The bottom figure is its stereographic projection from the center of $\mathbb{D}_0^2$. We marked the base point of every disk by a cross. The missing crosses on the right
		are on the backside of the sphere. 
		The blue line above is always mapped to the blue line below by the stereographic projection.\\
		(c) The action of $\tau_3$ on the image of $c$ in $\szero(2)$. We use the same drawing conventions as in (b).}
	\label{Fig}
\end{figure} 
\end{proof}

\subsection{A characterization of cyclic $\RBr$-algebras in a symmetric monoidal bicategory}
A (non-cyclic)
$\RBr$-algebra
in a symmetric monoidal bicategory
$\cat{M}$
is a homotopy coherent associative algebra (i.e.\ an object $X\in\cat{M}$ with multiplication $\mu$, associator $\alpha$, unit $u$ and unitors $\ell$ and $r$) plus two isomorphisms, namely a braiding and a balancing
		\begin{align}
\input{bbmonoid.tikz}
\end{align}
subject to the analogue of the relations (T1), (T2), (B1) and (B2), but with the coherence isomorphisms $\alpha$, $\ell$ and $r$ inserted in the necessary places, see \cite{joyalstreet} and also
\cite[Section~4]{timw} and
 \cite[Section~5.4.2]{woike}.
We call $X\in\cat{M}$ endowed with this structure a \emph{balanced braided algebra} in $\cat{M}$.
Taking this notion as a starting point, we can define the following algebraic structure:

\begin{definition}
	A \emph{self-dual balanced braided algebra} in a symmetric monoidal bicategory $\cat{M}$ is a balanced braided algebra $A$ with a non-degenerate pairing $\kappa : A \otimes A \to I$ with 
	an isomorphism $\gamma$ 
	\begin{equation}
		\input{cyclicity2.tikz}
	\end{equation} subject to the following conditions:
\begin{itemize}
	\item[(U)] The isomorphism 
\begin{equation}
	\input{condgamma.tikz}
\end{equation}
induced by $\gamma$ agrees with the one induced by the left unitor $\ell$ (this is condition~(H2) from Theorem~\ref{thmstrucasalg}). 

\item[(PB)] 
 The automorphisms of $\kappa$ induced by $\theta \otimes \id$ and $\id\otimes\theta$ agree:
 \begin{equation}
 	\input{rt2.tikz}
 \end{equation}
\end{itemize}

	\end{definition}

The above definition suggests that a self-dual balanced braided algebra is in particular a self-dual algebra in the sense of Definition~\ref{def-self-dual-algebra}. A priori, this is not obvious, but it follows from the following Lemma:

\begin{lemma}\label{lemmaselfdualbalancedbraided}
	The structure of a self-dual balanced braided algebra on an object $X$ in a symmetric monoidal bicategory amounts precisely to the following:
	\begin{enumerate}
		\item[(BB)]
	The object $X$ carries the structure of a balanced braided algebra in $\cat{M}$ with product $\mu$,
	associator $\alpha$, unit $u$, unitors $\ell$ and $r$, braiding $c$ and balancing $\theta$.
	
	\item[(PZH)] The object $X$ is endowed with a non-degenerate symmetric pairing $\kappa :X\otimes X\to I$ with symmetry isomorphism $\Sigma$ in the sense of Theorem~\ref{thmstrucasalg} (P) and an isomorphism $\gamma$ 
	\begin{equation}
		\input{cyclicity2.tikz}
	\end{equation}
	which satisfy the compatibility conditions (H1) and (H2) from Theorem~\ref{thmstrucasalg}.

	\item[(RM)] 	The following relations are satisfied (the first one is formulated in terms of $\psi$ which was defined using $\gamma$ in \eqref{eqndefpsi}):	\begin{align}
		\input{relation_braiding.tikz}
	\end{align}

\end{enumerate}
	\end{lemma}

\begin{proof}
	We observe 
	that if we are given the structure from the statement of the Lemma, we can forget the symmetry isomorphism. Then we have a self-dual balanced braided algebra. In fact, all axioms hold by definition except for (PB) which, however, easily follows from (RT). 
	
	The non-trivial part is the converse: Suppose we are given a self-dual balanced braided algebra and want to obtain the algebraic structure described in the Lemma.
	The structure and conditions that are part of the structure described in the above Lemma, but are \emph{not} part of the definition of a self-dual balanced braided algebra are the symmetry data for $\kappa$, (H1), (RB) and (RT).
	And in fact, we can show that these four pieces of information are redundant.
	The symmetry isomorphism $\Sigma$ is redundant information because, if~(RB) is supposed to hold, then $\Sigma$ must be given by
	\begin{equation}
		\input{rb.tikz} \label{eqndefsigma}
	\end{equation}
	and in fact, this way of defining $\Sigma$ automatically ensures that it squares to the identity (this follows from (T1), (PB) and (T2)).
	At the same time, any symmetry isomorphism defined using~\eqref{eqndefsigma} automatically satisfies (RB), thereby making (RB) redundant.   Moreover, we note that (H1), when spelled out explicitly, means that the isomorphism
	\begin{align}
		\input{rb3.tikz}
	\end{align}
	is the identity. This indeed follows from a direct computation using~(T1), (PB) and (T2), which makes (H1) redundant.
	Finally, we observe that (RB) implies (RT). 
	\end{proof}

\begin{theorem}\label{thmcycrbr}
	The structure of a cyclic $\RBr$-algebra on an object $X$
	in a symmetric monoidal bicategory $\cat{M}$
	amounts precisely to the structure of a self-dual balanced braided algebra in $\cat{M}$.
	\end{theorem}

The proof of the Theorem follows from Lemma~\ref{lemmaselfdualbalancedbraided} and the following Proposition:

\begin{proposition}\label{proprbrstructure}
		The structure of a cyclic $\RBr$-algebra on an object $X$
	in a symmetric monoidal bicategory $\cat{M}$
	amounts precisely to the  structure described in Lemma~\ref{lemmaselfdualbalancedbraided}.
	\end{proposition}

\begin{proof}
By the Lifting Theorem~\ref{thmlifting} and its version in terms of generators and relations (Corollary~\ref{remgenrel}) a cyclic $\RBr$-algebra amounts precisely to the following:
	\begin{enumerate}[label=(\roman*)]
		\item 
		A non-cyclic $\RBr$-algebra, which gives us precisely the structure of a balanced algebra in point (BB).
		\item Isomorphisms from  generating  objects and the relations between them. But the generating \emph{objects} and their relations are identical for $\As$ and $\RBr$ (the two operads just differ on the level of morphisms), therefore this gives us precisely the isomorphisms and relations found in Theorem~\ref{thmstrucasalg}, i.e.\ (PZH).
		\item Finally, we get according to Corollary~\ref{remgenrel} (M) one relation for each generating morphism:
		\begin{itemize}
			
			\item One relation for the balancing, namely (RT) --- because the cyclic action preserves the balancing on one strand by definition.
			
			\item One relation for the braiding, namely (as follows from \eqref{eqncyclicactiononbraiding}) the commutativity of the following square formulated in terms of $\Omega$ (which is related to $\gamma$ by \eqref{eqndefgamma}): 
			\begin{align}\label{eqnrelationomega000}
				\input{relation_omega.tikz}
			\end{align}
			where the lower arrow is given by $\Omega^{-1}$ which is the value of the natural isomorphism for the cyclic action at $\mu^{\opp}$.
		\end{itemize}
	\end{enumerate}
	In order to complete the proof, it remains to show \eqref{eqnrelationomega000} $\Longleftrightarrow$  (RB):
	
	\begin{enumerate}
		\item[$(\Longrightarrow)$] If we insert in \eqref{eqnrelationomega000} the unit into the first argument from the left, we find after a short calculation (RB) by using the definition of $\gamma$ in terms of $\Omega$ in \eqref{eqndefgamma}, the definition of $\psi$ in terms of $\gamma$ in \eqref{eqndefpsi}, relation~\eqref{eqnr10}, and the fact that braiding with the unit is trivial.

		\item[$(\Longleftarrow)$]
		If we express  $\psi$ entirely in terms of $\Omega$ and $\Sigma$ (using \eqref{eqndefpsi} and \eqref{eqndefgamma}), we can deduce from a lengthy, but straightforward computation that the commutativity of \eqref{eqnrelationomega000}
		is equivalent to the commutativity of the hexagon
		\begin{align}\label{eqnrelationomega0001}
			\input{510.tikz}
		\end{align}
		In order to prove that this hexagon really commutes,
		we replace $\psi$ by means of (RB). 
		The isomorphism obtained by composing in \emph{clockwise} direction in \eqref{eqnrelationomega0001} is now (we suppress the associator in the notation):
		\begin{align}
			\input{rechnung_5_9.tikz}
		\end{align}
		The isomorphism that we arrive at on the right hand side of this equation is the isomorphism obtained by composing in 
		\emph{counterclockwise} direction in \eqref{eqnrelationomega0001}. 
		Hence, \eqref{eqnrelationomega0001} and therefore \eqref{eqnrelationomega000} commutes. 
	\end{enumerate}
	\end{proof}

\subsection{Relation between cyclic framed little disks algebras and ribbon Grothendieck-Verdier categories\label{balancedGV}}
Theorem~\ref{thmcycrbr} can be expressed in terms of ribbon Grothendieck-Verdier structures defined in \cite{bd}:

\begin{definition}\label{Def: blance GV}
	A \emph{braided Grothendieck-Verdier category} is a Grothendieck-Verdier category whose underlying monoidal category is equipped with a braiding $c$.
	A \emph{balancing} on $\cat{C}$ 
	is a natural automorphism of the identity functor $\id_\cat{C}$ whose components $\theta_X : X \to X$ satisfy
	\begin{align}
	\theta_{X\otimes Y}& = c_{Y,X} c_{X,Y} (\theta_X \otimes \theta_Y) \quad \text{for}\quad X,Y \in \cat{C} \label{comptwistbraidingeqn} \ ,\\
	\theta_I &= \id_I \ . \label{eqnribbonunit} 
	\end{align}
	A braided Grothendieck-Verdier structure with a balancing 
	satisfying
	\begin{align}D\theta_X = \theta_{DX} \quad \text{for}\quad X \in \cat{C}  \label{eqnribboncondition}\end{align}
	will be referred to as a \emph{ribbon Grothendieck-Verdier structure}. In that case, the balancing is also called a \emph{ribbon twist}.
\end{definition}

\begin{definition}\label{defcomppivbrbal}
	Let $\cat{C}$ be a braided Grothendieck-Verdier category with pivotal structure $\psi$
	(Definition~\ref{defpivbd})
	and balancing  $\theta$. We call $\psi$ and $\theta$ \emph{compatible} if for all $X,Y \in \cat{C}$ the triangle
	\begin{equation}
	\begin{tikzcd}
	\cat{C}(K,X \otimes Y) \ar[swap]{dd}{\cat{C}(K,c_{X,Y}^{-1})} \ar{rr}{\psi_{X,Y}} & & \cat{C}(K,Y\otimes X)   \\
	& & \\
	\cat{C}(K,Y\otimes X)\ar[swap]{rruu}{\cat{C}(K,\id_Y \otimes \theta_X^{-1}   )} & & 
	\end{tikzcd} \label{compateqn}
	\end{equation} 
	commutes.
\end{definition}

\begin{lemma}\label{lemmapivribbon}
	For any braided Grothendieck-Verdier category $\cat{C}$,  a ribbon twist on $\cat{C}$   gives rise to a unique pivotal structure compatible with the ribbon twist.
	Hence, a ribbon twist on $\cat{C}$ can be equivalently described as a pivotal structure and a ribbon twist  which are compatible.
\end{lemma}

\begin{proof}
	The statement can be reduced to the following: If we are given a balancing and a braiding and \emph{define} $\psi$ by \eqref{compateqn}, the so-defined $\psi$ is a pivotal structure.
	Indeed, the needed conditions~\eqref{eqncohpivotal1} and~\eqref{eqncyclictriangle}
	 follow
	 from a direct computation using~\eqref{comptwistbraidingeqn}-\eqref{eqnribboncondition}
	 (it is helpful to first deduce $\theta_K=\id_K$). 
	\end{proof}

\begin{remark}
As discussed in Remark~\ref{remgvfa}, we may describe the pairing $\kappa : \cat{C}\boxtimes\cat{C}\to\FinVect$ in terms of a trace $\varepsilon \colon \cat{C} \longrightarrow \FinVect$ related to $\kappa$ by
$\varepsilon(X)= \kappa(I,X)$.
Reformulating Definition~\ref{Def: blance GV} in
terms of $\varepsilon$ gives rise to a structure related to the notion 
of a \emph{contraction} on a balanced braided monoidal category defined in~\cite[Definition 12]{Enriquez}. The key differences,
however,
 are that the pairing defined by
$\kappa(X,Y)\coloneqq \varepsilon(X\otimes Y)$ is not required to be non-degenerate    
and that various coherence isomorphisms (including the pivotal structure) are 
assumed to be the identity in \cite{Enriquez}.
\end{remark}

As for the associative operad, we can express Theorem~\ref{thmcycrbr}
in terms of Grothendieck-Verdier duality.

\begin{theorem}\label{thmcycle2}
	The structure of a cyclic $\RBr$-algebra in $\Fin$ 
	amounts precisely  to a ribbon Grothen\-dieck-Verdier category in $\Fin$.
\end{theorem}

\begin{proof}
	Thanks to Proposition~\ref{proprbrstructure}, a cyclic $\RBr$-algebra in $\Fin$ amounts to
	the algebraic structure described in Lemma~\ref{lemmaselfdualbalancedbraided}, i.e.\ a balanced braided monoidal category in $\Fin$ (this is (BB))
	and a pivotal Grothendieck-Verdier structure for the underlying monoidal structure
	(this is (PZH) as follows from our treatment of cyclic $\As$-algebras in Theorem~\ref{thmstrucasalg} and~\ref{thmcyclas}).
	Therefore, by Lemma~\ref{lemmapivribbon},
	it suffices to prove that (RB) $\Longleftrightarrow$ \eqref{compateqn}
	and that
	(RT) $\Longleftrightarrow$ \eqref{eqnribboncondition}. While the former follows from the definitions, the latter requires a proof:
	Relation (RT) is the commutativity of
	\begin{equation}\label{cyclicinvarianceoftwist}
	\begin{tikzcd}
	\kappa(X,Y) \ar{rr}{\Sigma_{X,Y} } \ar[swap]{dd}{  \kappa(X,\theta_Y)   } & &\kappa(Y,X)  \ar{dd}{  \kappa(Y,\theta_X)   }   \\
	& & \\
	\kappa(X,Y)         \ar{rr}{\Sigma_{Y,X} } & & \kappa(Y,X) \ , 
	\end{tikzcd} 
	\end{equation}
	where  $\Sigma$ is the symmetry isomorphism of $\kappa$.
	By recalling the description of the pivotal Grothendieck-Verdier structure in terms of the symmetric pairing $\kappa$
	(Theorem~\ref{thmcyclas}), we see 
	that (RT) is equivalent to the commutativity of the diagram
	\begin{equation}
	\begin{tikzcd}
	\cat{C}(K,X\otimes Y) \ar{rr}{\psi_{X, Y} } \ar[swap]{dd}{   \cat{C}(K,X\otimes \theta_Y)   } & & \cat{C}(K,Y\otimes X)  \ar{dd}{  \cat{C}(K,Y\otimes \theta_X)    }   \\
	& & \\
	\cat{C}(K,X\otimes Y)         \ar{rr}{\psi_{X,Y} } & & \cat{C}(K,Y\otimes X) \ , 
	\end{tikzcd} 
	\end{equation}
	where $\psi$ is the pivotal structure. By the naturality of $\psi$ this is equivalent to the equality
	\begin{equation}\label{eqnequalityofmaps}
	\begin{tikzcd}
	\cat{C}(K,X\otimes Y) \ar{rrrrrr}{\cat{C}(K,\theta_X\otimes Y)= \cat{C}(K,X\otimes \theta_Y) } & & &&&& \cat{C}(K,X\otimes Y)  
	\end{tikzcd} 
	\end{equation}
	of maps. 
	Under the natural isomorphism $\cat{C}(DX,Y)\cong \cat{C}(K,X\otimes Y)$
	from the Grothendieck-Verdier structure, this is equivalent to the equality
	\begin{equation}\label{eqnequalityofmaps2}
	\begin{tikzcd}
	\cat{C}(DX,Y) \ar{rrrrrr}{\cat{C}(D\theta_X,Y)= \cat{C}(DX,\theta_Y) } & & &&&& \cat{C}(DX,Y)  \ .  
	\end{tikzcd} 
	\end{equation}
	The equivalence of (RT) and $D\theta_X =\theta_{DX}$ for all $X\in \cat{C}$ is thereby reduced to the equivalence
	\begin{align} \cat{C}(D\theta_X,Y)= \cat{C}(DX,\theta_Y) \ \text{for all $X,Y\in\cat{C}$} \quad \Longleftrightarrow \quad D\theta_X =\theta_{DX} \ \text{for all}\ X \in \cat{C} \ , 
	\end{align}
	which we will prove now:
	\begin{enumerate}
		\item[$(\Longrightarrow)$] This follows by setting $Y=DX$ and evaluating the left hand side on $\id_{DX}$.
	\item[$(\Longleftarrow)$]
	For this, we take  any morphism $f:DX\to Y$ and observe
	\begin{align}
	\theta_Y f &= f \theta_{DX} \quad \text{(by naturality of $\theta$)} \\
	&= f D\theta_X \quad \text{(since $D\theta_X =\theta_{DX}$)} \ . 
	\end{align}
	\end{enumerate}
	\end{proof}

\begin{theorem}\label{thmalgrbr}
	The 2-groupoid of cyclic $\framed$-algebras in an arbitrary symmetric monoidal bicategory $\cat{M}$ is equivalent to the 2-groupoid of self-dual balanced braided algebras in $\cat{M}$.
	For $\cat{M}=\Fin$, this 2-groupoid is equivalent to the 2-groupoid of ribbon Grothendieck-Verdier categories in $\Fin$. 
	\end{theorem}

\begin{proof}
	The Theorems~\ref{thmcycrbr} and~\ref{thmcycle2} characterize cyclic $\RBr$-algebras. Using the Comparison~Theorem~\ref{comparisonthm}, we can transfer these results to the cyclic $\framed$-operad which is equivalent to $\RBr$ by Proposition~\ref{Thm: Relation ribbon braids and E2}.
	\end{proof}

\begin{example}\label{exampleGVnonrigid}
In order to discuss a class of	 ribbon  Grothendieck-Verdier categories, let us recall the construction of pointed braided fusion categories from Abelian group cocycles, see \cite{eilenbergmaclane}  and \cite[Section~8.4]{egno}:
For a finite Abelian group $G$, denote by $\FinVect_G$ the category of finite-dimensional $G$-graded vector spaces over the complex numbers.  
For $G$-graded vector spaces $V$ and $W$, one can define a monoidal product $V\otimes W$ by
\begin{align}
(V\otimes W)_g = \bigoplus_{ab=g} V_a\otimes W_b \quad\text{for all}\quad g \in G \ . 
\end{align}
In order to specify the associator, we denote by $\mathbb{C}_g$ the ground field $\mathbb{C}$ seen as $G$-graded vector space supported in degree $g$.
The associator is determined by its values on the simple objects $\mathbb{C}_g$ and given on the simple objects by
\begin{align}
\alpha_{\mathbb{C}_{g_1},\mathbb{C}_{g_2},\mathbb{C}_{g_3}} \colon (\mathbb{C}_{g_1}\otimes \mathbb{C}_{g_2})\otimes \mathbb{C}_{g_3} &\to \mathbb{C}_{g_1}\otimes( \mathbb{C}_{g_2}\otimes \mathbb{C}_{g_3}) \\ (v_1\otimes v_2)\otimes v_3  &\mapsto \lambda(g_1,g_2,g_3) v_1\otimes (v_2\otimes v_3) \ , 
\end{align}
where the numbers $\lambda(g_1,g_2,g_3)\in \mathbb{C}^\times$ form a 3-cocycle $\lambda \in Z^3(G;\mathbb{C}^\times)$. 
In order to construct a braiding for this monoidal product, we need to complete $\lambda$ to an Abelian 3-cocycle $\omega =(\lambda,\tau)\in Z^3_{\operatorname{ab}}(G; \C^\times)$, i.e.\ we additionally need a 2-cochain $\tau$ on $G$ such that
\begin{align}
\lambda(g_2,g_3,g_1)\tau(g_1,g_2g_3) \lambda(g_1,g_2,g_3) &= \tau(g_1,g_3)\lambda(g_2,g_1,g_3)\tau(g_1,g_2) \ \ , \\ 
\lambda(g_3,g_1,g_2)^{-1} \tau(g_1g_2,g_3) \lambda(g_1,g_2,g_3)^{-1} &= \tau(g_1,g_3)\lambda(g_1,g_3,g_2)^{-1}\tau(g_2,g_3)\quad \text{for all}\quad g_1,g_2,g_3 \in G \ \ .
\end{align}
Now a braiding is given by
\begin{align}
c_{ \mathbb{C}_{g_1},\mathbb{C}_{g_2}} : \mathbb{C}_{g_1}\otimes \mathbb{C}_{g_2} &\to \mathbb{C}_{g_2}\otimes \mathbb{C}_{g_1} \\ v_1\otimes v_2  &\mapsto \tau(g_1,g_2) v_2\otimes v_1 
\end{align}
This monoidal category is rigid with the left and right dual $V^*$ of $V\in \FinVect_G$ given by $(V^*)_g =V_{g^{-1}}^*$.
Therefore, finite-dimensional $G$-graded vector spaces with the structure specified above by means of the Abelian 3-cocycle $\omega =(\lambda,\tau)$ give us a  braided fusion category that we denote by ${\FinVect_G}^\omega$. It is pointed in the sense that all simple objects are invertible, and in fact, all pointed braided fusion categories are of this form.

The Abelian 3-cocycle $\omega=(\lambda,\tau)$ can be equivalently described by a quadratic form:
A \emph{quadratic form} on a finite Abelian group $G$ is a map $q:G\to \mathbb{C}^\times$ of sets with $q\left(g^{-1}\right)=q(g)$ for all $g\in G$ such that the symmetric function $b_q:G\times G\to k^\times$ defined by \begin{align} b_q(g,h):=\frac{q(gh)}{q(g)q(h)} \quad \text{for}\quad g,h\in G
\end{align} is a bicharacter in the sense that $b(g_1	g_2,h)=b(g_1,h)b(g_2,h)$ for $g_1,g_2,h\in G$. 
Then by \cite{eilenbergmaclane} the
canonical map
\begin{align}
H_\text{ab}^3(G;\mathbb{C}^\times) \to \text{Quad}(G;\mathbb{C}^\times)\ , \quad (\lambda,\tau) \mapsto (g \mapsto \tau(g,g)) 
\end{align}
is an isomorphism.

Although ${\FinVect_G}^\omega$ is rigid, it can still have Grothendieck-Verdier structures that do not come from rigidity: Since any Grothen\-dieck-Verdier duality has to be an anti-equivalence which maps the simple unit $\mathbb{C}_e$ to the dualizing object $K$, the dualizing object must be simple and hence given by $K=\mathbb{C}_{g_0}$ for some fixed $g_0 \in G$. It is easy to observe that for each such choice, we can find a canonical Grothendieck-Verdier structure with duality functor $D_{g_0}=\mathbb{C}_{g_0}\otimes (-)^*$. Note that $D_{e}=(-)^*$ coincides with the usual (rigid) duality. 

From \cite[Theorem~4.2.2]{Zetsche}, we may now deduce the following statement:
Suppose $g_0 = h_0^{-2}$ for some $h_0 \in G$, 
and denote by $q:G\to\mathbb{C}^\times$ the quadratic form associated to the Abelian cocycle $\omega$ and by $b_q:G\times G\to\mathbb{C}^\times$ the symmetric function corresponding to $q$. We define the group morphism $\eta:G\to\mathbb{C}^\times$ by $\eta (g):=b_q(g,h_0)$ for $g\in G$. 
Then ${\FinVect_G}^\omega$ together with duality $D_{g_0}=D_{h_0^{-2}}$ and balancing
\begin{align}
\theta_{\mathbb{C}_g} : \mathbb{C}_g &\to \mathbb{C}_g \\
v &\mapsto q(g) \eta(g) v=q(g) b_q(g,h_0) v = \frac{q(gh_0)}{q(h_0)} v  
\end{align} 
is a pivotal braided Grothendieck-Verdier category with compatible balancing.
\end{example}

\section{The calculus construction \label{calculusconstruction}}
Consider  a $\Cat$-valued
modular operad $\mathcal{O} : \Graphs \to \Cat$ and a modular $\mathcal{O}$-algebra structure $A:\mathcal{O}\to\End_\kappa^X$
on an object $X$ in a symmetric monoidal bicategory $\cat{M}$ 
equipped with a non-degenerate symmetric pairing $\kappa : X \otimes X\to I$ 
(see Definition~\ref{defmodalgM}).  In this section, we provide 
a general construction of a calculus induced by the 
algebra structure. This calculus will formalize and simplify computations related to the algebra $A$. 
It should be seen
as an auxiliary construction needed for the applications in quantum topology
that we will present in the last section.
In particular, it will allow us to concisely formulate and prove gluing properties for modular algebras.

 As a preparation,
 denote by $\Omega$ the symmetric monoidal category of finite sets and surjections (the monoidal product is disjoint union).
 For $T\in \Graphs$, we denote by $\mathsf{C}(T)$ the set of components of $T$. In other words, $T=\sqcup_{c\in \mathsf{C}(T)} K_c$, where $K_c$ is a corolla. By abuse of notation, we will just write $c$ for the corolla $K_c$.
 A morphism $\Gamma : T \to T'$ induces a surjective map $\mathsf{C}(\Gamma):\mathsf{C}(T)\to\mathsf{C}(T')$. This turns $\mathsf{C}:\Graphs \to \Omega$ into a symmetric monoidal functor. 
 Now we define for $X\in \cat{M}$ and a non-degenerate symmetric pairing $\kappa : X\otimes X\to I$,
 a category $\cat{E}_\kappa^X$ whose objects are pairs $(T,\underline{\varphi}   )$, where $T\in\Graphs$ and 
 $\underline{\varphi}= (\varphi_c)_{c\in \mathsf{C}(T)}$ with
 $\varphi_c \in \cat{M}(I,X^{\otimes \Legs(c)})$.  A morphism 
 $(S, \underline{\varphi}= (\varphi_c)_{c\in \mathsf{C}(S)}   )\to (T, \underline{\psi}=(\psi_d)_{d\in \mathsf{C}(T)}   )$ is a morphism $\Gamma:S\to T$ in $\Graphs$ and for any $d\in \mathsf{C}(T)$ a morphism
 $\bigotimes_{c\in \mathsf{C}(\Gamma)^{-1}(d)}   \varphi_c \to  \psi_d^\Gamma   $
 in $\cat{M}(I, \bigotimes_{c\in \mathsf{C}(\Gamma)^{-1}(d)} X^{\otimes\Legs(c)}  )$,
 where $\psi_d^\Gamma$ is given by
 \begin{align}
 	\psi_d^\Gamma : I \ra{\psi_d} X^{\otimes\Legs(d)} \ra{  \substack{\text{ $\Delta$ applied for each internal edge} \\ \text{ of the restriction  $\Gamma_d :\sqcup_{c\in  \mathsf{C}(\Gamma)^{-1}(d)  }  c \to d$ of $\Gamma$ to $d$ } }  }      \bigotimes_{c\in \mathsf{C}(\Gamma)^{-1}(d)} X^{\otimes\Legs(c)} \ . 
 	\end{align}
 The category $\cat{E}_\kappa^X$ comes with a projection $\cat{E}_\kappa^X\to\Graphs$. Moreover, disjoint union and juxtaposition of families of morphisms provide a symmetric monoidal structure on $\cat{E}_\kappa^X$ such that $\cat{E}_\kappa^X\to\Graphs$ becomes a symmetric monoidal functor.

In the next step,
 let us recall a familiar construction:\label{pagegrothendieck}
For a functor $F:\cat{C} \to \Cat$, we denote by $\int F$ its \emph{Grothendieck construction}, see e.g.\
\cite[Section~I.5]{maclanemoerdijk}.
By definition this is the category whose objects are pairs $(c,x)$, where $c\in \cat{C}$ and $x\in F(c)$. A morphism $(c,x)\to(c',x')$ is a pair $(f,\alpha)$, where $f:c\to c'$ is a morphism in $\cat{C}$ and $\alpha : F(f)x\to x'$ is a morphism in $F(c')$. 
The Grothendieck construction comes with a projection functor $\int F \to \cat{C}$.

The Grothendieck construction 
applied to $\mathcal{O}:\Graphs \to \Cat$ and 
$\End_\kappa^X :\Graphs\to\Cat$
gives us a functor $\int \mathcal{O} \longrightarrow \Graphs$, and we may define the category $\mathcal{O} \star_\kappa X$ as the pullback 
\begin{equation}
\begin{tikzcd}
\mathcal{O} \star_\kappa X \ar[r] \ar[d] & \int \mathcal{O} \ar[d] \\ 
\int \cat{E}_\kappa^X \ar[r] & \Graphs
\end{tikzcd}
\end{equation}  
in $\Cat$ (it is also a homotopy pullback because $\int \mathcal{O}\to \Graphs$ is a Grothendieck fibration).
The category $\mathcal{O} \star_\kappa X$ comes with a projection $\pi_\mathcal{O}^X : \mathcal{O} \star_\kappa X\to \Graphs$ (by composition of either pair of composable arrows in the above square) and is  naturally a symmetric monoidal category such that $\pi_\mathcal{O}^X$ is a symmetric monoidal functor.

By virtue of the monoidal product in the symmetric monoidal bicategory $\cat{M}$,
the category $\cat{M}(I,I)$ is symmetric monoidal (the monoidal product can also be described through the composition as an Eckmann-Hilton argument shows; this will then also establish the symmetry). The symmetric monoidal structure on 
$\cat{M}(I,I)$ can be used to define a symmetric monoidal functor $\underline{\cat{M}(I,I)}\colon \Omega \to \Cat$ 
from the category of finite sets and surjections to $\Cat$
sending a finite set $S$ to $\cat{M}(I,I)^{\times S}$ and a surjective map $f\colon S 
\longrightarrow S'$ to the functor
$\cat{M}(I,I)^{\times S} \to \cat{M}(I,I)^{\times S'}$ whose component for $s'\in S$ is
\begin{align}
	\cat{M}(I,I)^{\times S}=\prod_{s'\in S'}  \cat{M}(I,I)^{\times f^{-1}(s')} \ra{\text{projection}}
	\cat{M}(I,I)^{\times f^{-1}(s')} \ra{\otimes} \cat{M}(I,I)
\end{align}

\begin{lemma}\label{lemmacalc10}
	Let $\cat{O}$ be a $\Cat$-valued modular operad and $A: \cat{O}\to\End_\kappa^X$ a modular algebra on an object $X$ in a symmetric monoidal bicategory $\cat{M}$ with non-degenerate symmetric pairing $\kappa$. 
	Then $A$ induces a lax (in the sense that the coherence morphisms are not necessarily isomorphisms) symmetric monoidal transformation between symmetric monoidal functors between symmetric monoidal bicategories (where the involved 1-categories are seen as bicategories without non-trivial 2-morphisms):
	\begin{equation}\label{calculuseqn0}
		\begin{tikzcd}
			& &  \  \ar[Rightarrow,shorten <= 0.01cm, shorten >= 0.01cm]{dd}{\gamma^A} & & \\  \    
			\mathcal{O} \star_\kappa X  \ar[bend left=40]{rrrr}{\star} \ar[bend left=-40,swap]{rrrr}{\underline{\cat{M}}(I,I) \circ \mathsf{C}} & &          & &  \Cat \ . 
			\\ 
			& & \ & &
		\end{tikzcd} 
	\end{equation}
	Here $\star : \mathcal{O} \star_\kappa X\to\Cat$ is the constant functor whose single value is the one-point category.
\end{lemma}

\begin{proof}
	The component of $\gamma^A$ at $(T,o,\underline{\varphi})$ with $\underline{\varphi} =(\varphi_c)_{c\in \mathsf{C}(T)}$ is given at follows:
	We have \begin{align} A_o\in\prod_{c\in\mathsf{C}(T)} \cat{M}(X^{\otimes\Legs(c)},I)\end{align} (this uses the $\cat{O}$-algebra structure on $X$)
	and denote the $c$-component by $A_o^c\in \cat{M}(X^{\otimes\Legs(c)},I)$.
	With this notation, we define
	\begin{align}
		\gamma^A_{(T,o,\underline{\varphi})} :=    \left(   I\ra{   \varphi_c  } X^{\otimes\Legs(c)} \ra{ A_o^c } I         \right)_{c\in\mathsf{C}(T)} \ . 
		\end{align}
The value at a morphism given by the data $\Gamma \colon S \to T$, $ \omega \colon \mathcal{O}(\Gamma)(o) \to o'$ and $ \theta_d \colon \bigotimes_{c\in \mathsf{C}(\Gamma)^{-1}(d)}   \varphi_c \to  \psi_d^\Gamma $ is the natural transformation 
\begin{equation}
\begin{tikzcd}
* \ar[rr] \ar[dd, "	\gamma^A_{(S,o,\underline{\varphi})}", swap] & & *\ar[dd, "	\gamma^A_{(T,o',\underline{\psi})}"] \\ 
 & & \\ 
\mathcal{M}(I,I)^{|\mathsf{C}(S)|} \ar[rr, swap, "{\underline{\mathcal{M}}(I,I)\circ \mathsf{C} (\Gamma)}" ] \ar[rruu, Rightarrow,  "{\gamma_{(\Gamma, \omega, \theta)}}", shorten <= 10,shorten >= 10  ] & &
\mathcal{M}(I,I)^{|\mathsf{C}(T)|}
\end{tikzcd}
\end{equation} 
given by the collection of morphisms $\gamma_{(\Gamma, \omega, \theta)	}^d$ for $d\in \mathsf{C}(T)$
\begin{align}
\gamma_{(\Gamma, \omega, \theta)}^d \colon \left( \bigotimes_{c\in \mathsf{C}(\Gamma)^{-1}(d)}  A^c_o \right)\circ  \varphi_c 
\ra{\theta_d}  \left(\bigotimes_{c\in \mathsf{C}(\Gamma)^{-1}(d)} A^c_o \right) \circ \psi_d^\Gamma \cong A_{\mathcal{O}(\Gamma_d)(\times_{c \in C(\Gamma)^{-1}(d)} o_c)} \circ \psi_d 
 \ra{A_\omega  } A_{o'}\circ \psi_d \ \ .
\end{align}
	A direct computation shows that this is a symmetric monoidal transformation.
	\end{proof}

\begin{proposition}[Calculus functor]
\label{propcalculusfunctor}
Let $\cat{M}$ be a symmetric monoidal bicategory.
		For an $\cat{M}$-valued modular algebra $A$ on $(X,\kappa)$ over a $\Cat$-valued modular operad $\mathcal{O}$, 
		\begin{align}
		\Calc_A : \mathcal{O} \star_\kappa X\to\cat{M}(I,I) \ , \quad
		(T,o,\underline{\varphi}) \mapsto \left(   \gamma^A_{(T,o,\underline{\varphi})}  \right)^\otimes \ . 
		\end{align}
		is a symmetric monoidal functor that we refer to as the
		 calculus (functor) for $A$. 
	\end{proposition}

\begin{remark}
	The calculus construction cannot only be performed for modular algebras, but --- by replacing $\Graphs$ with $\RForests$ or $\Forests$ --- also for algebras over ordinary operads or cyclic operads.
	\end{remark}

\begin{proof}[\slshape Proof of Proposition~\ref{propcalculusfunctor}]
	A morphism
	$f:(T,o,\underline{\varphi})\to (T',o',\underline{\varphi'})$
	in $\mathcal{O} \star_\kappa X$
	gives rise to a natural transformation in the left square of the following diagram:
	\begin{equation}
	\begin{tikzcd}
	\star  \ar[rr,"\gamma_{(T,o,\underline{\varphi})}"] \ar[dd,"=",swap] && \cat{M}(I,I)^{\times \pi_0(T)} \ar[rrrrd,"\otimes"] \ar[dd,"\underline{\cat{M}(I,I)} (  \mathsf{C}(f) )"] \ar[ddll,Rightarrow]\\ &&&&&& \cat{M}(I,I)\ . \\
	\star  \ar[rr,"\gamma_{(T',o',\underline{\varphi'})}",swap] && \cat{M}(I,I)^{\times \pi_0(T')}\ar[rrrru,"\otimes",swap]
	\end{tikzcd}
	\end{equation} 
The square commutes up to a canonical natural transformation.
	The triangle on the right commutes up to a canonical natural isomorphism.
	As a consequence, we obtain a morphism
	$\Calc_A(T,o,\underline{\varphi}) = \left(   \gamma^A_{(T,o,\underline{\varphi})}  \right)^\otimes 
	\to \Calc_A(T',o',\underline{\varphi'}) = \left(   \gamma^A_{(T',o',\underline{\varphi'})}  \right)^\otimes$ in $\cat{M}(I,I)$ which we define to be $\Calc_A(f)$. This way, $\Calc_A$ extends to a functor 
	$\Calc_A : \mathcal{O} \star_\kappa X\to\cat{M}(I,I)$. 
	
Since $\gamma^A$ from \eqref{lemmacalc10}
	is a symmetric monoidal transformation, one concludes that $\Calc_A$ is a symmetric monoidal functor:
	For a disjoint union $(T,o,\underline{\varphi})=(T^{(0)},o_0,\underline{\varphi_0})\sqcup (T^{(1)},o_1,\underline{\varphi_1})$, where without loss of generality $T^{(0)}$ and $T^{(1)}$ are connected,
	we find $\gamma_{(T,o,\underline{\varphi})}^A \cong \gamma_{(T^{(0)},o_0,\underline{\varphi_0})}^A \times \gamma_{(T^{(0)},o_0,\underline{\varphi_0})}^A$ in $\cat{M}(I,I)^{\times 2}$ by a canonical isomorphism because $\gamma^A$ is symmetric monoidal. 
	As a consequence,
	\begin{align}
			\Calc_A (T,o,\underline{\varphi})\cong\gamma_{(T^{(0)},o_0,\underline{\varphi_0})}^A \otimes \gamma_{(T^{(1)},o_1,\underline{\varphi_1})}^A = \Calc_A (T^{(0)},o_0,\underline{\varphi_0})\otimes \Calc_A (T^{(1)},o_1,\underline{\varphi_1})
		\end{align}
	by a canonical isomorphism.
	\end{proof}

 Given an $\cat{M}$-valued modular $\mathcal{O}$-algebra $A$,
 the calculus construction assigns to 
an operation $o\in \mathcal{O}(T_n)$ and 
$\varphi \in \cat{M}(I,X^{\otimes (n+1)})$ the object in $\cat{M}(I,I)$ obtained by evaluation of $A_o$ on $\varphi$.

In the case $\cat{M}=\Fin$, we will use the following conventions:
\begin{itemize}
	\item The category $\cat{M}(I,I)=\Fin(\Vect,\Vect)$ is canonically equivalent to $\Vect$ and we will therefore identify $\cat{M}(I,I) $ in this case with $\Vect$. Therefore, the calculus functor will be seen as a $\Vect$-valued functor.
	
	\item Let $\cat{C}\in\Fin$ be the underlying object for our algebra.
	For $o\in \cat{O}(T_n)$, the evaluation  $\Calc_A(T_n,o,-)$ of the calculus on $(T_n,o,-)$ is a functor $\Fin(\Vect,\cat{C}^{\boxtimes (n+1)})\to \Vect$, but through the identification $\Fin(\Vect,\cat{C}^{\boxtimes (n+1)})\simeq  \cat{C}^{\boxtimes (n+1)}$, we agree to see it as functor
	\begin{align}
	\Calc_A(T_n,o,-):\cat{C}^{\boxtimes (n+1)} \to\Vect
	\end{align}
	which after these identifications is just $A_o$.
	In particular, $\Calc_A(T_n,o,-)$
	can be naturally seen as a left exact functor.
	
	\end{itemize}

One of the key properties of the calculus construction that we will need later is its `locality' that we formulate in the following excision result. It crucially relies on the relation between the composition in the endomorphism operad and Lyubashenko's left exact coend that we established earlier in Proposition~\ref{propointcoend}.

\begin{theorem}[Excision]\label{thmexcision}
	Let $A$ be a $\Fin$-valued modular algebra on $(\cat{C},\kappa)$ over a $\Cat$-valued modular operad $\mathcal{O}$, moreover $\Gamma : T \to T'$  a morphism in $\Graphs$, where $T'$ is a connected.
	Then for $o\in \mathcal{O}(T)$ and $o':=\mathcal{O}(\Gamma)o$,
	 we have a canonical natural isomorphism
	\begin{align}
	\Calc_A(T',o',-)\cong\oint^{X_1,\dots,X_r\in\cat{C}} \Calc_A(T,o,\dots,X_j^\kappa,\dots,X_j,\dots)
	\end{align} 
	of functors $\cat{C}^{\boxtimes \Legs(T')}\to\FinVect$, where 
	 $\oint^{X_1,\dots,X_r\in\cat{C}}$ is the left exact coend running over $r$ variables, each one corresponding to an internal edges of $\Gamma$.
	\end{theorem}

The notation on the right hand side was explained in Proposition~\ref{propointcoend}.

\begin{proof}
	By definition and the conventions above we have $\Calc_A(T',o',-)=A_{\mathcal{O}(\Gamma)o}$ as functors $\cat{C}^{\boxtimes \Legs(T')} \to \FinVect$. Thanks to naturality of $A$ up to coherent isomorphism, we find \begin{align} A_{\mathcal{O}(\Gamma)o}\cong \End_\kappa^\cat{C}(\Gamma)A_o \ , \end{align} where $A_o$ is the evaluation of $A$ on $o$.
	By
	Proposition~\ref{propointcoend}
	we have a canonical isomorphism $\End_\kappa^\cat{C}(\Gamma)A_o\cong \oint^{X_1,\dots,X_r\in\cat{C}} A_o^{\boxtimes}(\dots,X_j^\kappa,\dots,X_j,\dots)$.
By combining these facts we obtain the assertion.
	\end{proof}

\section{Applications to quantum topology\label{secapp}}
In this section, we present two applications of our characterization of cyclic algebras.
Besides the calculus construction  from Section~\ref{calculusconstruction}, a key ingredient will be the modular envelope of a cyclic operad \cite{costello}, a concept recalled in Section~\ref{secmodularenvelope}.

	\subsection{A reminder on the modular envelope\label{secmodularenvelope}}
	Any modular operad may be seen as a cyclic operad, i.e.\
	there is a 
 forgetful functor 
	 from modular operads to 
	 cyclic operads.
	 Conversely, one can assign to any cyclic operad the `smallest' modular operad containing it, the 
	 the so-called \emph{modular envelope} \cite{costello}.
	 Formally, the modular envelope of a cyclic operad $\mathcal{O}$ 
	 is obtained via left Kan extension along the inclusion $\ell : \Forests \to \Graphs$.
	This construction can be performed for arbitrary cocomplete target categories.
	Depending on the target category and the extent to which the axioms of a cyclic and modular operad are relaxed, one needs to work with the `homotopically correct' version of the modular envelope.
In our framework of $\Cat$-valued cyclic operads,
the modular envelope of a cyclic operad $\cat{O}$
 evaluated at $T\in \Graphs$ will be given by a relaxed version of a colimit in categories (more precisely, an \emph{oplax colimit} in the most common terminology)
for the functor $\ell/T\to \Forests \ra{\mathcal{O}}\Cat$.
This type of colimit may be modeled by the Grothendieck construction recalled on page~\pageref{pagegrothendieck}, see \cite{hgn} and \cite[Theorem~10.2.3]{johnsonyau}.
Hence,
\begin{align}
(\Envint \mathcal{O})(T) = \int \left(    \ell/T\to \Forests \ra{\mathcal{O}}\Cat   \right) .  
\end{align}	
The subscript `$\int$' is supposed to remind us that, for $\Cat$-valued cyclic operads, we consider a version of the modular envelope defined via the Grothendieck construction.
With this definition,
it follows from Thomason's Theorem \cite[Theorem~1.2]{thomason}
that the topological modular operad 
$|B \Envint \mathcal{O}|$ obtained by applying arity-wise the nerve and the geometric realization of $\Envint\mathcal{O}$ is characterized by the homotopy equivalence
\begin{align}
\hocolimsub{(T_0,\Gamma) \in \ell /T}   |B  \mathcal{O}(T_0)| \ra{\simeq}    |B \Envint \mathcal{O}| (T) = |B \Envint \mathcal{O}(T)| \ ,    \label{eqnrelenv}
\end{align}
i.e.\
it is given by the appropriately derived version of the modular envelope of the 
topological cyclic operad $|B\mathcal{O}|$, as considered in \cite{giansiracusa}.
If we denote the derived modular envelope of topological cyclic operads via $\mathbb{L}\Env$, \eqref{eqnrelenv} gives us a homotopy equivalence
\begin{align}
	\mathbb{L} \Env |B\cat{O}| \ra{\simeq} |B\Envint \cat{O}| \label{equivenvelope}
	\end{align}
of topological modular operads.

\begin{proposition}[Modular extension]\label{propmodext}
	Let $\mathcal{O}$ be a cyclic operad in $\Cat$ and $A: \mathcal{O}\to\End_\kappa^X$ a cyclic $\mathcal{O}$-algebra on an object $X$ in a symmetric monoidal bicategory $\cat{M}$ with non-degenerate symmetric pairing $\kappa : X \otimes X\to I$.
	Then $A$ naturally gives rise to a modular algebra $\Envint \mathcal{O} \to \End_\kappa^X$ over the modular envelope $\Envint \mathcal{O}$ which we denote by $\widehat{A}$ and refer to as the modular extension of the cyclic $\mathcal{O}$-algebra  to a modular $\Envint \mathcal{O}$-algebra.
	\end{proposition}

\begin{proof}
	For $T\in \Graphs$ and $(T^{(0)},\Gamma) \in \ell /T $, the cyclic $\mathcal{O}$-algebra $A$ provides us with maps
	\begin{align}\label{eqnpromaphocolim}	\mathcal{O}(T^{(0)}) \xrightarrow{\ A_{T^{(0)}} \ } (\End_\kappa^X)(T^{(0)})    \xrightarrow{\   (\End_\kappa^ X)(\Gamma)  \     }       (\End_\kappa^X)(T) 
	\end{align} that, by virtue of $A$ being a symmetric monoidal transformation $\mathcal{O}\to\End_\kappa^X$,
	form a  co-cone up to natural transformation
	to the $\ell/T$-shaped diagram sending $(T^{(0)},\Gamma)$ to $\mathcal{O}(T^{(0)})$.
	In more detail, for any morphism $\Omega:(T^{(0)},\Gamma_0)\to (T^{(1)},\Gamma_1)$ in $\ell/T$,
	the diagram
		\begin{equation}\label{diagramtofilleqn}
	\begin{tikzcd}
	\mathcal{O}(T^{(0)}) \ar[dd,"\mathcal{O}(\Omega)", swap] \ar[rr, "A_{T^{(0)}}"] && (\End_\kappa^X)(T^{(0)}) \ar[dd,"(\End_\kappa^X)(\Omega)"]  \ar[drr,"(\End_\kappa^X)(\Gamma_0)"] \\ && &&     (\End_\kappa^X)(T)    \\
\mathcal{O}(T^{(1)})  \ar[rr, "A_{T^{(1)}}",swap] & &  (\End_\kappa^X)(T^{(1)}) \ar[rru,swap,"(\End_\kappa^X)(\Gamma_1)"]
	\end{tikzcd}
	\end{equation}	
	commutes up to canonical natural isomorphism: The natural isomorphism for the square is part of data of $A$, the natural isomorphism for the triangle comes from the functoriality of the endomorphism operad.
	
	As a consequence, the maps \eqref{eqnpromaphocolim} induce a map
	\begin{align}
	(\Envint \mathcal{O})(T) = \int \left(    \ell/T\to \Forests \ra{\mathcal{O}}\Cat   \right)\to \End_\kappa^X(T) 
	\end{align} 
	providing us with a modular $\mathcal{O}$-algebra structure on $(X,\kappa)$.  
	\end{proof}

\begin{remark}\label{remmodext}
	For the maps~\eqref{eqnpromaphocolim} to descend to the Grothendieck construction, it is not needed that the transformations exhibited in
	\eqref{diagramtofilleqn} are actually isomorphisms (a transformation running from top to bottom would have sufficed). Nonetheless, the fact that they actually are isomorphisms will become relevant in the following situation:
	Suppose that in Proposition~\ref{propmodext} the operad $\cat{O}$ is actually groupoid-valued.
	Then the proof above tells us that the modular extension $\widehat{A}$ of any cyclic $\cat{O}$-algebra $A$
	will consist of functors $\Envint \cat{O}(T)\to\End_\kappa^X(T)$ for all $T\in \Graphs$ which send \emph{all} morphisms to isomorphisms. In other words, $\widehat{A}$ descends to a modular $\Pi |B\Envint \cat{O}|$-algebra that by abuse of notation we also denote by $\widehat{A}$.
	Any $\Pi |B\Envint \cat{O}|$-algebra $B$ can be restricted to a cyclic $\cat{O}$-algebra (this would not work for an arbitrary $\Envint \cat{O}$-algebra) that we denote by $B_0$. One can then confirm that by construction we find $\widehat{B_0}\simeq B$. 
	\end{remark}

	\begin{example}
		\label{exenvelopeas}
		The modular envelope of the cyclic associative operad has a description in terms of ribbon graphs as noted by Costello \cite{costello}: 
		For a corolla $T\in \Graphs$, we define the category $\RGraphs(T)$ as the category of connected ribbon graphs whose legs are identified with those of $T$. More precisely, the objects of  $\RGraphs(T)$ are connected graphs $\Gamma$ with an identification $\pi_0(\Gamma)\cong T$ and a cyclic order of the  half edges incident to each vertex. Note that we may see $\Gamma$ as a morphism $\Gamma: \nu(\Gamma) \to T$ in $\Graphs$. Then the cyclic order for $\Gamma$ amounts to a cyclic order of the legs of the disjoint union $\nu(\Gamma)$ of corollas, i.e.\ to an element in the discrete category $\As(\nu(\Gamma))$.
			Morphisms in $\RGraphs(T)$ are contractions of disjoint unions of trees.
		If we denote the category of connected graphs $\Gamma$ with identification $\pi_0(\Gamma)\cong T$ by $\catf{Gr}_{\operatorname{conn}}(T)$, then we can describe
		$\RGraphs(T)$
		formally as the Grothendieck construction
		\begin{align}
		\label{eqngcrgraphs}	\RGraphs(T) = \int \left( \catf{Gr}_{\operatorname{conn}}(T) \ra{\Gamma \mapsto \As(\nu(\Gamma))} \Cat  \right) \ . 
			\end{align}
	 It is straightforward to see that the assignment $T\mapsto \RGraphs(T)$ extends to a modular operad $\RGraphs:\Graphs\to\Cat$, the \emph{modular operad of ribbon graphs}.
	 There is a canonical functor
	 $  \ell / T \to \catf{Gr}_{\operatorname{conn}}(T)$ from 
	 the slice of the inclusion $\ell:\Forests\to \Graphs$ over $T$
	 to
	 $\catf{Gr}_{\operatorname{conn}}(T)$ to  sending $\Gamma : T' \to T$ in $\ell / T$
	  to $\Gamma$.
	  This gives us actually an equivalence
	  \begin{align}\label{eqneequivell} \ell / T \ra{\simeq} \catf{Gr}_{\operatorname{conn}}(T) \end{align}
	   of categories.
	   This follows in a straightforward way from the definition of $\catf{Gr}_{\operatorname{conn}}(T)$ and is also noted in \cite[Section~2.1]{giansiracusa}.
	  If we restrict the Grothendieck construction~\eqref{eqngcrgraphs}
	   along the equivalence~\eqref{eqneequivell},
	   we obtain by definition
	   $\Envint\As(T)$.
	   This is natural in $T$, and hence the equivalences~\eqref{eqneequivell} induce an equivalence \begin{align} \Envint\As \ra{\simeq} \RGraphs \label{eqnmaphandlebody}\end{align}
			from the modular envelope of the associative operad $\As$ to the modular operad of ribbon graphs $\RGraphs$.
		In combination with the calculus construction (Section~\ref{calculusconstruction}), the modular extension (Proposition~\ref{propmodext}) and Theorem~\ref{thmcyclas},
		this yields for any  pivotal Grothendieck-Verdier category $\cat{C}$ in $\Fin$ with associated non-degenerate symmetric pairing $\kappa:\cat{C}\boxtimes\cat{C}\to\FinVect$  a symmetric monoidal functor
			\begin{align}
			\Calc_{\widehat{\cat{C}}} :
			\RGraphs \star_\kappa \widehat{\cat{C}} \to \Vect \ , 
			\end{align}
			i.e.\
			$\cat{C}$ gives rise to a symmetric monoidal functor from the category of ribbon graphs with $\cat{C}$-labeled legs to vector spaces. 	
		To any object $(T,\Gamma,X) \in \RGraphs \star_\kappa \widehat{\cat{C}}$, i.e.\
		\begin{itemize}\item
		an object $T \in \Graphs$ (without loss of generality, we assume that $T$ is connected, i.e.\ a corolla), \item a ribbon graph $\Gamma \in \RGraphs(T)$, 
		which we can see as a morphism $\Gamma : T'\to T$ in $\Graphs$, \item 
		and an object $X\in\cat{C}^{\boxtimes \Legs(T)}$ (to be thought of as a $\cat{C}$-label for each of the legs of $\pi_0(\Gamma)\cong T$),
		\end{itemize}
	the calculus construction assigns a vector space $\Calc_{\widehat{\cat{C}}}(T,\Gamma,X)$. 
	In order to discuss an example, let us assume that $\Gamma$ is a bouquet of $r$ circles 
	with $n+1$ legs attached to the base vertex.
		For any order of $\Legs(T)$, i.e.\ an identification of $T$ with $T_n$, and the associated order 
		$X_0\boxtimes\dots\boxtimes X_n \in \cat{C}^{\boxtimes (n+1)}$
		of $X$ (which without loss of generality is treated as a pure tensor here), there is a canonical isomorphism
		\begin{align}  \Calc_{\widehat{\cat{C}}}(T,\Gamma,X)\cong
		\cat{C}\left(K,X_0 \otimes\dots\otimes X_n\otimes \mathbb{F}^{\otimes      r } \right)
	 \quad \text{with the coend}\quad \mathbb{F}=\int^{X\in\cat{C}} X\otimes X^\kappa \ . \label{eqnorderiso} \end{align}
	When it comes to how exactly the  coend $\mathbb{F}=\int^{X\in\cat{C}} X\otimes X^\kappa$
	has to be understood, a little care is required: It is defined as the image of 
	the coend $\Delta=\int^{X \in \cat{C}} X \boxtimes X^\kappa$ under
	the monoidal product.
		The isomorphism~\eqref{eqnorderiso} follows from the definition of the modular extension in Proposition~\ref{propmodext} (in particular \eqref{eqnpromaphocolim}) and
		excision (Theorem~\ref{thmexcision}).
		Note that this isomorphism is canonical only after the choice of the order.
		The left hand side $\Calc_{\widehat{\cat{C}}}(T,\Gamma,X)$ is defined independently of this choice (this is a strength of the calculus construction). 
		\end{example}

	\subsection{Application I: Handlebody group representation from  ribbon Groth\-en\-dieck-Verdier structures\label{app-hbdy}}
	As a first application, we will prove that a ribbon Grothendieck-Verdier structure gives rise to explicitly computable handlebody group representations. 
	The idea is to combine our characterization of cyclic framed little disks algebras in Theorem~\ref{thmcycle2}
	and Theorem~\ref{thmalgrbr} with the relation between the derived modular envelope of $\framed$ and the modular operad of handlebodies found by Giansiracusa \cite{giansiracusa} (this statement refers to topological operads).
	 
The \emph{(groupoid-valued) modular handlebody operad} is the symmetric monoidal functor $\hbdy : \Graphs \to \Grpd$ which assigns to a corolla $T$ the groupoid $\hbdy(T)$ defined as follows: Objects are compact connected oriented handlebodies $H$ (hereafter just referred to as handlebodies for brevity) with $|\Legs(T)|$ many disks embedded in $\partial H$ with a parametrization of the  disks, i.e.\ an orientation-preserving embedding $\psi : \sqcup_{\Legs(T)} \mathbb{D}^2 \to H$ which is an orientation-preserving diffeomorphism to the disks embedded in $\partial H$. 
Morphisms in $\hbdy(T)$ are isotopy classes of orientation-preserving diffeomorphisms that respect the disk parametrizations (hence, the automorphism groups are precisely the handlebody groups). 
Operadic composition is by gluing of handlebodies along their boundaries, see \cite[Section~4.3]{giansiracusa} for the details.
Note however that our handlebody operad allows all handlebodies while Giansiracusa only considers those handlebodies $H_{g,n}$ with genus $g$ and $n$ embedded disks for which $(g,n)\neq (0,0),(0,1)$. 
In order to distinguish both versions, we denote Giansiracusa's (groupoid-valued) handlebody operad by $\hbdyg$, where the `a' indicates that we restrict to the handlebodies allowed in \cite{giansiracusa}. It comes with an inclusion $\hbdyg\subset \hbdy$ of modular operads. 
By taking entry-wise the classifying spaces, we may see $\hbdy$ and $\hbdyg$ as topological operads that we denote by $\Hbdy$ and $\Hbdyg$. We may identify $\Pi \Hbdy$ and $\Pi \Hbdyg$ with $\hbdy$ and $\hbdyg$, respectively; in formulae
\begin{align}
\Pi \Hbdy \simeq \hbdy \ , \quad \Pi \Hbdyg \simeq \hbdyg \ . 
\end{align}
With a subscript `0', we will indicate the restriction to handlebodies of genus zero. It is straightforward to observe that the restriction $\hbdy_0$ of $\hbdy$ to genus zero is a cyclic operad. Similarly, $\Hbdy_0$, $\hbdyg_0$ and $\Hbdyg_0$ are cyclic operads. 
We will use the following crucial result on the modular envelope of $ \hbdyg_0$. It follows from
the results of \cite{giansiracusa}
and~\eqref{equivenvelope}:

	\begin{theorem}[{Giansiracusa \cite[Theorem~A]{giansiracusa}}]\label{thmgiansiracusa}
	There is a canonical map of $\Cat$-valued operads
	\begin{align} \Envint \hbdyg_0 \to  \hbdyg \label{eqnmaphandlebody}\end{align}
	which is arity-wise an isomorphism on $\pi_0$
	and which when evaluated on
	$T \in \Graphs$ induces a homotopy equivalence after taking nerve and geometric realization if $T$ has at least one leg.
	If $T$ has no legs, this remains true except for the path component corresponding to the  solid closed torus.
\end{theorem}

We refer to Remark~\ref{remtorus} for a comment on the exception occurring on the component for the solid closed torus.

In \cite{giansiracusa} the Theorem is proven using a version of cyclic operads without arity zero operations. When applying this result in our context we hence have to ignore the arity zero operations in $\framed$. This is also the reason why the solid three-dimensional ball has to be excluded.

\begin{remark}\label{remgiansiracusa}
By Theorem~\ref{Thm: Relation ribbon braids and E2}
we obtain a decomposition
\begin{align}
\Envint \hbdyg_0 (T_{n-1}) =\Envint\Pi \Hbdyg_0(T_{n-1}) = \bigsqcup_{g\in \mathbb{N}_0} M(g,n) \quad \text{for}\quad n\ge 0 \ , 	\end{align}
where $M(g,n)$ is a connected category.
The reason why $T_{n-1}$ appears on the left hand side while we use the index $n$ on the right hand side is that $T_{n-1}$ by definition has $n$ legs (with the convention that $T_{-1}$ is the corolla without legs).
If $(g,n)\neq (1,0)$, 
$|BM(g,n)|$ is equivalent to the classifying space of 
the mapping class group $\Map(H_{g,n})$ of the handlebody $H_{g,n}$ with genus $g$ and $n$ embedded disks. 
In the sequel, we will denote a base point of $M(g,n)$ by $o_{g,n}$.
\end{remark}
If we replace in the definition of $\hbdy$ the handlebodies with surfaces and, consequently, isotopy classes of orientation-preserving diffeomorphisms of handlebodies by isotopy classes of orientation-preserving diffeomorphisms of surfaces, 
we obtain the \emph{(groupoid-valued) modular operad of surfaces} $\surf : \Graphs \to \Grpd$ and, by taking classifying spaces, its topological counterpart $\Surf$. 
The automorphism groups of $\surf$ are precisely the mapping class groups of oriented surfaces (with boundary). 
Since the handlebody group $\Map(H)$ of a handlebody $H$ may be identified with the subgroup of the mapping class group $\Map(\partial H)$ of those isotopy classes of orientation-preserving diffeomorphisms of $\partial H$ that extend to all of $H$, there are maps $\hbdy \subset \surf$ and $\Hbdy \subset \Surf$ of modular operads that induce equivalences of cyclic operads $\hbdy_0 \simeq \surf_0$ and $\Hbdy_0\simeq \Surf_0$ after restriction to genus zero. This follows from the well-known statement that the handlebody subgroups agree with the entire mapping class group for genus zero surfaces, see e.g.\ \cite[Proposition~2.1]{hahe}. 
	
		The topological cyclic operad $\szero$ from Section~\ref{secszeroframed} can be identified in a straightforward way with $\Surf_0$. We may therefore conclude from Proposition~\ref{Thm: Relation ribbon braids and E2}:

	\begin{lemma}\label{lemmaequivcyclicoperadsrbretc}
		There are equivalences of cyclic $\Grpd$-valued operads $ \hbdy_0\simeq  \surf_0 \simeq \Pi \framed \simeq \RBr$.
		\end{lemma}
	
	Together with Theorem~\ref{thmcycle2} this observation implies:
	
	\begin{corollary}
		A $\Fin$-valued cyclic algebra over $\hbdy_0$ or  $\surf_0$ can be equivalently described as a ribbon Grothendieck-Verdier category in $\Fin$. 
		\end{corollary}

	If we are given a ribbon Grothendieck-Verdier category $\cat{C}$ in $\Fin$, we may --- by this result --- see it as a cyclic $ \hbdy_0$-algebra and restrict it to a cyclic $\hbdyg_0$-algebra that we denote by $\cat{C}^\text{\normalfont a}$. Its modular extension $\widehat{\cat{C}^\text{\normalfont a}}$ in the sense of Proposition~\ref{propmodext} is a modular algebra over the  modular envelope $\Envint\hbdyg_0$. Recall from Proposition~\ref{propcalculusfunctor} that this modular algebra comes with a calculus functor, i.e.\
	a symmetric monoidal functor $ \Calc_{\widehat{\cat{C}^\text{\normalfont a}}}: \Envint \hbdyg_0 \star_\kappa \cat{C} \to \FinVect$, where $\kappa : \cat{C}\boxtimes\cat{C}\to\FinVect$ is the pairing that $\cat{C}$ comes equipped with.
For this calculus functor, we can prove the following statement that we rephrase afterwards in more concrete terms:

	\begin{theorem}\label{thmmainbalancedbraidedcalc1}
		Let $\cat{C}$ be a ribbon Grothendieck-Verdier category in $\Fin$ and $\kappa:\cat{C}\boxtimes\cat{C}\to\FinVect$ the associated non-degenerate symmetric pairing.
		Then the calculus functor of the modular $\Envint\hbdyg_0$-algebra $\widehat{\cat{C}^\text{\normalfont a}}$ is a symmetric monoidal functor
		\begin{align}\label{eqnthecalcfunctor} \Calc_{\widehat{\cat{C}^\text{\normalfont a}}}: \Envint  \hbdyg_0 \star_\kappa \cat{C} \to \FinVect \end{align}
		that we may explicitly describe as follows:
		 After the choice of an order for the $n$ legs of $T_{n-1}$, there is an isomorphism
		\begin{align}
	\label{eqnformulaforcalc}	\Calc_{\widehat{\cat{C}^\text{\normalfont a}}}  (T_{n-1},o_{g,n},X_1,\dots,X_n) \cong \cat{C}(K,X_1\otimes\dots\otimes X_n \otimes \mathbb{F}^{\otimes g}) \quad \text{for all}\quad X_1,\dots,X_n\in\cat{C} \ , 
		\end{align}
		where \begin{itemize}
			\item the $o_{g,n}$ for non-negative integers $g$ and $n$ are the chosen base points in the components of $\Envint\hbdyg_0(T_{n-1})$, see Remark~\ref{remgiansiracusa}, \item and $\mathbb{F}=\int^{X\in\cat{C}} X\otimes X^{\kappa} \in \cat{C}$ is 
			defined by applying the monoidal product to the coevaluation object $\Delta = \int^{X \in \cat{C}} X \otimes X^\kappa$, i.e.\ $\mathbb{F}=\otimes (\Delta)$.
			\end{itemize}
		\end{theorem}

	 \begin{proof}
	 	As explained before the statement of the result, we obtain the symmetric monoidal functor as a direct consequence of our characterization
	 	of cyclic $\RBr$-algebras and the calculus construction. 
	 	It remains to prove~\eqref{eqnformulaforcalc}:
	 	For $g,n\ge 0$, consider the graph $\Gamma_{g,n}$ with one vertex, $n$ legs and $g$ internal edges. 
	 	Then $\Gamma_{g,n} : T_{n-1+2g}\to T_{n-1}$ is a morphism in $\Graphs$ with $\Envint\hbdyg_0(\Gamma_{g,n})o_{0,n+2g}\cong o_{g,n}$ by Theorem~\ref{thmgiansiracusa} (on level of $\pi_0$) and Remark~\ref{remgiansiracusa}. 
	 	Since by definition
	 \begin{align}	\Calc_{\widehat{\cat{C}^\text{\normalfont a}}}  (T_{n-1+2g},o_{0,n+2g},X_1,\dots,X_n,Y_1,\dots,Y_{2g}) \cong \cat{C}(K,X_1\otimes\dots\otimes X_n\otimes Y_1 \otimes \dots \otimes Y_{2g}) \ , 
	 \end{align}
	 we can conclude from the Excision Theorem~\ref{thmexcision}
	 \begin{align}
	 \Calc_{\widehat{\cat{C}^\text{\normalfont a}}}  (T_{n-1},o_{g,n},X_1,\dots,X_n)\cong \oint^{Y_1,\dots,Y_g \in \cat{C}} \cat{C}(K,X_1\otimes\dots\otimes X_n\otimes Y_1 \otimes Y_1^\kappa \otimes \dots\otimes  Y_g \otimes Y_{g}^\kappa) \ . 
	 \end{align}
	 Now \eqref{eqnformulaforcalc} follows from Lemma~\ref{lemmaleftexactcoend} which allows us to express the left exact coend by means of  $\mathbb{F}=\int^{X\in\cat{C}} X\otimes X^{\kappa} \in \cat{C}$.
	 	\end{proof}

A less concise, but more explicit version of this result reads as follows:

\begin{theorem}\label{thmmainbalancedbraidedcalc2}
	Let $\cat{C}$ be a ribbon Grothendieck-Verdier category in $\Fin$. Then for $(g,n)\in\mathbb{N}_0^2$ and $X_1,\dots,X_n\in\cat{C}$, the morphism space $\cat{C}(K,X_1\otimes\dots\otimes X_n \otimes \mathbb{F}^{\otimes g})$ comes with a canonical action of the handlebody group $\Map(H_{g,n})$ for the handlebody $H_{g,n}$ of genus $g$ and $n$ boundary disks if $(g,n)\neq (1,0)$. 
	\end{theorem}

\begin{proof}
	Since $\Map(H_{g,n})=1$ for $(g,n)=(0,0)$ and $(g,n)=(0,1)$, these two cases are trivial.
	
	Let now $(g,n)\neq (1,0),(0,1),(0,0)$.
	 Theorem~\ref{thmmainbalancedbraidedcalc1}, when combined with the definition of the calculus construction,
	provides for us functors
	$M(g,n)\to\Vect$ sending $o_{g,n}$ to the vector space 
	$\Calc_{\widehat{\cat{C}^\text{\normalfont a}}}  (T_{n-1},o_{g,n},X_1,\dots,X_n)$. 
	The category $M(g,n)$ was defined in Remark~\ref{remgiansiracusa} and has the property that $|BM(G,n)|$ is a classifying space of $\Map(H_{g,n})$.
	From the fact that $\hbdyg$ is groupoid-valued 
	and Remark~\ref{remmodext},
	it follows that this functor $M(g,n)\to \Vect$ 
	sends \emph{all} morphisms to isomorphisms.
	As a consequence, the functor $M(g,n)\to \Vect$ descends to the localization at all morphisms. Hence, $\Calc_{\widehat{\cat{C}^\text{\normalfont a}}}  (T_{n-1},o_{g,n},X_1,\dots,X_n)$ inherits an action of $\Map(H_{g,n})$.
	Now the assertion follows from~\eqref{eqnformulaforcalc}.
\end{proof}

\begin{remark}\label{remcaveatorder}
	As a small caveat, we should mention that the statement that $\cat{C}(K,X_1\otimes\dots\otimes X_n \otimes \mathbb{F}^{\otimes g})$ carries an action of the handlebody group contains a small abuse of language and must be interpreted correctly:
	In the first place,
	the value of the calculus functor carries the handlebody group representation. It is then transferred to the morphism space 
	$\cat{C}(K,X_1\otimes\dots\otimes X_n \otimes \mathbb{F}^{\otimes g})$ by means of \eqref{eqnformulaforcalc} after \emph{ordering} the boundary components (this is the same choice discussed in Example~\ref{exenvelopeas}).
 When working with the calculus functor, we do not have to make this choice (there the labels are really attached to the boundaries and never numbered). We still use in Theorem~\ref{thmmainbalancedbraidedcalc2} the formulation through the morphism spaces and a numbering even if it contains this abuse of notation. This is in order to make contact to the literature where this abuse of notation seems to be standard (after all, it is not very problematic if one is aware of the subtlety).
	\end{remark}

\begin{remark}\label{remtorus}
	The exception made in Theorem~\ref{thmmainbalancedbraidedcalc2} for the solid closed torus
	is a consequence of the corres\-ponding exception appearing in Giansiracusa's result.
	This issue arises from the non-contractibility of the disk complex for the torus as explained in \cite[Section~6.2]{giansiracusa}. 
	Nonetheless, the calculus functor for a ribbon Grothendieck-Verdier category $\cat{C}$ actually can be evaluated on the solid closed torus, where it  yields the vector space $\cat{C}(K,\mathbb{F})$, see Theorem~\ref{thmmainbalancedbraidedcalc1}.
	This vector space comes with additional structure from the component of the modular envelope attached to the solid closed torus. The investigation of this structure is beyond the scope of this article, see however the statements that can be made in the modular case below in Proposition~\ref{propcomplyu}.
	\end{remark}

\begin{example}\label{exwithoutrigidity}
	Let $G$ be an Abelian group and ${\FinVect_G}^{\omega,g_0}$ the ribbon Grothendieck-Verdier category associated 
	to an Abelian 3-cocycle $\omega$ on $G$ with coefficients in $\mathbb{C}^\times$ and duality $D_{g_0} = \mathbb{C}_{g_0} \otimes (-)^*$ with $g_0=h_0^{-2}$ for some $h_0\in G$, see Example~\ref{exampleGVnonrigid}.
	For this \emph{semisimple}
	category, we can conclude by arguments similar to those for \cite[Corollary~5.1.9]{kl} that the coend $\mathbb{F}$ (as defined in Theorem~\ref{thmmainbalancedbraidedcalc1}) is given by
	 \begin{align}
	 \mathbb{F}\cong \bigoplus _{g\in G} \mathbb{C}_g \otimes \mathbb{C}_{g_0g^{-1}} \cong \bigoplus_{g \in G} \mathbb{C}_{g_0}=\mathbb{C}[G]\otimes\mathbb{C}_{g_0} \ ,
	 \end{align}
	 where $\mathbb{C}[G]\otimes\mathbb{C}_{g_0}$ is the tensoring of $\mathbb{C}_{g_0}$ with the free $\mathbb{C}$-vector space on the set $G$.
	 This implies $\mathbb{F}^{\otimes \ell}\cong \C[G]^{\otimes \ell} \otimes \mathbb{C}_{g_0^\ell}$ for $\ell \ge 0$.
	 As a consequence, the vector space associated by Theorem~\ref{thmmainbalancedbraidedcalc1} to a (for simplicity closed) surface of genus $\ell\ge 1$ is
	 \begin{align}
	 {\FinVect_G}^{\omega,g_0}(\mathbb{C}_{g_0} ,      (\mathbb{C}[G]\otimes\mathbb{C}_{g_0})^{\otimes \ell}       )\cong \C[G]^{\otimes \ell} \delta_{g_0,g_0^\ell} \ . 
	 \end{align}
	For $\ell\ge 2$, the handlebody group representation on this vector space can be computed explicitly using excision:
	For example, it follows from the definition of the balancing in Example~\ref{exampleGVnonrigid} that
	the Dehn twist around the $m$-th handle, $1\le m\le \ell$
	acts as the linear automorphism $\C[G]^{\otimes \ell} \delta_{g_0,g_0^\ell}$
	which acts as the identity map on all tensor factors $\C[G]$ except for the $m$-th one where it is given by the map
	\begin{align} \C[G]\to \C[G]\ , \quad g \mapsto \frac{q(g h_0)}{q(h_0)} \cdot g \ ; 
	\end{align}
	here $q:G\to\mathbb{C}^\times$ is the quadratic form associated to $\omega$.
\end{example}

In Example~\ref{exwithoutrigidity}, a category whose Grothendieck-Verdier structure does not come from actual rigidity was covered.
But of course, Theorem~\ref{thmmainbalancedbraidedcalc2} applies in particular in the rigid case.  In order to exploit this, let us recall some more terminology:
A \emph{finite tensor category} \cite{etingofostrik} is
a finite category
(as defined in Example~\ref{exsymmoncat})
 with a rigid monoidal structure and simple unit. 
A \emph{finite ribbon category} is a finite braided tensor category equipped with a balancing compatible with the duality.
By Theorem~\ref{thmcycle2} this is precisely a $\Fin$-valued cyclic $\RBr$-algebra with simple 
unit whose Grothendieck-Verdier structure comes from rigidity.

One way to obtain a finite ribbon category
 is by taking categories of finite-dimensional modules 
 over a finite-dimensional ribbon Hopf algebra $A$, 
 see e.g.\ \cite[XIV.6]{kassel}. 
 In this case, the coend $\mathbb{F}$ is isomorphic to $A_\text{coadj}^*$ \cite[Theorem~7.4.13]{kl}, 
 i.e.\ the dual $A^*$ of $A$ with the coadjoint action of $A$ given by
\begin{align}
	A\otimes A^* \to A^*, \quad a \otimes \alpha \to (b \mapsto \alpha (S(a_{(1)})  ba_{(2)}  )) \ , 
\end{align}
where $S$ is the antipode of $A$ and $\Delta a = a_{(1)} \otimes a_{(2)}$ the Sweedler notation for the coproduct. 
Now Theorem~\ref{thmmainbalancedbraidedcalc2} specializes to:

\begin{corollary}\label{corhopf}
	Let $A$ be a finite-dimensional ribbon Hopf algebra.
	Then for any non-negative integers $g$ and $n$ with $(g,n)\neq (1,0)$ 
	and any finite-dimensional $A$-modules $X_1,\dots,X_n$,
	the vector space 	
	\begin{align} \Hom_A\left( k ,  X_1\otimes \dots \otimes X_n \otimes \left(   A_\text{coadj}^*   \right)^{\otimes g}  \right)
		\end{align}
	of $A$-invariants of the module 
	$X_1\otimes \dots \otimes X_n \otimes \left(   A_\text{coadj}^*   \right)^{\otimes g}$
	comes canonically with an action of the 
	mapping class group of the handlebody with genus $g$ and $n$ boundary components.
	\end{corollary}

	The handlebody group representations from Theorem~\ref{thmmainbalancedbraidedcalc2} (and in particular Corollary~\ref{corhopf}) when given in this generality (note in particular that no non-degeneracy of the braiding is assumed) are new to the best of our knowledge.
	However, under much stronger assumptions on the category $\cat{C}$, namely \emph{modularity} (to be defined momentarily), they relate to the Lyubashenko construction \cite{lyubacmp,lyu,lyulex} as we will explain now:
For a finite braided tensor category $\cat{C}$, we may define the \emph{Müger center} \cite{mueger}, i.e.\ the full subcategory of $\cat{C}$ spanned by all objects $X\in\cat{C}$ such that $c_{Y,X}c_{X,Y}=\id_{X\otimes Y}$ for all $Y\in \cat{C}$ (these objects are called \emph{transparent}). 
A finite braided tensor category whose Müger center is trivial in the sense that it is generated under finite direct sums by the monoidal unit is called \emph{non-degenerate}. 
Recently, there has been significant progress in the understanding of non-degeneracy through the equivalent characterizations given in \cite{shimizumodular} and the factorization homology approach in \cite{bjss}. 
A \emph{modular category} is a finite ribbon category whose underlying finite braided category is non-degenerate (it is important to remark that this definition does \emph{not} include semisimplicity). 
Modular categories are absolutely central objects living at the intersection of conformal field theory, topological field theory and representation theory.
A thorough discussion of modular categories is beyond the scope of this article; a biased and by no means exhaustive list of references is \cite{rt1,baki,huang, kl, turaev,BDSPV15, jfcs}. 

By the Lyubashenko construction \cite{lyubacmp,lyu,lyulex} a modular category $\cat{C}$ gives rise to a consistent system of projective mapping class group representations on the morphism spaces $\cat{C}(I,X_1\otimes\dots\otimes X_n \otimes \mathbb{F}^{\otimes g})$, see also \cite{jfcs} for a perspective on these representations through the Lego-Teichmüller game and \cite{dmf} for a homotopy coherent perspective. 
The vector spaces $\cat{C}(I,X_1\otimes\dots\otimes X_n \otimes \mathbb{F}^{\otimes g})$ together with their mapping class group actions
are often referred to as \emph{conformal blocks}. 
These mapping class group representations can be restricted to the handlebody part (then they will be non-projective, but actually linear).
On the other hand, any modular category is in particular a ribbon Grothendieck-Verdier category. Hence, we also have handlebody group representations by Theorem~\ref{thmmainbalancedbraidedcalc2}. The following result is a comparison:

	\begin{proposition}\label{propcomplyu}
	Let $\cat{C}$ be a modular category. Then the following two actions of $\Map(H_{g,n})$ on $\cat{C}(I,X_1\otimes\dots\otimes X_n\otimes \mathbb{F}^{\otimes g})$ are equivalent:
	\begin{itemize}
	
	\item The action from Theorem~\ref{thmmainbalancedbraidedcalc2} (it extends in this case also to the solid closed torus).
	 
	\item The restriction of the projective Lyubashenko mapping class group action to the handlebody part.
	
	\end{itemize}
	\end{proposition}

	We could give the proof already here, but it can be formulated much more concisely using the terminology of the next section. Therefore, we defer the proof to page~\pageref{hereistheproofforproplyu}.

	Proposition~\ref{propcomplyu} has the following significance: The Lyubashenko construction is to a large extent an algebraic construction which starts from the vector spaces $\cat{C}(K,X_1\otimes\dots\otimes X_n \otimes \mathbb{F}^{\otimes g})$ and establishes the corresponding mapping class group actions by a presentation of mapping class groups in terms of generators and relations. 
	Proposition~\ref{propcomplyu} now tells us that the vector spaces $\cat{C}(K,X_1\otimes\dots\otimes X_n \otimes \mathbb{F}^{\otimes g})$ and at least the handlebody part of the actions do not only have an intrinsically  topological description, but in fact also a universal property coming from the  modular envelope of the cyclic operad of genus zero surfaces.

	\subsection{Application II: Grothendieck-Verdier duality for the evaluation  of a modular functor on the circle\label{app-mf}}
	In the preceding subsection,
	it was already mentioned that the Lyubashenko construction does not only yield handlebody group representations, but actually projective mapping class group representations --- they form a structure that is commonly referred to as a \emph{modular functor}.
	While modular functors are certainly not the main object of study for the present article, 
	we may still use our results to prove a duality statement for modular functors.
	
	A modular functor \cite{turaev,tillmann,baki} is, roughly speaking, a consistent system of projective mapping class group representations.
	Many variants of this notion exist; we will momentarily present the version used in this article and briefly comment on the relation to other definitions.
	
	As already mentioned, modular functors feature certain \emph{projective} mapping class representations, and this projectivity is described by considering certain central extensions of mapping class groups.
	The relevant central extensions are obtained by means of 2-cocycles on the mapping class groups arising from the framing anomaly \cite{atiyahframing,gilmermasbaum}. 
	In the language of modular operads, this can be formulated as follows:
	For a corolla $T$, we define the groupoid $\surfc(T)$. Objects are compact oriented surfaces with boundary, a parametrization $\sqcup_{\Legs(T)} \mathbb{S}^1 \to \Sigma$ of the boundary $\partial \Sigma$ and a maximal isotropic subspace $\lambda$ of the presymplectic vector space $H_1(\Sigma;\mathbb{Q})$ with respect to the intersection pairing $H_1(\Sigma;\mathbb{Q})\otimes H_1(\Sigma;\mathbb{Q})\to\mathbb{Q}$ (if $\Sigma$ is closed, $H_1(\Sigma;\mathbb{Q})$ is symplectic, and a maximal isotropic subspace is precisely a Lagrangian subspace). A morphism $(\Sigma,\lambda)\to(\Sigma',\lambda')$ is a pair $(\phi,n)$ of an isotopy class of orientation-preserving diffeomorphisms $\phi : \Sigma\to\Sigma'$ compatible with the boundary parametrizations, and a weight $n\in\mathbb{Z}$. The composition of composable morphisms $(\Sigma,\lambda) \ra{(\phi_0,n_0)} (\Sigma',\lambda')\ra{(\phi_1,n_1)} (\Sigma'',\lambda'')$ is the pair $(\phi_1\phi_0,n_0+n_1+\mu(   \lambda'' , {\phi_1}_* \lambda' , {\phi_1\phi_0}_* \lambda       ))$, where $\mu(-,-,-)$ is the Maslov index of three maximal isotropic subspaces of $H_1(\Sigma'';\mathbb{Q})$, see \cite[IV.3.5]{turaev} for a detailed definition. 
	To a morphism $\Gamma : T \to T'$ in $\Graphs$, we may associate a functor $\surfc(\Gamma):\surfc(T)\to\surfc(T')$ as follows: For $(\Sigma,\lambda)\in \surfc(T)$, the graph $\Gamma$ prescribes a gluing of $\Sigma$ along those boundary components attached to legs arising from internal edges of $\Gamma$. We denote the glued surface by $\Sigma^\Gamma$ and the quotient map by $q^\Gamma : \Sigma \to \Sigma_\Gamma$. Having established this notation, we set $\surfc(\Gamma)(\Sigma,\lambda):=(\Sigma^\Gamma,q^\Gamma_* \lambda)$. 
	This way,
	we obtain the \emph{(groupoid-valued) central extension of the surface operad} $\surfc: \Graphs \to \Grpd$, a modular operad that comes with a epimorphism $\surfc\to\surf$. We denote by $\Surfc$ the associated topological modular operad. 
The well-known fact that the cocycle describing the framing anomaly vanishes on the handlebody subgroups of the mapping class groups can be phrased as follows in terms of modular operads:
	
\begin{lemma}\label{lemmainclhbdyc}
	The map $\hbdy \to \surf$ of modular operads canonically lifts to a map $\hbdy \to \surfc$ of modular operads. 
	\end{lemma}

\begin{proof}
	The desired map $\hbdy \to \surfc$ sends a handlebody $H \in \hbdy(T)$ to its boundary surface $\Sigma_H\in  \surf(T)$ plus the maximal isotropic subspace $\lambda(H):= \ker \left(   H_1(\Sigma_H;\mathbb{Q} ) \to H_1(H;\mathbb{Q})    \right)$. A morphism $\phi : H\to H'$ in $\hbdy(T)$ is sent to the induced map $\Sigma_\phi : \Sigma_H\to\Sigma_{H'}$ and weight zero. The map $\Sigma_\phi$ then sends $\lambda(H)$ to $\lambda(H')$. For composable morphisms $H\ra{\phi_0}H'\ra{\phi_1}H''$ in $\hbdy(T)$, we observe
	\begin{align} (\Sigma_{\phi_1},0)\circ(\Sigma_{\phi_0},0)= (\Sigma_{\phi_1\phi_0} ,   \mu (  \lambda(H''), {\Sigma_{\phi_1}} _*   \lambda(H'), {\Sigma_{\phi_1\phi_0}} _*   \lambda(H)     )   ) = (\Sigma_{\phi_1\phi_0}, \mu (   \lambda(H''),\lambda(H''),\lambda(H'')  ))
	\end{align} But the Maslov index $\mu (   \lambda(H''),\lambda(H''),\lambda(H'')  )$ is zero by equation (3.5.a) in \cite[IV.3.5]{turaev}. This proves that the assignments actually yield a functor $\hbdy(T)\to\surfc(T)$. It can be easily seen to also be a map of modular operads. 
	\end{proof}

The modular operad $\surfc$ allows us to give a concise definition of the notion of a modular functor:

	 \begin{definition}\label{defmodularfunctor}
	 	A \emph{($\Fin$-valued) modular functor} is a modular $\surfc$-algebra in $\Fin$. 
	 	\end{definition}
 	
 	A modular functor, according to the above definition has an underlying category $\cat{C}\in\Fin$, which we think of as assigned to the circle. The structure of a $\surfc$-algebra on $\cat{C}\in\Fin$ assigns to surfaces with $p$ incoming and $q$ outgoing boundary components a map $\cat{C}^{\boxtimes p}\boxtimes\cat{C}^{\boxtimes q}\to \FinVect$ in $\Fin$
 	 (which can be seen as a map $\cat{C}^{\boxtimes p}\to     \cat{C}^{\boxtimes q}$ by duality). Mapping classes of the surface translate to natural isomorphisms of this functor. The gluing of surfaces translates to left exact coends.
 	 This explains why our definition is in line with the ones given in \cite{jfcs,dmf}.
 	 It is more general than the notion from \cite{turaev,baki} that additionally builds in semisimplicity, simplicity of the unit and a normalization axiom (these additional assumptions are not really topological in the sense that they are not part of the modular operad $\surfc$; 
 	 we comment very briefly on how to include them in Corollary~\ref{corrcategory} below). 
 	 On the topological side, i.e.\ as far as the definition of $\surfc$ is concerned, Definition~\ref{defmodularfunctor}
 	 is also essentially in line with \cite{tillmann}. Here, however, the key difference is that in \cite{tillmann} a different target category of linear categories is considered. This choice ultimately leads again to semisimplicity.

	By definition a modular functor has an underlying category $\cat{C}\in\Fin$.
	The structure of a modular functor endows this value on the circle with more structure about which we can make, using our previous results, the following statement:
	
	\begin{theorem}\label{thmcirclesectormf}
		The category obtained by evaluation  of a $\Fin$-valued modular functor 
		on the circle
		naturally comes with a ribbon Grothendieck-Verdier structure.
		\end{theorem}
	
	\begin{proof}
		Any modular functor is by definition a $\surfc$-algebra and hence can be restricted along the map $\hbdy_0\to \surfc_0$ obtained from Lemma~\ref{lemmainclhbdyc} after restriction to genus zero. This proves that its evaluation on the circle comes with the structure of a cyclic $\hbdy_0$-algebra. 
			By  Lemma~\ref{lemmaequivcyclicoperadsrbretc} this is a cyclic $\RBr$-structure which amounts to a  
			ribbon Grothendieck-Verdier structure by Theorem~\ref{thmcycle2}.
		\end{proof}

	When imposing stronger assumptions on the value of a modular functor on the circle (either directly or indirectly by choice of a different target category), we recover the following result which --- in a slightly different language --- is part of \cite[Section~3]{tillmann} and \cite[Theorem~5.7.10]{baki}:

	\begin{corollary}\label{corrcategory}
		Consider a $\Fin$-valued modular functor whose evaluation $\cat{C}$
		on the circle
		 is semisimple and has a simple monoidal unit.
		Under these assumptions, the ribbon Grothendieck-Verdier structure on $\cat{C}$ from Theorem~\ref{thmcirclesectormf}
		 is a ribbon r-structure if and only if
		the modular functor is normalized  in the sense  
		 that its value on the sphere is a one-dimensional vector space. 
		\end{corollary}

	\begin{proof}
		By Theorem~\ref{thmmainbalancedbraidedcalc1} we know that the vector space that the modular functor assigns to the sphere is the morphism space $\cat{C}(K,I)$.
		Hence, it remains to prove
		\begin{align}
		\dim \cat{C}(K,I) = 1 \quad \Longleftrightarrow\quad K\cong I \ .  \label{eqnhomKI}
		\end{align}
	To this end, recall that $K$ is the image of $I$ under the duality functor $D$, which is an anti-equivalence. Since $I$ is simple, so is $K$. By semisimplicity of $\cat{C}$, we have
	\begin{align}
	\cat{C}(K,I)\cong \left\{ \begin{array}{cl} k \id_I \ , & \text{if}\ K\cong I \ , \\ 0 \ , & \text{else} \ .\end{array} \right.
	\end{align}
	This implies \eqref{eqnhomKI}.
		\end{proof}
	
	We end the subsection by giving the proof of Proposition~\ref{propcomplyu} that we still owe:
	
		\begin{proof}[\textsl{Proof of Proposition~\ref{propcomplyu}}]\label{hereistheproofforproplyu}
		By Lyubashenko's construction the modular category $\cat{C}$ gives rise to a modular functor, i.e.\ a modular $\Fin$-valued algebra over $\surfc$ that we denote by $\cat{C}^\text{Lyu}$;
		this is essentially a reformulation of \cite{lyubacmp,lyu,lyulex} using a different language. 
		
		Consider now the map $h: \Envint \hbdyg_0 \to \hbdyg\subset \hbdy \to \surfc$ of modular operads. We will now compare the pullback $h^* \cat{C}^\text{Lyu}$ with the modular $\Envint\hbdyg_0$-algebra $\widehat{\cat{C}^\text{a}}$ featuring in Theorem~\ref{thmmainbalancedbraidedcalc1}. Once we prove
		\begin{align}\widehat{\cat{C}^\text{a}}\simeq h^* \cat{C}^\text{Lyu}\label{eqncomparison} \end{align} we obtain by means of the Theorems~\ref{thmmainbalancedbraidedcalc1} and \ref{thmmainbalancedbraidedcalc2} immediately the desired statement if $(g,n)\neq (0,0),(0,1),(1,0)$. In the case $(g,n)=(0,0),(0,1)$, the statement can be easily verified directly. In the case, $(g,n)=(1,0)$, it follows from the fact that by $h^* \cat{C}^\text{Lyu}$ by construction factors through $\hbdy$. Then by \eqref{eqncomparison} the same is true for $\widehat{\cat{C}^\text{a}}$.
		
		Hence, it remains to prove \eqref{eqncomparison}: To this end, we use
		$h^* \cat{C}^\text{Lyu}=\widehat{\left(h^*\cat{C}^\text{Lyu}\right)_0}$ which follows from
		Remark~\ref{remmodext}. Here $\left(h^*\cat{C}^\text{Lyu}\right)_0$ denotes 
		 the restriction of $h^*\cat{C}^\text{Lyu}$ to a cyclic algebra.
		Note that Remark~\ref{remmodext} uses that $h^*\cat{C}^\text{Lyu}$ inverts all morphisms in the categories of operations (which is the case because $h^*\cat{C}^\text{Lyu}$ comes from a modular algebra over $\hbdyg$ which is groupoid-valued).
		 Thanks to  $(h^*\cat{C}^\text{Lyu})_0 = (\cat{C}^\text{Lyu})_0$ (this holds by definition), \eqref{eqncomparison} now reads $\widehat{\cat{C}^\text{a}}\simeq \widehat{\left(\cat{C}^\text{Lyu}\right)_0}$ and hence can be reduced to the equivalence $\cat{C}^\text{a}\simeq \left(\cat{C}^\text{Lyu}\right)_0$ of cyclic $\hbdyg_0$-algebras. This equivalence can be seen as follows:
		Clearly, both assign the same left exact functors to genus zero handlebodies. In fact, they also assign the same isomorphisms of such functors to isomorphisms of genus zero handlebodies (phrased equivalently, they agree on the mapping class groups of genus zero surfaces --- which are ribbon braid groups). Explicitly, the braiding generators act through the braiding of $\cat{C}$, and the Dehn twist around a boundary component acts by the balancing. This description is tautologically true for $\cat{C}$ as a cyclic $\RBr$-algebra (and therefore for $\cat{C}^\text{a}$). In the Lyubashenko construction, i.e.\ for $\cat{C}^\text{Lyu}$ and hence for $\left(\cat{C}^\text{Lyu}\right)_0$, it holds by definition, see  \cite{lyubacmp,lyu,lyulex} or the Lego-Teichmüller version of the construction in \cite[Section~2\&3]{jfcs}.
	\end{proof}

	\end{document}

%% file: rfa-diagrams.tikzstyles
\tikzstyle{black dot}=[fill=black, draw=black, shape=circle, minimum size=3pt, inner sep=0pt]
\tikzstyle{black dot small}=[fill=black, draw=black, shape=circle, minimum size=3pt, inner sep=0pt]
\tikzstyle{big white circle}=[fill=white, draw=black, shape=circle, minimum width=0.75cm]
\tikzstyle{white dot big}=[fill=white, draw=black, shape=circle, inner sep=1pt]
\tikzstyle{white dot}=[fill=white, draw=black, shape=circle, minimum size=3pt, inner sep=0pt]
\tikzstyle{flat box}=[fill=white, draw=black, shape=rectangle, minimum width=2.5cm, minimum height=0.5cm]
\tikzstyle{square}=[fill=white, draw=black, shape=rectangle]
\tikzstyle{flat box 2}=[fill=white, draw=black, shape=rectangle, minimum height=0.5cm, minimum width=1.0cm]
\tikzstyle{over }=[front]
\tikzstyle{theta}=[fill=black, draw=black, shape=ellipse, minimum height=6pt, minimum width=6pt, inner sep=0pt]
\tikzstyle{thetabig}=[fill=black, draw=black, shape=ellipse, minimum width=1cm, minimum height=0.01cm]
\tikzstyle{thetainv}=[fill=white, draw=black, shape=ellipse, minimum height=6pt, minimum width=6pt, inner sep=0pt]
\tikzstyle{thetabinv}=[fill=white, draw=black, shape=ellipse, minimum width=1cm, minimum height=0.01cm]

\tikzstyle{mid arrow}=[-, postaction={on each segment={mid arrow}}]
\tikzstyle{end arrow}=[->]
\tikzstyle{red mid arrow}=[-, draw={rgb,255: red,214; green,42; black,51}, postaction={on each segment={mid arrow}}, line width=1pt]
\tikzstyle{blue}=[-, draw=black, line width=1.5pt]
\tikzstyle{blue mid arrow}=[-, draw={rgb,255: red,23; green,37; black,167}, postaction={on each segment={mid arrow}}, line width=1pt]
\tikzstyle{over}=[-, link]
\tikzstyle{mapsto}=[{|->}]

%% file: Mueller_Woike_main.bbl
\small

\begin{thebibliography}{BDSPV15}
	
	\bibitem[Ati90]{atiyahframing}
	M.~Atiyah.
	\newblock {On Framings of 3-Manifolds}.
	\newblock {\em Topology}, 29(1):1--7, 1990.
	
	\bibitem[Bar79]{barr}
	M.~Barr.
	\newblock {\em $\star$-autonomous categories}, volume 572 of {\em Lecture Notes
		in Math.}
	\newblock Springer, 1979.
	
	\bibitem[BD13]{bd}
	M.~Boyarchenko and V.~Drinfeld.
	\newblock A duality formalism in the spirit of {G}rothendieck and {V}erdier.
	\newblock {\em Quantum Top.}, 4(4):447--489, 2013.
	
	\bibitem[BDSPV15]{BDSPV15}
	B.~Bartlett, C.~L. Douglas, C.~Schommer-Pries, and J.~Vicary.
	\newblock Modular categories as representations of the 3-dimensional bordism
	category.
	\newblock arXiv:1509.06811 [math.AT], 2015.
	
	\bibitem[BJSS21]{bjss}
	A.~Brochier, D.~Jordan, P.~Safronov, and N.~Snyder.
	\newblock {Invertible braided tensor categories}.
	\newblock {\em Alg. Geom. Top.}, 21(4):2107--2140, 2021.
	
	\bibitem[BK00]{bakifm}
	B.~Bakalov and A.~Kirillov.
	\newblock {On the {L}ego-{T}eichmüller game}.
	\newblock {\em Transf. Groups}, 6:207--244, 2000.
	
	\bibitem[BK01]{baki}
	B.~Bakalov and A.~Kirillov.
	\newblock {\em Lectures on tensor categories and modular functors}, volume~21
	of {\em University Lecture Series}.
	\newblock Am. Math. Soc., 2001.
	
	\bibitem[BM07]{bm}
	C.~Berger and I.~Moerdijk.
	\newblock {Resolution of coloured operads and rectification of homotopy
		algebras}.
	\newblock {\em Contemp. Math.}, 431:31--58, 2007.
	
	\bibitem[Bud08]{budney}
	R.~Budney.
	\newblock {The framed discs operad is cyclic}.
	\newblock {\em J. Pure Appl. Alg.}, 212(1):193--196, 2008.
	
	\bibitem[BV68]{bv68}
	J.~M. Boardman and R.~M. Vogt.
	\newblock {Homotopy-everything $H$-spaces}.
	\newblock {\em Bull. Amer. Math. Soc.}, 74:1117--1122, 1968.
	
	\bibitem[BV73]{bv73}
	J.~M. Boardman and R.~M. Vogt.
	\newblock {\em Homotopy invariant algebraic structures on topological spaces},
	volume 347 of {\em Lecture Notes in Math.}
	\newblock Springer, 1973.
	
	\bibitem[BZBJ18]{bzbj}
	D.~Ben-Zvi, A.~Brochier, and D.~Jordan.
	\newblock {Integrating quantum groups over surfaces}.
	\newblock {\em J. Top.}, 11(4):874--917, 2018.
	
	\bibitem[CIW19]{CIW}
	R.~Campos, N.~Idrissi, and T.~Willwacher.
	\newblock Configuration spaces of surfaces.
	\newblock arXiv:1911.12281 [math.QA], 2019.
	
	\bibitem[Cos04]{costello}
	K.~Costello.
	\newblock The {A}-infinity operad and the moduli space of curves.
	\newblock arXiv:math/0402015 [math.AG], 2004.
	
	\bibitem[DSPS19]{dss}
	C.~L. Douglas, C.~Schommer-Pries, and N.~Snyder.
	\newblock {The balanced tensor product of module categories}.
	\newblock {\em Kyoto J. Math.}, 59(1):167--179, 2019.
	
	\bibitem[EGNO15]{egno}
	P.~Etingof, S.~Gelaki, D.~Nikshych, and V.~Ostrik.
	\newblock {\em Tensor categories}, volume 205 of {\em Math. Surveys Monogr.}
	\newblock Am. Math. Soc., 2015.
	
	\bibitem[EML53]{eilenbergmaclane}
	S.~Eilenberg and S.~Mac~Lane.
	\newblock {On the groups $H (\pi, n)$, I}.
	\newblock {\em Ann. Math.}, 58(1):55--106, 1953.
	
	\bibitem[Enr10]{Enriquez}
	B.~Enriquez.
	\newblock Half-balanced braided monoidal categories and {T}eichmueller
	groupoids in genus zero.
	\newblock arXiv:1009.2652 [math.QA], 2010.
	
	\bibitem[EO04]{etingofostrik}
	P.~Etingof and V.~Ostrik.
	\newblock {Finite tensor categories}.
	\newblock {\em Mosc. Math. J.}, 4(3):627--654, 2004.
	
	\bibitem[Fre17]{FresseI}
	B.~Fresse.
	\newblock {\em Homotopy of operads and Grothendieck-Teichm\"uller groups. Part
		1: The Algebraic Theory and its Topological Background}, volume 217 of {\em
		Math. Surveys and Monogr.}
	\newblock Am. Math. Soc., 2017.
	
	\bibitem[FS17]{jfcs}
	J.~Fuchs and C.~Schweigert.
	\newblock {Consistent systems of correlators in non-semisimple conformal field
		theory}.
	\newblock {\em Adv. Math.}, 307:598--639, 2017.
	
	\bibitem[FSS20]{fss}
	J.~Fuchs, G.~Schaumann, and C.~Schweigert.
	\newblock {Eilenberg-Watts calculus for finite categories and a bimodule
		Radford $S^4$ theorem}.
	\newblock {\em Trans. Am. Math. Soc.}, 373:1--40, 2020.
	
	\bibitem[Gia11]{giansiracusa}
	J.~Giansiracusa.
	\newblock {The framed little 2-discs operad and diffeomorphisms of
		handlebodies}.
	\newblock {\em J. Top.}, 4(4):919--941, 2011.
	
	\bibitem[Gia13]{giansiracusam}
	J.~Giansiracusa.
	\newblock {Moduli spaces and modular operads}.
	\newblock {\em Morfismos}, 17(2):101--125, 2013.
	
	\bibitem[GK95]{gk}
	E.~Getzler and M.~Kapranov.
	\newblock Cyclic operads and cyclic homology.
	\newblock In R.~Bott and S.-T. Yau, editors, {\em Conference proceedings and
		lecture notes in geometry and topology}, pages 167--201. Int. Press, 1995.
	
	\bibitem[GK98]{gkmod}
	E.~Getzler and M.~Kapranov.
	\newblock {Modular operads}.
	\newblock {\em Compositio Math.}, 110:65--126, 1998.
	
	\bibitem[GM13]{gilmermasbaum}
	P.~M. Gilmer and G.~Masbaum.
	\newblock {{Maslov index, Lagrangians, Mapping Class Groups and TQFT}}.
	\newblock {\em Forum Math.}, 25(5):1067--1106, 2013.
	
	\bibitem[HGN17]{hgn}
	R.~Haugseng, D.~Gepner, and T.~Nikolaus.
	\newblock {Lax colimits and free fibrations in $\infty$-categories}.
	\newblock {\em Doc. Math.}, 22:1225--1266, 2017.
	
	\bibitem[HH12]{hahe}
	U.~Hamenstädt and S.~Hensel.
	\newblock {The geometry of the handlebody groups I: {D}istortion}.
	\newblock {\em J. Top. Anal.}, 4(1):71--97, 2012.
	
	\bibitem[Hor17]{horel}
	G.~Horel.
	\newblock {Factorization homology and calculus à la Kontsevich Soibelman}.
	\newblock {\em J. Noncommutative Geom.}, 11(2):703--740, 2017.
	
	\bibitem[Hua08]{huang}
	Y.-Z. Huang.
	\newblock {Rigidity and modularity of vertex tensor categories}.
	\newblock {\em Commun. Contemp. Math.}, 10(1):871--911, 2008.
	
	\bibitem[JS91]{joyalstreet}
	A.~Joyal and R.~Street.
	\newblock {The geometry of tensor calculus, I}.
	\newblock {\em Adv. Math.}, 88(1):55--112, 1991.
	
	\bibitem[JY20]{johnsonyau}
	N.~Johnson and D.~Yau.
	\newblock 2-dimensional categories.
	\newblock arXiv:2002.06055 [math.CT], 2020.
	
	\bibitem[Kas15]{kassel}
	C.~Kassel.
	\newblock {\em Quantum Groups}, volume 155 of {\em Graduate Texts in Math.}
	\newblock Springer, 2015.
	
	\bibitem[KL01]{kl}
	T.~Kerler and V.~V. Lyubashenko.
	\newblock {\em Non-Semisimple Topological Quantum Field Theories for
		3-Manifolds with Corners}, volume 1765 of {\em Lecture Notes in Math.}
	\newblock Springer, 2001.
	
	\bibitem[Lei98]{leinster}
	T.~Leinster.
	\newblock Basic bicategories.
	\newblock arXiv:math/9810017 [math.CT], 1998.
	
	\bibitem[Lur17]{lurieHA}
	J.~Lurie.
	\newblock Higher algebra.
	\newblock 2017.
	
	\bibitem[LV12]{lodayvallette}
	J.-L. Loday and B.~Vallette.
	\newblock {\em Algebraic Operads}, volume 346 of {\em Grundlehren der math.
		Wiss.}
	\newblock Springer, 2012.
	
	\bibitem[Lyu95a]{lyubacmp}
	V.~V. Lyubashenko.
	\newblock {Invariants of 3-manifolds and projective representations of mapping
		class groups via quantum groups at roots of unity}.
	\newblock {\em Comm. Math. Phys.}, 172:467--516, 1995.
	
	\bibitem[Lyu95b]{lyu}
	V.~V. Lyubashenko.
	\newblock {Modular transformations for tensor categories}.
	\newblock {\em J. Pure Appl. Alg.}, 98:279--327, 1995.
	
	\bibitem[Lyu96]{lyulex}
	V.~V. Lyubashenko.
	\newblock {Ribbon abelian categories as modular categories}.
	\newblock {\em J. Knot Theory and its Ramif.}, 5:311--403, 1996.
	
	\bibitem[Mü03]{mueger}
	M.~Müger.
	\newblock {On the structure of modular categories}.
	\newblock {\em Proc. London Math. Soc.}, 3(87):291--308, 2003.
	
	\bibitem[ML98]{maclane}
	S.~Mac~Lane.
	\newblock {\em Categories for the Working Mathematician}, volume~5 of {\em
		Graduate Texts Math.}
	\newblock Springer, 1998.
	
	\bibitem[MLM92]{maclanemoerdijk}
	S.~Mac~Lane and I.~Moerdijk.
	\newblock {\em Sheaves in Geometry and Logic}.
	\newblock Springer Universitext. Springer, 1992.
	
	\bibitem[MSS02]{markl}
	M.~Markl, S.~Shnider, and J.~D. Stasheff.
	\newblock {\em Operads in Algebra, Topology and Physics}, volume~96 of {\em
		Math. Surveys and Monogr.}
	\newblock Am. Math. Soc., 2002.
	
	\bibitem[MW20]{littlebundles}
	L.~Müller and L.~Woike.
	\newblock {The Little Bundles Operad}.
	\newblock {\em Alg. Geom. Top.}, 20:2029--2070, 2020.
	
	\bibitem[Pst14]{Pstragowski}
	P.~Pstragowski.
	\newblock On dualizable objects in monoidal bicategories, framed surfaces and
	the cobordism hypothesis.
	\newblock arXiv:1411.6691 [math.AT], 2014.
	
	\bibitem[RT90]{rt1}
	N.~Reshetikhin and V.~G. Turaev.
	\newblock {Ribbon graphs and their invariants derived from quantum groups}.
	\newblock {\em Comm. Math. Phys.}, 127:1--26, 1990.
	
	\bibitem[Shi19]{shimizumodular}
	K.~Shimizu.
	\newblock {Non-degeneracy conditions for braided finite tensor categories}.
	\newblock {\em Adv. Math.}, 355:106778, 2019.
	
	\bibitem[SP09]{schommerpries}
	C.~J. Schommer-Pries.
	\newblock {\em The classification of two-dimensional extended topological field
		theories}.
	\newblock PhD thesis, Berkeley, 2009.
	
	\bibitem[Str04]{street}
	R.~Street.
	\newblock {Frobenius monads and pseudomonoids}.
	\newblock {\em J. Math. Phys.}, 45(10):3930--3948, 2004.
	
	\bibitem[SW03]{salvatorewahl}
	P.~Salvatore and N.~Wahl.
	\newblock {Framed discs operads and {B}atalin-{V}ilkovisky algebras}.
	\newblock {\em Quart. J. Math.}, 54:213--231, 2003.
	
	\bibitem[SW21a]{dmf}
	C.~Schweigert and L.~Woike.
	\newblock {Homotopy Coherent Mapping Class Group Actions and Excision for
		Hochschild Complexes of Modular Categories}.
	\newblock {\em Adv. Math.}, 386:107814, 2021.
	
	\bibitem[SW21b]{dva}
	C.~Schweigert and L.~Woike.
	\newblock {The Hochschild Complex of a Finite Tensor Category}.
	\newblock {\em Alg. Geom. Top.}, 21(7):3689--3734, 2021.
	
	\bibitem[Tho79]{thomason}
	R.~W. Thomason.
	\newblock {Homotopy colimits in the category of small categories}.
	\newblock {\em Math. Proc. Cambridge Philos. Soc.}, 85(1):91--109, 1979.
	
	\bibitem[Til98]{tillmann}
	U.~Tillmann.
	\newblock {$\mathcal{S}$-Structures for $k$-Linear Categories and the
		Definition of a Modular Functor}.
	\newblock {\em J. London Math. Soc.}, 58(1):208--228, 1998.
	
	\bibitem[Tur94]{turaev}
	V.~G. Turaev.
	\newblock {\em Quantum Invariants of Knots and 3-Manifolds}, volume~18 of {\em
		Studies in Math.}
	\newblock De Gruyter, 1994.
	
	\bibitem[Wah01]{WahlThesis}
	N.~Wahl.
	\newblock {\em Ribbon braids and related operads}.
	\newblock PhD thesis, Oxford, 2001.
	
	\bibitem[Wee19]{timw}
	T.A.N. Weelinck.
	\newblock A topological origin of quantum symmetric pairs.
	\newblock {\em Sel. Math.}, 25:1--47, 2019.
	
	\bibitem[Woi20]{woike}
	L.~Woike.
	\newblock {\em Higher Categorical and Operadic Concepts for Orbifold
		Constructions --- A Study at the Interface of Topology and Representation
		Theory}.
	\newblock PhD thesis, Hamburg, 2020.
	
	\bibitem[Zet18]{Zetsche}
	S.~J. Zetsche.
	\newblock Generalised duality theory for monoidal categories and applications.
	\newblock Master thesis, {Hamburg}, 2018.
	
\end{thebibliography}
